\documentclass[a4paper,reqno]{amsart}

\usepackage[utf8]{inputenc}
\usepackage{amssymb}
\usepackage{mathrsfs}
\usepackage{mathtools}
\usepackage{enumerate}
\usepackage{amsmath}

\usepackage{xcolor}
\usepackage{hyperref}
\usepackage{enumitem}

\newtheorem{theorem}{Theorem}[section]

\newtheorem{lemma}[theorem]{Lemma}
\newtheorem{proposition}[theorem]{Proposition}

\theoremstyle{definition}

\newtheorem{remark}[theorem]{Remark}

\numberwithin{equation}{section}
\newcommand\RR{\mathbb{R}}
\newcommand\NN{\mathbb{N}}

\allowdisplaybreaks[4]
\begin{document}
	\parindent=0pt
	\title{On melting for the 3D radial Stefan problem}
 \author[C.Zhang]{Chencheng Zhang}
\address{School of Mathematical Sciences,
University of Science and Technology of China, Hefei 230026, Anhui, China}
\email{zccmaths@mail.ustc.edu.cn}
		 \begin{abstract}
     We consider the three-dimensional radial  Stefan problem which describes the evolution of a radial symmetric ice ball with free boundary
     \begin{equation*}
	\left\{\begin{aligned}
		&\partial_{t}u-\partial_{rr}u-\frac{2}{r}\partial_{r}u=0 \quad in\ r\geq\lambda(t),\\
		&\partial_{r}u(t,\lambda(t))=-\dot{\lambda}(t),\\
		&u(t,\lambda(t))=0,\\
		&u(0,\cdot)=u_{0},\quad \lambda(0)=\lambda_{0}.
	\end{aligned}\right.
\end{equation*}
We prove the existence in the radial class of finite time melting with rates
\begin{equation*}
	\lambda(t)=\left\{\begin{aligned}
		&4\sqrt{\pi}\frac{\sqrt{T-t}}{|\log (T-t)|}(1+o_{t\rightarrow T}(1)),\\
		&c(u_{0},k)(1+o_{t\rightarrow T}(1))(T-t)^{\frac{k+1}{2}},\quad k\in\NN^{*},
	\end{aligned}\right.
\end{equation*}
which respectively correspond to the fundamental stable melting rate and a sequence of codimension $k$ unstable rates. Our analysis mainly depend on the methods developed in \cite{MP} which deals with the similar problems in two dimensions and also the construction of both stable and unstable finite time blow-up solutions for the harmonic heat flow in \cite{RS}, \cite{RS1}.
 \end{abstract}
	\maketitle

	\section{Introduction}
	 Stefan problem aims to describe the evolution of the boundary between two phases of a material undergoing a phase change, for example the melting of a solid, such as ice to water. There has been a long history since the problems were posed. Readers can refer to the monographs \cite{M, Ru} for furhter physical backgrounds and histories of the early mathematical studies.  In this paper we consider the classical three dimensional one-phase Stefan problem on an external domain which describes the melting of a ball of ice surrounded by water. The exterior region is denoted by $ \Omega(t) \subset \RR^{3}$ and $ u(t,x) $ denotes the temperture of the water, which is assumed to be zero at the ice phase $ \partial\Omega(t) $. The temperature of the ice is supposed to be maintained at zero.  Thus the temperature function $ u:\Omega(t)\rightarrow \RR $ evolves as
	\begin{equation}\label{eq:1}
		\left\{\begin{aligned}
      &\partial_{t}u-\Delta{u}=0\quad in \ \Omega(t),\\
      &\frac{\partial{u}}{\partial\vec{n}}=V_{\partial\Omega(t)} \quad on \ \partial\Omega(t),\\
      &u=0 \quad on \ \partial\Omega(t),
		\end{aligned}\right.
	\end{equation}
where $ V_{\partial\Omega(t)} $ stands for the normal velocity of the moving boundary $ \partial\Omega(t) $. 

%{\it There is a large amount of literature on the Cauchy theory for the  classical one-phase Stefan problem. Firstly, weak solutions were  defined by Kamenomostskaja \cite{K}, Friedman \cite{F} and Lady\v{z}enskaja, Solonnikov and Ural\'ceva \cite{LSU} and then the existence and the regularity of the free boundary were further studied (see \cite{CL1}, \cite{CL2}, \cite{CL3},\cite{CL4}, \cite{FK}, \cite{KN}, \cite{KN1} and references therein). Since the stefan problem satisfies the maximum principle, its analysis is ideally suited to another type of weak solution called the viscosity solution. Existence of  viscosity solutions  for the one phase problem was established by Kim \cite{Kim}. For an exhaustive overview and introduction to the regularity theory of viscosity solutions we refer the readers to Caffarelli and Salsa \cite{CS}. See also \cite{ACS1}, \cite{ACS2}, \cite{CK}, \cite{KH} and references therein for further discussions of the viscosity solutions. As to the classical solutions, the first local
%existence results of classical solutions for the classical Stefan problem were established by Meirmanov \cite{M} and Hanzawa \cite{H}, but their results all with derivative loss. Then Had\v{z}i\'{c} and Shkoller established the local-in-time existence, uniqueness and regularity for the classical Stefan problem in $L^{2}$-based Sobolev space without derivative loss and they also studied the global stablity of steady states (see \cite{HS1}, \cite{HS2}). See also \cite{SF}, \cite{PSS} for the well-posedness in $ L^{p} $-type spaces.}

In the extensive body of literature addressing the Cauchy theory for the classical one-phase Stefan problem, the groundwork was laid by defining weak solutions, credited to Kamenomostskaja \cite{K}, Friedman \cite{F}, and Lady\v{z}enskaja, Solonnikov, and Uralceva \cite{LSU}. Subsequent research delved into establishing the existence and regularity of the free boundary, with notable contributions found in \cite{CL1}, \cite{CL2}, \cite{CL3}, \cite{CL4}, \cite{FK}, \cite{KN}, \cite{KN1}, and related references. Given the Stefan problem's adherence to the maximum principle, an alternative avenue of analysis emerges, focusing on viscosity solutions. Kim \cite{Kim} established the existence of viscosity solutions for the one-phase problem. For a comprehensive understanding and an introduction to the regularity theory of viscosity solutions, readers are directed to Caffarelli and Salsa \cite{CS}, along with additional insights in \cite{ACS1}, \cite{ACS2}, \cite{CK}, \cite{KH}, and related discussions. Concerning classical solutions, Meirmanov \cite{M} and Hanzawa \cite{H} were pioneers in establishing local existence results for the classical Stefan problem. However, their outcomes suffered from derivative loss. Had\v{z}i{c} and Shkoller later achieved local-in-time existence, uniqueness, and regularity for the classical Stefan problem within $L^{2}$-based Sobolev space, without the aforementioned derivative loss. Their work also extended to studying the global stability of steady states, as detailed in \cite{HS1} and \cite{HS2}. For additional perspectives on well-posedness in $L^{p}$-type spaces, refer to \cite{SF}, \cite{PSS}.

From now on, we restrict our discussion to the radially symmetric case since in this case the Cauchy theory is simpler. The domain $ \Omega(t) $ is then given by 
\begin{equation*}
	\Omega(t)=\{x\in \RR^{3}|\ |x|\geq \lambda(t)\},
\end{equation*}
and equation (\ref{eq:1}) becomes
\begin{equation}\label{eq:2}
	\left\{\begin{aligned}
		&\partial_{t}u-\partial_{rr}u-\frac{2}{r}\partial_{r}u=0 \quad in\ \Omega(t),\\
		&\partial_{r}u(t,\lambda(t))=-\dot{\lambda}(t),\\
		&u(t,\lambda(t))=0,\\
		&u(0,\cdot)=u_{0},\quad \lambda(0)=\lambda_{0}.
	\end{aligned}\right.
\end{equation}
As was established in \cite{MP} for 2D case and Appendix B for 3D case, the Cauchy problem \eqref{eq:2} is wellposed in $H^{2}$: for any $  (u_{0},\lambda_{0})\in\ H^{2}\times \RR_{+}^{*} $ with $ u_{0} $ radially symmetric, there exists a unique solution $ (u(t),\lambda(t))\in C([0,T);H^{2}(\Omega(t)))\times C^{1}([0,T),\RR_{+}^{*}) $ to (\ref{eq:2}), and
\begin{equation*}
	T<+\infty\quad \Leftrightarrow\quad \lim_{t\rightarrow T}\lambda(t)=0\ {\rm{or}}\ \lim_{t\rightarrow T}\lVert u(t)\rVert_{H^{2}(\Omega(t))}=+\infty.
\end{equation*}
Moreover, the following basic properties hold:
\begin{itemize}
    \item Scaling invariance: If $(u,\lambda)$ solves (\ref{eq:2}), then so does
    \begin{equation*}
    u_{\mu}(t,r):=u(\mu^{2}t,\mu r),\quad \lambda_{\mu}(t):=\frac{\lambda(t)}{\mu}.
    \end{equation*}
    \item Energy identities:
    \begin{equation*}
    \frac{1}{2}\frac{d}{dt}\lVert u\rVert^{2}_{L^{2}(|x|\geq\lambda(t))}+\lVert\nabla u\rVert^{2}_{L^{2}(|x|\geq\lambda(t))}=0,
    \end{equation*}
    \begin{equation*}
        \frac{1}{2}\frac{d}{dt}\lVert \nabla u\rVert^{2}_{L^{2}(|x|\geq \lambda(t))}+\lVert\Delta u\rVert^{2}_{(L^{2}(|x|\geq\lambda(t)))}=2\pi(\lambda(t))^{2}(\dot{\lambda}(t))^{3},
    \end{equation*}
\end{itemize}
In particular, in the melting regime $\dot{\lambda}<0$, from the above energy identities, we can also arrive at the dissipation of $\lVert u\rVert_{L^{2}(|x|\geq \lambda(t))}$ and $\lVert \nabla u\rVert_{L^{2}(|x|\geq \lambda(t))}$. \\

As for the long time behavior for the solutions to the Stefan problem \eqref{eq:2}, the melting is one of the typical senarios. The first description of melting was given in the pioneering work by Herrero and Vel\'{a}zquez \cite{HV} which predicts the existence of a discrete sequence of melting rates in 3D:
\begin{equation*}
	\lambda(t)\sim\left\{\begin{aligned}
		&\frac{\sqrt{T-t}}{|\log (T-t)|},\\
		&(T-t)^{\frac{k+1}{2}},\quad k\in\NN^{*}.
	\end{aligned}\right.
\end{equation*}
Their methods mainly based on the delicate matching asymptotics procedure for the construction of non-selfsimilar type II blow up bubbles. However, as mentioned by the authous themselves (see \cite{HHV}, \cite{V}), this approach cannot address the stability of such melting and can only be applied in the radial case. Later, Rapha\"{e}l and Had\v{z}i\'{c} \cite{MP} avoided the formal matched asymptotics approach. They replaced it by an elementary derivation of an explicit and universal system of ODE which arrives at the melting rates and then they used a robust energy method to control the tail terms. As a result, they got the similar results in 2D and also studied their stability in the radial regime.\\
For the past twenty years, there have been many results in the field of constructing type II blow-up solutions, in both parabolic and dispersive problems. In the dispersive fields, the first such result was obtained by Bourgain and Wang \cite{BW}  for the mass critical Schr\"{o}dinger equations, and then for the same equation, Perelman \cite{P} constructed the  log-log blow-up solutions in one dimension for an even initial data. Later, Merle and Rapha\"{e}l (see \cite{MR1}, \cite{MR2}, \cite{MR3}, \cite{MR4}, \cite{MR5}) extended  Perelman's result in higher dimensions and the non-radial case using the modulation methods, energy monotonicity and local virial identities . Moreover, they also classified the blow-up solutions near the ground state and studied the stablity of such blow-up solutions (see \cite{R}).
%Then in a similar spirits, Rodnianski and Sterbenz \cite{IJ}, Krieger, Schlag and Tataru \cite{KST}, \\
For some other construction of type II blow-up solutions in the dispersive field, we refer the readers to see \cite{HR,IJ,J,KST2,MRP,FPIJ,MRR1,KST,PR} and references therein. To summarize, the approaches developed by Merle, Rapha\"{e}l, Rodnianski and the co-authors in various works (see for instance \cite{MR3,MR1,MRR1,MRP,PR,IJ})  mainly split the energy concentration problem into a finite-dimensional part which contains the leading order dynamics, and an infinite-dimensional part which is controlled using the energy methods. This method can also been applied to  parabolic equations and will also play an important role in our argument. First such result in parabolic field was obtained by Schweyer \cite{RS3} for the 4D radial energy critical heat equation. Then for the energy supercritical heat equation, Collot \cite{C1} constructed a non-radial type II blow-up solution in large dimensions using robust nonlinear analysis based on the modulation methods and energy methods, see also  \cite{CFP} for another type II blow-up solutions in this case.
%See also \cite{HR,J,KST2,MRP,FPIJ,MRR1,IJ,KST,PR,KK} for the energy critical wave equations
%, \cite{MRR1} for the energy critical Schr\"{o}dinger map equation\\
%and \cite{MRP,FPIJ} for the energy super critical Schr\"{o}dinger equations.\\
For the corotational energy-critical harmonic heat flow, Rapha\"{e}l and Schweyer \cite{RS}, \cite{RS1} constructed both stable and codimension $k$ stable type II blow-up solutions following mainly the spirits developed in \cite{MRR1} and \cite{PR}.   For the Keller-Segel model which describes the chemotactic aggregation, Rapha\"{e}l and Schweyer \cite{RS2} also constructed a stable type II blow-up solution in the radial case and then Collot, Ghoul, Masmoudi and Nguyen \cite{CTNV} constructed both stable and codimension $ k $ stable blow up solutions in the non-radial case. Their methods mainly based on a change of variables related to some partial mass and a detailed spectral analysis for the linearised dynamics which improves the earlier results and can also be extended to the non-radial case.

In this paper, we are devoted to the description of finite time melting for the classical Stefan problem in the radial case in 3D. We prove the existence of stable and unstable type II finite time melting solutions and then study the singularity formation process. 

\begin{theorem}[Melting dynamics]\label{th1}
	There exists a set of initial data $ (u_{0},\lambda_{0}) $ in $ H^{2}\times \RR^{*}_{+} $, such that the corresponding solution $(u,\lambda)\in C(\left[0,T\right);H^{2} )\times C^{1}([0,T);\RR_{+})$ to the exterior Stefan problem (\ref{eq:2}) melts in finite time $ 0<T(u_{0},\lambda_{0})<+\infty $ with the following dynamics:
	\begin{itemize}
		\item Stable regime: the fundamental melting rate is given by
		\begin{equation}\label{th1.1}
			\lambda(t)=4\sqrt{\pi}\frac{\sqrt{T-t}}{|\log (T-t)|}(1+o_{t\rightarrow T}(1)),
		\end{equation}
	and is stable by small radial perturbations in $ H^{2} $.
	\item Excited regimes: the excited melting rates are given by
	\begin{equation}\label{th1.2}
		\lambda(t)=c(u_{0},k)(1+o_{t\rightarrow T}(1))(T-t)^{\frac{k+1}{2}},\quad k\in\NN^{*}
	\end{equation}
for some fixed constant $ c(u_{0},k)>0 $ and it corresponds to the initial data lying on a locally Lipschitz $ H^{2} $ manifold of codimension $ k $.
\item Non-concentration of energy: In all cases, there exists $ u^{*}\in\dot{H}^{1} $ such that
\begin{equation}\label{th1.3}
	\lim_{t\rightarrow T}\lVert \nabla u\chi_{\{|x|\geq\lambda(t)\}}-\nabla u^{*}\rVert_{{L^{2}}(\RR^{3})}=0.
\end{equation}
	\end{itemize}
\end{theorem}
\begin{remark}   
We here consider the  case  $d=3$. In fact all $d\geq 3$ are similar since in this case they have the fundamental solution of the  form $\frac{1}{|x|^{d-2}}$ which plays an essential role in the construction of the approximate solutions as we will see below. At the same time, compared with the similar results in two dimension \cite{MP}, since $\frac{1}{|x|^{d-2}}$ generates more singularity than $\log|x|$ at the origin, the corresponding approximate solutions in $d\geq 3$ decay slower  which  requires our computation to be more precise and it's also the reason why the melting rates in $d\geq 3$ is slower than the case $d=2$.
\end{remark}
\begin{remark} For the stable regime, the constant $4 \sqrt{\pi}$ is sharp and depends only on the dimension for $d\geq 3$ as predicted by Herrero and Vel\'{a}zquez \cite{HV}. In dimension two, as computed in \cite{MP}, the authors obtained the stable melting rate $\lambda(t)=(T-t)^{\frac{1}{2}}e^{-\frac{\sqrt{2}}{2}\sqrt{|\log (T-t)|}+O(1)}$, but the constant $e^{O(1)}$ is not explicit. While in the codimension $k$ stable case, from our computation\footnote{see for example (\ref{3.6:7}) below.}, the melting rates will depend on the initial data and $k$.
\end{remark}
\begin{remark} We restrict our initial data in the radial regime, however, we don't know whether the stability holds for the non-radial initial perturbation. This may be the first main open problem following this work. Formal asymptotics for finite time melting is presented in \cite{AHV} in addition to a wealth of other possible singularity formation scenarios. There the problem is formulated by the Baiocchi transform and the authors identify ellipsoids in three dimensions as the asymptotic attractors for the melting dynamics in suitably rescaled variables.
\end{remark}
 \begin{remark} As in \cite{MP}, our methods can also be applied  to the construction of freezing solutions (i.e. the case $\dot{\lambda}>0$) that asymptotically converge to a steady state ($u=0,\lambda=\lambda_{\infty}>0$) with freezing rates both stable and codimension $k$ stable. The proof bases on a deep underlying duality between the derivation of the melting and the freezing rates which is analogous to the computation for the melting regime. So we omit here.
\end{remark}
Our approach mainly consists of two steps: construction of a high order approximate solution based on the expansion of the eigenfunctions with respect to a well chosen small parameter, and then using the energy method to control the remaining infinite dimensional part. This method has already been used to deal with the similar questions in two dimensions in \cite{MP} and our argument heavily relies on the technique developed by them.\\
We now briefly describe the main ideas behind the proof of Theorem \ref{th1}. Since the equation (\ref{eq:2}) is a free boundary problem, so we renormalise the equation by setting:
\begin{equation*}
    u(t,r)=v(s,y),\quad \frac{ds}{dt}=\frac{1}{\lambda^{2}(t)},\quad y=\frac{x}{\lambda(t)}
\end{equation*}
and obtain the renormalised equation with a fixed boundary:
\begin{equation}\label{in1}
    \left\{\begin{aligned}
        &\partial_{s}v-\frac{\lambda_{s}}{\lambda}\Lambda v-\Delta v=0,\quad |y|>1,\\
        &v(s,1)=0,\\
        &(\partial_{y}v)(s,1)=-\frac{\lambda_{s}}{\lambda}.
    \end{aligned}\right.
\end{equation}
Then it's equivalent to study the equation (\ref{in1}) in the following argument. The proof is divided into three parts.\\
(i). $Spectral\ analysis.$ We firstly analysis the spectra of the linear operator, the proof mainly based on the properties of the standard harmonic oscillator and the Lyapunov-Schmidt argument which has already been used in \cite{MP} to deal with the similar problem in two dimensions. Compared with the case in two dimensions, the remainder terms in our cases have slower decay mainly due to the fact that the fundamental solutions of Laplace's equation in three dimensions has more singularity at the origin than that in two dimensions (see for instance (\ref{in2}) which is small in two dimensions).\\
(ii). $Stable\ regime.$ In this case, following by the the formal computations and the inspirations from \cite{MP}, \cite{MRR1}, \cite{PR}, \cite{RS1}, we firstly construct the approximate of the form (\ref{3.5}). Then we derive the modulation equation which gives the melting rate (\ref{th1.1}). This follows mainly from the orthogonality condition and the boundary conditions (see Lemma \ref{le4.1}).  It remains to control the remainder terms of the solution, this can be done by the weighted energy estimates. Two points are crucial in the derivation of the energy bounds. The first is the boundary equations (see Lemma \ref{le4.1}) and the modulation equations which gives the estimate of the boundary terms. The second is the coercivity estimate of the linear operator $\mathcal{H}_{b}$ (see Lemma \ref{leCH{b}}) which gives the sharp constant $2$ on the left side of (\ref{3.7:0}).\\
%Compared with the case in two dimensions, here due to the remainder terms of the eigenfunctions computed in the first part have slower decay, we can only get that $|a-b|\lesssim b^{\frac{3}{2}}$ (see (\ref{3.27})) and unable to compute the second order expansion of $a$ as in two dimensions or improve the constant $\frac{3}{2}$. It takes some difficulties since in our case we cannot compare the term \\
(iii). $Exicited\ regimes.$ In this case, due to the new degrees of freedom, we need to introduce $k+1$ parameters, one corresponds to the stable direction,  another $k$ parameters correspond to $k$ unstable directions and the set of the initial $k$ parameters in fact form a codimension $k$ manifold. Then we consider the approximate solution of the form (\ref{in3}), similar computations as before, we can derive the leading order equations and obtain the main terms of the parameters (see (\ref{3.9})). It remains to control the tail terms of the parameters. It can be divided into two steps. Firstly projecting onto the modulation equations and then by the orthogonality condition, we obtain the relations between the parameters (see (\ref{3.43})). Secondly, for the stable direction, the bound follows by direct integration of the estimate (\ref{3.5:9}) and for the $k$ unstable directions, the bounds can infer from a standard Brouwer type arguments which has already been used in many different models (see for instance \cite{C}, \cite{MP}, \cite{MRP}, \cite{RS}).\\
\\
\textbf{Notations.}
		Let us denote the ball of radius $ R $ by
		\begin{equation*}
			B_{d}(R):=\{x\in \RR^{d}|\ |x|\leq{R}\}.
		\end{equation*}
	and given any $ \alpha\geq 0 $, we denote the external domain by
	\begin{equation*}
		\Omega(\alpha) :=\{x\in{\RR^{3}}|\ |x|\geq{\alpha}\}.
	\end{equation*}
We introduce the differential operator 
\begin{equation*}
	\Lambda{f}=y\partial_{y}f.
\end{equation*}
Let $ \chi $ be a positive nondecreasing smooth cut-off function with
	\begin{equation*}
		\chi(x)=
	\left\{\begin{aligned}
&0 \qquad {\rm{for}}\ x\leq 1,\\
&1 \qquad {\rm{for}}\ x\geq 2.
	\end{aligned}\right.
\end{equation*}
Given a parameter $ B>0 $, we denote
\begin{equation*}
	\chi_{B}(x)=\chi\Big(\frac{x}{B}\Big)
\end{equation*}
For any vector $ \vec{x}=(x_{i})_{1\leq{i}\leq{n}}, \vec{y}=(y_{i})_{1\leq{i}\leq{n}} $, we define the inner product
\begin{equation*}
	\langle \vec{x},\vec{y} \rangle=\sum_{k=1}^{n}x_{k}y_{k}
\end{equation*}
with the norm
\begin{equation*}
	|\vec{x}|^{2}=\sum_{k=1}^{n}|x_{k}|^{2},
\end{equation*}
which is equivalent to the norm
\begin{equation*}
	|||\vec{x}|||=max\{|x_{1}|,|x_{2}|,...,|x_{n}|\}.
\end{equation*}
We introduce the weight
\begin{equation*}
	\rho(z):= e^{-|z|^{2}/2}, \quad \rho_{b}(y):=e^{-b|y|^{2}/2}, \quad w(x)=x^{\mu}e^{-x}(\mu>-1).
\end{equation*}
In the following argument, for a radial function $f$ :$\RR^{3}\rightarrow \RR$ and $z\geq 0$ we denote $f(z)$ the value of $f(x)$ for $|x|=z$.\\
Then we define the inner product on $ \Omega(\sqrt{b}) $ by
\begin{equation*}
	\langle{f},{g}\rangle_{b}=\int_{\sqrt{b}}^{+\infty}fg\rho z^{2}dz,
\end{equation*}
with the associated norms
\begin{equation*}
\begin{aligned}
\Arrowvert{u}\Arrowvert_{L^{2}_{\rho,b}}=\Big(\int_{\sqrt{b}}^{+\infty}&u^{2}\rho z^{2}dz\Big)^\frac{1}{2},\quad \Arrowvert{u}\Arrowvert^{2}_{H^{1}_{\rho,b}}=\Arrowvert{\partial_{z} u}\Arrowvert^{2}_{L^{2}_{\rho,b}}+\Arrowvert{u}\Arrowvert^{2}_{L^{2}_{\rho,b}},\\ &\Arrowvert{u}\Arrowvert^{2}_{H^{2}_{\rho,b}}=\Arrowvert\Delta{u}\Arrowvert^{2}_{L^{2}_{\rho,b}}+\Arrowvert{u}\Arrowvert^{2}_{H^{1}_{\rho,b}}.
\end{aligned}
\end{equation*}

We then define for $ b>0 $ the Hilbert space
\begin{equation}\label{eq:3}
    L^{2}_{\rho,b}=\{u:\Omega(\sqrt{b})\rightarrow\RR\ |\ u\ {\rm{radial}}\ {\rm{with}}\ \Arrowvert{u}\Arrowvert_{L^{2}_{\rho,b}}<+\infty\ {\rm{and}}\  u(\sqrt{b})=0\},
\end{equation}
\begin{equation*}
	H^{1}_{\rho,b}=\{u:\Omega(\sqrt{b})\rightarrow\RR\ |\ u\ {\rm{radial}}\ {\rm{with}}\ \Arrowvert{u}\Arrowvert_{H^{1}_{\rho,b}}<+\infty\ {\rm{and}}\  u(\sqrt{b})=0\},
\end{equation*}
and for $ b=0 $,
\begin{equation*}
	H^{1}_{\rho,0}=\{u:\RR^{3}\rightarrow\RR\ |\ u\ {\rm{radial}}\ {\rm{with}}\ \Arrowvert{u}\Arrowvert_{H^{1}_{\rho,0}}<+\infty\}.
\end{equation*}
Next, we introduce the renormalized inner product
\begin{equation*}
	(f,g)_{b}=\int_{1}^{+\infty}{fg\rho_{b}y^{2}dy},
\end{equation*}
and the associated norms are
\begin{equation*}
	\Arrowvert{v}\Arrowvert_{b}=\Big(\int_{1}^{+\infty}{v^{2}\rho_{b}y^2dy}\Big)^{\frac{1}{2}}, \quad \Arrowvert{v}\Arrowvert^{2}_{H^{1}_{b}}=\Arrowvert{v}\Arrowvert^{2}_{b}+\Arrowvert\partial_{y}{v}\Arrowvert^{2}_{b},\quad \Arrowvert{v}\Arrowvert^{2}_{H^{2}_{b}}=\Arrowvert{v}\Arrowvert^{2}_{H^{1}_{b}}+\Arrowvert\Delta{v}\Arrowvert^{2}_{b}.
\end{equation*}
Then we  define  the Hilbert space 
\begin{equation}\label{eq:4}
 	H^{1}_{b}=\{v:\Omega(1)\rightarrow\RR\ |\ v\ {\rm{radial}}\ {\rm{with}}\ \Arrowvert{v}\Arrowvert_{H^{1}_{b}}<+\infty\ {\rm{and}}\ v(1)=0\}.
\end{equation}
Similarly, we  define the inner product
\begin{equation*}
	(f,g)_{\omega}=\int_{0}^{+\infty}fg\omega dx,
	\end{equation*}
with the associated norm
\begin{equation*}
	\Arrowvert f\Arrowvert_{\omega}^{2}=\int_{0}^{+\infty}f^{2}(x)\omega(x)dx,
\end{equation*}
and the Hilbert space
\begin{equation*}
	L_{\omega}^{2}=\{f:(0,+\infty)\rightarrow\RR\ |\ \Arrowvert{f}\Arrowvert_{\omega}<+\infty \}.
\end{equation*}
Finally, for simplity, we denote some constants which will always appear in the following sections
\begin{equation}
	B_{k}=\frac{\Gamma(\frac{1}{2}+k+1)}{k!\Gamma(\frac{1}{2}+1)},\quad A_{k}=2^{\frac{1}{4}}\sqrt{\frac{\Gamma(\frac{1}{2}+k+1)}{k!}}, \quad C_{k}=\frac{B_{k}}{A_{k}}=\frac{\sqrt{\Gamma(\frac{1}{2}+k+1)}}{2^{\frac{1}{4}}\Gamma(\frac{1}{2}+1)\sqrt{k!}}.
\end{equation}
The remainder of this paper is structured as follows. Section 2 is dedicated to compute the eigenvalues and eigenfunctions of $\mathcal{H}_{b}$ and the associated properties. Section 3 contains the geometrical decomposition and the bootstrap  assumptions on the corresponding solutions. Then, in section 4, we study the equation for the remainder term and derive the modulation equations for the parameters in the geometrical decomposition. The crucial energy estimates are then performed in Section 5. Eventually, in section 6, we close the bootstrap bounds using modulation equations, energy bounds and a topology argument, then we finish the proof of the main theorem.
	\section{Spectral analysis in weighted spaces}
In this section we compute the first $ k $ eigenvalues of the linear operator
\begin{equation*}
	\mathcal{H}_{b}=-\Delta+b\Lambda \quad {\rm{with}}\ {\rm{boundary}}\ {\rm{condition}}\ u(1)=0
\end{equation*}
and then we study the properties of the eigenfunctions and establish the associated spectral gap estimate in the perturbative regime $0<b<b^{*}(k)$ with $b^{*}(k)$ small enough. Note firstly that $\mathcal{H}_{b}$ is self-adjoint with respect to the measure $e^{-\frac{b|y|^{2}}{2}}dy$ and has (up to a shift) compact resolvent in the Hilbert space $H^{1}_{b}$ (see (\ref{eq:4})) and hence discrete spectrum. However, the limit as $b\rightarrow 0$ is the Laplacian with resonant eigenmodes and continuous spectrum which generates some singularity at $b=0$. Hence we change variables and set $u(\sqrt{b}y)=v(y)$, then it's equivalent to study the operator
\begin{equation*}
    H_{b}=-\Delta+\Lambda,\quad u(\sqrt{b})=0,
\end{equation*}
which is formally a perturbation of  the standard harmonic oscillator with the Dirichlet boundary condition.
\subsection{Properties of the harmonic oscillator}
In this subsection, we recall some basic properties of harmonic oscillator without proof.\\
Consider firstly the following Laguerre equation 
\begin{equation*}
	xy^{\prime\prime}+(\mu+1-x)y^{\prime }+\lambda{y}=0\ (\mu>-1)
\end{equation*}
or equivalently  in Strum-Liouville form 
\begin{equation*}
	(x^{\mu+1}e^{-x}y^{\prime})^{\prime}+\lambda x^{\mu}e^{-x}y=0\ (\mu>-1).
\end{equation*}
As was shown in \cite{MF}\cite{L}\cite{SG} that 
$ \lambda=n\ (n\in\NN) $ are the only eigenvalues of the above equation under the constriant that the solution $ y $ is analytic at zero and $ y\in L^{2}_{\omega} $. Moreover, when $ \lambda=n $, the corresponding eigenfunction is
\begin{equation*}
	L_{n}^{\mu}(x)=\sum_{k=0}^{n}(-1)^{k}\frac{\Gamma(\mu+n+1)}{k!(n-k)!\Gamma(\mu+k+1)}x^{k}.
\end{equation*}
And for any $ m,n\in\NN $, we have
\begin{equation*}
	\int_{0}^{+\infty}L_{n}^{\mu}(x)L_{m}^{\mu}x^{\mu}e^{-x}dx=\frac{\Gamma(\mu+n+1)}{n!}\delta_{mn}.
\end{equation*}
Then we consider the harmonic oscillator $ -\Delta+\Lambda $ on $ (H^{1}_{\rho,0},\langle\cdot,\cdot\rangle_{0}) $. The operator is self-adjoint for the $ \langle\cdot,\cdot\rangle_{0} $ inner product, since
\begin{equation}\label{eq:7}
-\Delta+\Lambda=-\frac{1}{\rho z^{2}}\partial_{z}(\rho{z}^{2}\partial_{z}).
\end{equation}
If we take 
\begin{equation*}
	\mu=\frac{1}{2},
\end{equation*}
 and set 
\begin{equation}\label{eq:8}
	P_{k}(x)=\frac{L_{k}^{\mu}\Big(\frac{x^{2}}{2}\Big)}{A_{k}}.
\end{equation}
Then $ P_{k}(x) $ diagonalises the harmonic oscillator:
\begin{equation}\label{eq:9}
	(-\Delta+\Lambda)P_{k}=2kP_{k},
\end{equation}
and we have
\begin{equation}\label{eq:10a}
	\langle P_{k},P_{j}\rangle_{0}=\delta_{jk}.
\end{equation}
Moreover, direct computation gives
\begin{equation}\label{eq:10}
	P_{k}(0)=\frac{B_{k}}{A_{k}}=C_{k}=\frac{\sqrt{\Gamma(\frac{1}{2}+k+1)}}{2^{\frac{1}{4}}\Gamma(\frac{1}{2}+1)\sqrt{k!}},\quad \forall k\in\NN,
\end{equation}
\begin{equation}\label{eq:11}
	\Lambda{P_{k}}(x)=2k\bigg(P_{k}(x)-\frac{2k+1}{2k}\frac{A_{k-1}}{A_{k}}P_{k-1}(x)\bigg),\quad \forall k\in{\NN^{*}}.
\end{equation}
Finally there holds the spectral gap estimate for the harmonic oscillator: $ \forall u\in{H}^{1}_{\rho,0} $ with
\begin{equation*}
	\langle u,P_{j}\rangle_{0}=0,\quad 0\leq{j}\leq{k},
\end{equation*}
we have
\begin{equation}\label{eq:12}
	\Arrowvert\partial_{z}u\Arrowvert^{2}_{L^{2}_{\rho,0}}\geq(2k+2)\Arrowvert{u}\Arrowvert^{2}_{L^{2}_{\rho,0}}. 
\end{equation}
\subsection{Almost invertibility of the renormalised operator}
We consider in this section the renormalised operator
\begin{equation*}
	H_{b}=-\Delta+\Lambda\quad {\mathrm{with\ the\ boundary\ condition\ }} u(\sqrt{b})=0
\end{equation*}
in the radial case where $ 0<b<b^{*} $ small enough. Thanks to the boundary condition $ u(\sqrt{b})=0 $ and (\ref{eq:7}), $ H_{b} $ is self-adjoint  on the Hilbert space $ L^{2}_{\rho,b} $ given by (\ref{eq:3}). We will show the invertiliblity property of $ H_{b} $ simliar  to that has been done in \cite{MP} which serves as the starting point of the Lyapunov-Schmidt argument.\\
Fix a frequency size $ K\in\NN $ and a sufficiently small parameter $ b^{*}(K) $ such that
\begin{equation*}
	0<b<b^{*}(K)\ll{1}.
\end{equation*}
Define the Gram matrix
\begin{equation}\label{eq:13}
	M_{b,k}=(\langle P_{i},P_{j}\rangle_{b})_{0\leq{i,j}\leq{k}},\quad {0}\leq k\leq{K}.
\end{equation}
Since (\ref{eq:10a}) implies
\begin{equation*}
	\langle P_{i},P_{j}\rangle_{b}=\langle P_{i},P_{j}\rangle_{0}-\int_{0}^{\sqrt{b}}P_{i}P_{j}y^{2}e^{-y^{2}/2}dy=\delta_{ij}+O(b^{\frac{3}{2}}),
\end{equation*}
one has
\begin{equation}\label{eq:15}
	M_{b,k}=Id+O(b^{\frac{3}{2}}),
	\end{equation}
which is invertible for $ 0\leq{b}<b^{*} $ small enough. Then for the vector $ \vec{P_{k}}(z)=(P_{j}(z))_{0\leq{j}\leq{k}} $, we introduce the function 
\begin{equation}\label{eq:16}
	m_{k}(b,z)=\langle M_{b,k}^{-1}\vec{P_{k}}(\sqrt{b}),\vec{P_{k}}(z)\rangle.
\end{equation}
As a consequence of  (\ref{eq:8}), (\ref{eq:10}), (\ref{eq:15}), 
\begin{equation*}
	m_{k}(b,z)=\sum_{j=0}^{k}(C_{j}+O(b))P_{j}(z).
	\end{equation*}
 
\begin{lemma}[Near inversion of $H_{b}-2k$]\label{le:1}
	Let $ k\in\NN $ and $ 0<b<b^{*}(k) $ small enough. Then for all $f\in{L^{2}_{\rho,b}} $ with
	\begin{equation*}
		\langle f,P_{j}\rangle_{b}=0, \quad 0\leq{j}\leq{k},
	\end{equation*}
there exists a unique solution $ u\in{H^{1}_{\rho,b}} $ to the equation
	\begin{equation}\label{le1.2}
	\left\{\begin{aligned}
&\tilde{H}_{b,k}u=f, \\
&\langle u,P_{j} \rangle_{b}=0,\quad \forall\ 0\leq{j}\leq{k},
	\end{aligned}\right.
\end{equation}
where $$\quad \tilde{H}_{b,k}u=(H_{b}-2k)u-be^{-\frac{b}{2}}m_{k}(b,z)\partial_{z}u(\sqrt{b}).$$ Moreover, the following estimates hold
\begin{equation}\label{le1.3}
	\begin{aligned}
\Arrowvert{u}\Arrowvert_{L^{2}_{\rho,b}}&+\Arrowvert\partial_{z}u\Arrowvert_{L^{2}_{\rho,b}}+\Arrowvert\Delta{u}\Arrowvert_{L^{2}_{\rho,b}}+\Arrowvert\Lambda{u}\Arrowvert_{L^{2}_{\rho,b}}\\
&+	\Arrowvert zu \Arrowvert_{L^{2}_{\rho,b}}+\Arrowvert z^{2}u\Arrowvert_{L^{2}_{\rho,b}}+\lvert\sqrt{b}\partial_{z}u(\sqrt{b})\rvert\lesssim\Arrowvert{f}\Arrowvert_{L^{2}_{\rho,b}}.
 \end{aligned}
\end{equation}
\end{lemma}
Before giving the proof of lemma \ref{le:1}, we firstly recall a weighted Poincar\'{e}-type  inequality which will be used frequently in the following proof.
\begin{lemma}\label{le:2}
	Suppose $ u $ be a radial function on $ \RR^{3} $ and $ u\in{H^{1}_{\rho,0}} $, then we have
	\begin{equation*}
		\int_{0}^{+\infty}(zu)^{2}\rho{z}^{2}dz\lesssim\int_{0}^{+\infty}u^{2}\rho{z}^{2}dz+\int_{0}^{+\infty}(\partial_{z}u)^{2}\rho{z}^{2}dz.
	\end{equation*}
	\end{lemma} 
\begin{proof}
	Integrating by parts and note that $ \partial_{z}\rho=-z\rho $, we have
\begin{align*}
\int_{0}^{+\infty}(zu)^{2}\rho z^{2}dz&=\int_{0}^{+\infty}z^{3}u^{2}d(-\rho)=3\int_{0}^{+\infty}u^{2}\rho{z}^{2}dz+2\int_{0}^{+\infty}z^{3}\partial_{z}u{u}\rho dz\\
&\leq 3\int_{0}^{+\infty}u^{2}\rho{z}^{2}dz+\frac{1}{2}\int_{0}^{+\infty}(zu)^{2}\rho{z}^{2}dz+2\int_{0}^{+\infty}(\partial_{z}u)^{2}\rho z^{2}dz.
\end{align*}
Hence,
\begin{equation*}
	\int_{0}^{+\infty}(zu)^{2}\rho z^{2}dz\leq6\int_{0}^{+\infty}u^{2}\rho{z}^{2}dz+4\int_{0}^{+\infty}(\partial_{z}u)^{2}\rho z^{2}dz.
\end{equation*}
	\end{proof}
Now we come back to the proof of Lemma \ref{le:1}.
\begin{proof}[Proof of Lemma \ref{le:1}]  We divide the proof into two steps. \\ 
 {\it Step 1. The existence and uniqueness. }	We will show the existence and uniqueness of \eqref{le1.2} in $ H^{1}_{\rho,b} $ via a standard Lax-Milgram argument. For $ k\in\NN $, we consider the energy functional $ \mathcal{F}:\ H^{1}_{\rho,b}\rightarrow\RR $ where
 \begin{equation*}
 	\mathcal{F}(u)=\int_{\sqrt{b}}^{+\infty}|\partial_{z}u|^{2}\rho{z}^{2}dz-2k\int_{\sqrt{b}}^{+\infty}u^{2}\rho{z}^{2}dz-\langle f,u \rangle_{b},
 \end{equation*}
 over the constraint set 
 \begin{equation*}
     \mathcal{C}=\{u\in H^{1}_{\rho,b}|\ \langle u,P_{j}\rangle_{b}=0,\quad \forall\ 0\leq{j}\leq{k} \}.
 \end{equation*}
Let
 \begin{equation*}
 	I_{b}=\inf\limits_{u\in{\mathcal{C}}} \mathcal{F}(u).
 \end{equation*}
Note that the spectral gap estimate (\ref{eq:12}) and Lemma \ref{le:2} imply that for all ${v}\in{H}^{1}_{\rho,0}  $ with $ \langle v,P_{j}\rangle={0},\ \forall\ 0\leq{j}\leq{k} $, we have
\begin{equation}\label{le1.4}
	\int_{0}^{+\infty}|\partial_{z}v|^{2}\rho{z}^{2}dz-2k\int_{0}^{+\infty}{v}^{2}\rho{z}^{2}dz\gtrsim\int_{0}^{+\infty}(1+|z|^{2})|v|^{2}\rho{z}^{2}dz.
\end{equation}
Applying this to $ v(z)=u(z)\textbf{1}_{z\geq\sqrt{b}}\in H^{1}_{\rho,0} $, we conclude that $ I_b>-\infty $ by Cauchy inequality, and  any minimizing sequence $ u_{n} $ is uniformly bounded in $ H^{1}_{\rho,b} $. Therefore, up to a subsequence and due to the compactness of the Sobolev embedding  $ H^{1}_{\rho,b}\hookrightarrow\hookrightarrow\ L^{2}_{\rho,b} $,  there exists $ u\in{H}^{1}_{\rho,b} $, such that
\begin{equation*}
	u_{n}\rightharpoonup{u}\quad  in\ H^{1}_{\rho,b},\ \ \mathrm{and}\ 
		 u_{n}\rightarrow{u}\quad in\ L^{2}_{\rho,b}.
\end{equation*}
In particular, since $ H^{1}(\RR)\subset{L^{\infty}(\RR)} $, we obtain
\begin{equation}\label{le1.6}
	u(\sqrt{b})=0,\quad \langle u,P_{j}\rangle_{b}=0\quad \forall\ 0\leq{j}\leq{k}.
\end{equation}
Thus $ u $ is a minimizer of $ \mathcal{F} $ over $\mathcal{C}$. By Lagrange multiplier theorem, there exists $ \lambda_{j}\in{\RR} $ such that
\begin{equation}\label{le1.7}
	(H_{b}-2k)u=f-\sum_{j=0}^{k}\lambda_{j}P_{j}.
	\end{equation}
Standard regularity theory of elliptic equation implies that $ u\in H_{loc}^{2}(z\geq\sqrt{b}) $. Then taking inner product with $ P_{j} $ on both sides of (\ref{le1.7}), we have
\begin{equation}\label{n1}
    \langle H_{b}, P_{j}\rangle_{b}=\langle f,P_{j}\rangle_{b}-\sum_{i=0}^{k}\lambda_{i}\langle P_{i},P_{j}\rangle_{b}.
\end{equation}
Direct computation gives
\begin{equation*}
\begin{aligned}
\langle H_{b}u,P_{j} \rangle_{b}=&\int_{\sqrt{b}}^{+\infty}\left(-\frac{1}{\rho{z}^{2}}\partial_{z}(\rho{z}^{2}\partial_{z}{u})P_{j}\rho{z}^{2}\right)dz\\
=&-\rho{z}^{2}\partial_{z}uP_{j}\bigg|^{+\infty}_{\sqrt{b}}+\int^{+\infty}_{\sqrt{b}}\rho{z}^{2}\partial_{z}u\partial_{z}P_{j}dz\\
=&\rho(\sqrt{b})b\partial_{z}u(\sqrt{b})P_{j}(\sqrt{b})+\langle u,H_{b}P_{j}\rangle_{b}\\
=&e^{-\frac{b}{2}}b\partial_{z}u(\sqrt{b})P_{j}(\sqrt{b})+2j\langle u,P_{j}\rangle_{b}\\
=&e^{-\frac{b}{2}}b\partial_{z}u(\sqrt{b})P_{j}(\sqrt{b}),
\end{aligned}
\end{equation*}
which together with (\ref{n1}), gives $$e^{-\frac{b}{2}}b\partial_{z}u(\sqrt{b})P_{j}(\sqrt{b})=-\sum_{i=0}^{k}\lambda_{i}\langle P_{i},P_{j}\rangle_{b}.$$ Combining with  (\ref{eq:13}), we obtain
$$(\lambda_{i})_{0\leq{i}\leq{k}}=-(M^{-1}_{b,k}\vec{{P}_{k}}(\sqrt{b}))e^{-\frac{b}{2}}b\partial_{z}u(\sqrt{b}).$$

Hence
\begin{equation}\label{le1.9}
	(H_{b}-2k)u=f+m_{k}(b,z)be^{-\frac{b}{2}}\partial_{z}u(\sqrt{b}),
	\end{equation}
which is exactly the equation (\ref{le1.2}).
Now observe that by the definition of $ m_{k}(b,z) $  and (\ref{le1.6}), we have
\begin{equation}\label{le1.10}
	\langle m_{k}(b,z),u\rangle_{b}=0.
\end{equation}
Then taking inner product with $ u $ on both sides of (\ref{le1.9}), combining with (\ref{le1.10}) and the identity $ \langle H_{b}u,u \rangle_{b}=\Arrowvert\partial_{z}u\Arrowvert^{2}_{L^{2}_{\rho,b}} $, we get 
\begin{equation}\label{le1.11}
	\Arrowvert\partial_{z}u\Arrowvert^{2}_{L^{2}_{\rho,b}}-2k\Arrowvert{u}\Arrowvert^{2}_{L^{2}_{\rho,b}}=\langle f,u \rangle_{b}.
\end{equation}
Applying (\ref{le1.4}) to $ v(z)=u(z)\textbf{1}_{z\geq\sqrt{b}}\in H^{1}_{\rho,0} $ and together with (\ref{le1.11}) yields
\begin{equation*}
	\Arrowvert{u}\Arrowvert^{2}_{L^{2}_{\rho,b}}\lesssim 	\Arrowvert\partial_{z}u\Arrowvert^{2}_{L^{2}_{\rho,b}}-2k\Arrowvert{u}\Arrowvert^{2}_{L^{2}_{\rho,b}}=\langle f,u \rangle_{b}\leq\Arrowvert{f}\Arrowvert_{L^{2}_{\rho,b}}\Arrowvert{u}\Arrowvert_{L^{2}_{\rho,b}}.
\end{equation*}
Hence
\begin{equation}\label{le1.12}
\Arrowvert{u}\Arrowvert_{L^{2}_{\rho,b}}\lesssim\Arrowvert{f}\Arrowvert_{L^{2}_{\rho,b}},
\end{equation}
which  implies the uniqueness of the solution to the equation (\ref{le1.2}).

{\it Step 2. Proof of (\ref{le1.3}).} Firstly, from (\ref{le1.11})  and (\ref{le1.12}), 
\begin{equation}\label{le1.13}
	\Arrowvert\partial_{z}u\Arrowvert_{L^{2}_{\rho,b}}\lesssim \Arrowvert{f}\Arrowvert_{L^{2}_{\rho,b}}.
\end{equation}
Then by Lemma \ref{le:1}, we obtain
\begin{equation}\label{n1}
	\Arrowvert zu \Arrowvert_{L^{2}_{\rho,b}}\lesssim \Arrowvert f\Arrowvert_{L^{2}_{\rho,b}}.
\end{equation}

We now show the estimate of $ |\sqrt{b}\partial_{z}u(\sqrt{b})| $.  Integration by parts gives
\begin{equation}\label{le1.14}
    \begin{aligned}
        \bigg\langle H_{b}u,\frac{1}{z}\bigg\rangle_{b}&=-\int_{\sqrt{b}}^{+\infty}\frac{1}{z}d(\rho{z}^{2}\partial_{z}u)\\
        &=e^{-\frac{b}{2}}\sqrt{b}\partial_{z}u(\sqrt{b})-\int_{\sqrt{b}}^{+\infty}\rho\partial_{z}udz\\
        &=e^{-\frac{b}{2}}\sqrt{b}\partial_{z}u(\sqrt{b})-\int_{\sqrt{b}}^{+\infty}\rho{z}udz.
    \end{aligned}
\end{equation}
Hence 
\begin{equation*}
    \begin{aligned}
        |\sqrt{b}\partial_{z}u(\sqrt{b})|&\lesssim\bigg|	\bigg\langle H_{b}u,\frac{1}{z}\bigg\rangle_{b}\bigg|+\bigg| \int_{\sqrt{b}}^{+\infty}\rho{z}udz\bigg|\\
        &\lesssim\Arrowvert{H_{b}}u\Arrowvert_{L^{2}_{\rho,b}}+\Arrowvert{u}\Arrowvert_{L^{2}_{\rho,b}}\\
        &\lesssim\Arrowvert{f}+2ku+be^{-\frac{b}{2}}m_{k}(b,z)\partial_{z}u(\sqrt{b})\Arrowvert_{L^{2}_{\rho,b}}+\Arrowvert{f}\Arrowvert_{L^{2}_{\rho,b}}\\
        &\lesssim b|\partial_{z}u(\sqrt{b})|+\Arrowvert{f}\Arrowvert_{L^{2}_{\rho,b}}.\\
            \end{aligned}
\end{equation*}
which implies
\begin{equation}\label{le1.15}
	|\sqrt{b}\partial_{z}u(\sqrt{b})|\lesssim \Arrowvert{f}\Arrowvert_{L^{2}_{\rho,b}}
\end{equation}
for $ 0<b<b^{*} $ small enough and thus
\begin{equation}\label{le1.16}
	\Arrowvert{H_{b}u}\Arrowvert_{L^{2}_{\rho,b}}\lesssim\Arrowvert{f}\Arrowvert_{L^{2}_{\rho,b}}.
\end{equation}
We then estimate $ \Arrowvert\Lambda{u}\Arrowvert_{L^{2}_{\rho,b}} $. 
Firstly, integrating by parts again yields
\begin{equation*}
	\begin{aligned}
		\langle z\partial_{z}u,H_{b}u\rangle_{b}&=\int_{\sqrt{b}}^{+\infty}z\partial_{z}u\big(-\frac{1}{\rho{z}^{2}}\partial_{z}(\rho{z}^{2}\partial_{z}u)\big)\rho{z}^{2}dz\\
		&=-z\partial_{z}u\rho{z}^{2}\partial_{z}u\bigg|^{+\infty}_{\sqrt{b}}+\int_{\sqrt{b}}^{+\infty}(z\partial_{zz}u+\partial_{z}u)\rho{z}^{2}\partial_{z}udz\\
		&=b^{\frac{3}{2}}e^{-\frac{b}{2}}(\partial_{z}u)^{2}(\sqrt{b})+\int_{\sqrt{b}}^{+\infty}(\partial_{z}u)^{2}\rho{z}^{2}dz+\int_{\sqrt{b}}^{+\infty}\rho{z}^{3}\partial_{z}u\partial_{zz}udz\\
		&=\frac{1}{2}b^{\frac{3}{2}}e^{-\frac{b}{2}}(\partial_{z}u)^{2}(\sqrt{b})-\frac{1}{2}\int_{\sqrt{b}}^{+\infty}(\partial_{z}u)^{2}\rho{z}^{2}dz+\frac{1}{2}\int_{\sqrt{b}}^{+\infty}(z\partial_{z}u)^{2}\rho{z}^{2}dz
	\end{aligned}
\end{equation*}
Hence by (\ref{le1.13}), (\ref{le1.15}), (\ref{le1.16}) and Cauchy-Schwarz inequality, 
\begin{equation*}	\Arrowvert\Lambda{u}\Arrowvert^{2}_{L^{2}_{\rho,b}}\lesssim\Arrowvert{f}\Arrowvert^{2}_{L^{2}_{\rho,b}}+\Arrowvert{\Lambda{u}}\Arrowvert_{L^{2}_{\rho,b}}\Arrowvert{H_{b}u}\Arrowvert_{L^{2}_{\rho,b}}\lesssim\Arrowvert{f}\Arrowvert^{2}_{L^{2}_{\rho,b}}+\Arrowvert{f}\Arrowvert_{L^{2}_{\rho,b}}\Arrowvert{\Lambda{u}}\Arrowvert_{L^{2}_{\rho,b}},
\end{equation*}
from which we  deduce that
\begin{equation}\label{n2}
	\Arrowvert\Lambda{u}\Arrowvert_{L^{2}_{\rho,b}}\lesssim\Arrowvert{f}\Arrowvert_{L^{2}_{\rho,b}}.
\end{equation}
Then applying Lemma \ref{le:2} to $zu$ and combining with (\ref{le1.12}), (\ref{n1}), (\ref{n2}), we get
\begin{equation*}
    \lVert z^{2}u\rVert_{L^{2}_{{\rho},b}}\lesssim \lVert f\rVert_{L^{2}_{\rho,b}}.
\end{equation*}
Finally, from (\ref{le1.16}) and (\ref{n2}), we also have
\begin{equation*}
\Arrowvert\Delta{u}\Arrowvert_{L^{2}_{\rho,b}}\lesssim\Arrowvert{f}\Arrowvert_{L^{2}_{\rho,b}}.
\end{equation*}
Then we finished the proof of (\ref{le1.3}).
 \end{proof} 
\subsection{Diagonalisation of $H_{b}$} 
We are now in a position to diagonalise $ H_{b} $ for frequencies $ 0\leq{k}\leq{K} $ under the assumption that $ 0<b<b^{*}(K) $ small enough.   The proof relies on a Lyapunov-Schmidt type argument which has already been used   to deal with the similar problem in two dimension in \cite{MP}.
\begin{proposition}[Diagonalisation of $ H_{b} $]
	Let $ K\in\NN $, then for any $ 0<b<b^*{(K)}$ small enough, $H_{b}$ admits a sequence of eigenvalues and eigenfunctions which satisfy
	\begin{equation}\label{pr2.3:1}
		H_{b}\psi_{b,k}=\lambda_{b,k}\psi_{b,k},\quad \psi_{b,k}\in{H^{1}_{\rho,b}}, \quad \forall\ 0\leq{k}\leq{K}.
	\end{equation}
\begin{itemize}
	\item Expansion of eigenvalues:\\
	\begin{equation}\label{pr2.3:2}
		\lambda_{b,k}=2k+C^{2}_{k}\sqrt{b}+O(b), \quad |\partial_{b}\lambda_{b,k}|\lesssim\frac{1}{\sqrt{b}}.
	\end{equation}
\item Expansion of eigenfunctions:\\
\begin{equation}\label{pr2.3:3}
		\left\{\begin{aligned}
&\psi_{b,k}(z)=T_{b,k}(z)+\tilde{\psi}_{b,k}(z),\\
&T_{b,k}(z)=P_{k}(z)\bigg(\frac{1}{z}-\frac{1}{\sqrt{b}}\bigg)+\sum_{j=0}^{k-1}\mu_{b,jk}P_{j}(z)\bigg(\frac{1}{z}-\frac{1}{\sqrt{b}}\bigg).
	\end{aligned}\right.
\end{equation}
with
\begin{equation}\label{pr2.3:4}
	\mu_{b,jk}=\frac{\langle Q_{k},P_{j}\rangle_{0}}{2(j-k)}\sqrt{b}+O(b),\quad |\partial_{b}\mu_{b,jk}|\lesssim\frac{1}{\sqrt{b}},
\end{equation}
here $ Q_{k}=-\frac{1}{z}P_{k}+\frac{2}{z^{2}}\partial_{z}P_{k} $.\\
Moreover, the following estimates hold:
\begin{align}\label{pr2.3:5}	\Arrowvert\tilde{\psi}_{b,k}\Arrowvert_{H^{1}_{\rho,b}}+\Arrowvert\Delta\tilde{\psi}_{b,k}\Arrowvert_{L^{2}_{\rho,b}}&+\Arrowvert\Lambda\tilde{\psi}_{b,k}\Arrowvert_{L^{2}_{\rho,b}}+\Arrowvert z\tilde{\psi}_{b,k}\Arrowvert_{L^{2}_{\rho,b}}\nonumber\\
    &+\Arrowvert z^{2}\tilde{\psi}_{b,k}\Arrowvert_{L^{2}_{\rho,b}}+\sqrt{b}|\partial_{z}\tilde{\psi}_{b,k}(\sqrt{b})|\lesssim 1,
	\end{align}
    and
\begin{equation}\label{pr2.3.6}
	\Arrowvert \partial_{z}\partial_{b}\tilde{\psi}_{b,k}\Arrowvert_{L^{2}_{\rho,b}}+\Arrowvert\partial_{b}\tilde{\psi}_{b,k}\Arrowvert_{L^{2}_{\rho,b}}+\sqrt{b}|\partial_{z}\partial_{b}\tilde{\psi}_{b,k}(\sqrt{b})|\lesssim \frac{1}{b}.
\end{equation}
\item Spectral gap estimate: if $ u\in H^{1}_{\rho,b} $ with
\begin{equation*}
	\langle u,\psi_{b,j}\rangle_{b}=0,\quad \forall\ 0\leq{j}\leq{k},
\end{equation*}
then 
\begin{equation}\label{pr2.3.7}
	\Arrowvert \partial_{z}u\Arrowvert^{2}_{L^{2}_{\rho,b}}\geq(2k+2+O(\sqrt{b}))\Arrowvert u\Arrowvert^{2}_{L^{2}_{\rho,b}}.
\end{equation}
\item Some further estimates: for any $ 0\leq{j}\leq{k} $,
\begin{equation}\label{pr2.3.8}
	\langle \psi_{b,k},\psi_{b,i}\rangle_{b}=\frac{1}{b}(\delta_{ik}+O(\sqrt{b})),
\end{equation}
and
\begin{equation}\label{pr2.3.9}
	\langle b\partial_{b}\psi_{b,k},\psi_{b,i}\rangle_{b}=-\frac{1}{2b}(\delta_{ik}+O(\sqrt{b})).
\end{equation}
Moreover,
\begin{equation}\label{pr2.3.10}
	\Arrowvert 2b\partial_{b}\psi_{b,k}+\psi_{b,k}\Arrowvert_{L^{2}_{\rho,b}}+\Arrowvert H_{b}(2b\partial_{b}\psi_{b,k}+\psi_{b,k})\Arrowvert_{L^{2}_{\rho,b}}\lesssim{1},
	\end{equation}
\begin{equation}\label{pr2.3.11}
	\Arrowvert\Lambda\psi_{b,0}\Arrowvert_{L^{2}_{\rho,b}}\lesssim{1},\quad \Arrowvert\Lambda\psi_{b,k}\Arrowvert_{L^{2}_{\rho,b}}\lesssim\frac{1}{\sqrt{b}},\quad  \forall\ k\geq{1},
\end{equation}
\begin{equation}\label{pr2.3.12}
	\Arrowvert z\psi_{b,k}\Arrowvert_{L^{2}_{\rho,b}}\lesssim\frac{1}{\sqrt{b}},\quad \Arrowvert z^{2}\psi_{b,k}\Arrowvert_{L^{2}_{\rho,b}}\lesssim\frac{1}{\sqrt{b}},
	\end{equation}
\begin{align}\label{pr2.3.13}
	\bigg\lVert \Lambda\psi_{b,k}&-2k\bigg(\psi_{b,k}-\frac{2k+1}{2k}\frac{A_{k-1}}{A_{k}}\psi_{b,k-1}\bigg)\bigg\rVert_{L^{2}_{\rho,b}}\nonumber\\
    &+\bigg\lVert H_{b}\bigg(\Lambda\psi_{b,k}-2k\bigg(\psi_{b,k}-\frac{2k+1}{2k}\frac{A_{k-1}}{A_{k}}\psi_{b,k-1}\
    \bigg)\bigg)\bigg\rVert_{L^{2}_{\rho,b}}\lesssim{1}.
\end{align}
\end{itemize}
\end{proposition} 
       \begin{proof} We divide the proof into five steps. \\
       {\it\bf Step 1. Lyapunov-Schmidt argument.} For simplity, we firstly introduce the unverisal profile
       		\begin{equation*}
       			\mathcal{T}_{b}(z):=\frac{1}{z}-\frac{1}{\sqrt{b}}
       		\end{equation*}
       	and consider $ T_{b,k}(z) $ which is given by (\ref{pr2.3:3}). Direct computation gives
       	\begin{equation*}
       		\begin{aligned}
       			H_{b}T_{b,k}&=H_{b}(P_{k}\mathcal{T}_{b})+\sum_{j=0}^{k-1}\mu_{b,jk}H_{b}(P_{j}\mathcal{T}_{b})\\
       			&=(H_{b}P_{k})\mathcal{T}_{b}+2\frac{\partial_{z}P_{k}}{z^{2}}-\frac{P_{k}}{z}+\sum_{j=0}^{k-1}\mu_{b,jk}\bigg((H_{b}P_{j})\mathcal{T}_{b}+2\frac{\partial_{z}P_{j}}{z^{2}}-\frac{P_{j}}{z}\bigg).
       		\end{aligned}
       	\end{equation*}
Set
       \begin{equation*}
       	Q_{k}(z)=2\frac{\partial_{z}P_{k}}{z^{2}}-\frac{P_{k}}{z},
       \end{equation*}
   which combines with (\ref{eq:9}) yields
   \begin{equation}\label{pr2.3:6}
   	(H_{b}-2k)T_{b,k}(z)=Q_{k}(z)+\sum_{j=0}^{k-1}\mu_{b,jk}[2(j-k)\mathcal{T}_{b}(z)P_{j}(z)+Q_{j}(z)].
   \end{equation}
Now we study the eigenvalue problem
\begin{equation}\label{pr2.3:7}
	(H_{b}-2k)\psi_{b,k}=\mu_{k}\psi_{b,k}.
\end{equation}
If the solution is of the form (\ref{pr2.3:3}), then in view of (\ref{pr2.3:7}) and (\ref{pr2.3:6}),  $ \tilde{\psi}_{b,k} $ satisfies
\begin{equation}\label{pr2.3:8}
	\begin{aligned}
		(H_{b}-2k)\tilde{\psi}_{b,k}&=\mu_{k}(T_{b,k}+\tilde{\psi}_{b,k})-(H_{b}-2k)T_{b,k}\\
		&=-Q_{k}(z)-\sum_{j=0}^{k-1}\mu_{b,jk}[2(j-k)\mathcal{T}_{b}(z)P_{j}(z)+Q_{j}(z)]+\mu_{k}(T_{b,k}+\tilde{\psi}_{b,k})\\
		&:=f(\tilde{\psi}_{b,k}),		
	\end{aligned}
	\end{equation}
 where the coefficients $ \mu_{k} $, $ \mu_{b,jk} $  $ (0\leq{j}\leq{k-1}) $ can be determined by the relation
 \begin{equation*}
 	(M^{-1}_{b,k}(\langle f(\tilde{\psi}_{b,k}),\vec{P_{k}}\rangle_{b}))_{i}=be^{-\frac{b}{2}}\partial_{z}\tilde{\psi}_{b,k}(\sqrt{b})(M_{b,k}^{-1}\vec{P_{k}}(\sqrt{b}))_{i}\quad\ \forall\ 0\leq{i}\leq{k},
 \end{equation*}
 which is equivalent to
 \begin{equation}\label{pr2.3:9}
 \langle f(\tilde{\psi}_{b,k}),P_{j}\rangle_{b}=be^{-\frac{b}{2}}\partial_{z}\tilde{\psi}_{b,k}(\sqrt{b})P_{j}(\sqrt{b}),\quad \forall\ 0\leq{j}\leq{k}
 \end{equation}
Next, (\ref{pr2.3:8}) can be rewritten as
\begin{align*}
	\tilde{H}_{b,k}\tilde{\psi}_{b,k}=&(H_{b}-2k)\tilde{\psi}_{b,k}-m_{k}(b,z)be^{-\frac{b}{2}}\partial_{z}\tilde{\psi}_{b,k}(\sqrt{b})\nonumber\\
    =&f-\sum_{j=0}^{k}(M^{-1}_{b,k}(\langle f(\tilde{\psi}_{b,k}),\vec{P_{k}}\rangle_{b}))_{j}P_{j}:=F(\tilde{\psi}_{b,k}),
\end{align*}
where $m_{k}(b,z)$ is defined as in (\ref{eq:16}).\\
Thus in order to find $ \tilde{\psi}_{b,k} $, it suffices to solve the fixed point equation
\begin{equation*}
	\tilde{\psi}=\tilde{H}^{-1}_{b,k}\circ{F}(\tilde{\psi}).
\end{equation*}
We then consider the closed ball 
\begin{equation*}
	\begin{aligned}
		B_{\alpha}=\Big\{\tilde{\psi}\in {H}^{1}_{\rho,b}|&\ \Arrowvert\Delta\tilde{\psi}\Arrowvert_{L^{2}_{\rho,b}}+\sqrt{b}|\partial_{z}\tilde{\psi}(\sqrt{b})|+\Arrowvert\tilde{\psi}\Arrowvert_{H^{1}_{\rho,b}}+\Arrowvert\Lambda\tilde{\psi}\Arrowvert_{L^{2}_{\rho,b}}+\Arrowvert z\tilde{\psi}\Arrowvert_{L^{2}_{\rho,b}}\\
        &+\Arrowvert z^{2}\tilde{\psi}\Arrowvert_{L^{2}_{\rho,b}}\leq{\alpha},\ 
		\tilde{\psi}(\sqrt{b})=0,\ \langle \tilde{\psi},P_{j}\rangle_{b}=0,\ \forall\ 0\leq{j}\leq{k}\Big\},
	\end{aligned}
\end{equation*}
where $ \alpha $ is a universal large number to be determined later.\\
In what follows, we claim that $ \tilde{H}^{-1}_{b,k}\circ{F} $ is a contraction from $ B_{\alpha} $ to itself.\\

In order to prove the above claim, we firstly need to determine $ \mu_{k}$ and $\ \mu_{jk}$ for  $0\leq{j}\leq{k-1}$.
For any $ \tilde{\psi}\in{B_{\alpha}} $ and  $0\leq{i}\leq{k-1}$,
\begin{equation*}
	{f}(\tilde{\psi})=-Q_{k}(z)-\sum_{j=0}^{k-1}\mu_{b,jk}(2(j-k)\mathcal{T}_{b}(z)P_{j}(z)+Q_{j}(z))+\mu_{k}(T_{b,k}+\tilde{\psi}).
\end{equation*}
Therefore, for any $0\leq{i}\leq{k-1}$, 
\begin{equation*}
	\begin{aligned}
		\langle f(\tilde{\psi}),P_{i}\rangle_{b}=&-\langle Q_{k},P_{i}\rangle_{b}-\sum_{j=0}^{k-1}\mu_{b,jk}(\langle 2(j-k)\mathcal{T}_{b}P_{j},P_{i}\rangle_{b}\\
  &+\langle Q_{j},P_{i}\rangle_{b})+\mu_{k}\langle T_{b,k}+\tilde{\psi},P_{i}\rangle_{b}\\
		=&\frac{2(i-k)}{\sqrt{b}}\langle P_{i},P_{i}\rangle_{b}\mu_{b,ik}-\langle Q_{k},P_{i}\rangle_{b}-\sum_{\substack{0\leq{j}\leq{k-1}\\{j}\neq{i}}}\mu_{b,jk}\langle 2(j-k)\mathcal{T}_{b}P_{j},P_{i}\rangle_{b}	\\	
		&-\sum_{j=0}^{k-1}\mu_{b,jk}\langle Q_{j},P_{i}\rangle_{b}+\mu_{k}\langle T_{b,k}+\tilde{\psi},P_{i}\rangle_{b}-2(i-k)\mu_{b,ik}\bigg\langle \frac{1}{z}P_{i},P_{i}\bigg\rangle_{b},
	\end{aligned}
	\end{equation*}
and for $ i=k $, 
\begin{equation*}
	\begin{aligned}
			\langle f(\tilde{\psi}),P_{k}\rangle_{b}=&-\langle Q_{k},P_{k}\rangle_{b}-\sum_{j=0}^{k-1}\mu_{b,jk}(\langle 2(j-k)\mathcal{T}_{b}P_{j},P_{k}\rangle_{b}\\
   &+\langle Q_{j},P_{k}\rangle_{b})+\mu_{k}\langle T_{b,k}+\tilde{\psi},P_{k}\rangle_{b}\\
			=&-\frac{1}{\sqrt{b}}\langle P_{k},P_{k}\rangle_{b}\mu_{k}-\langle Q_{k},P_{k}\rangle_{b}\\
   &-\sum_{j=0}^{k-1}\mu_{b,jk}(2(j-k)\langle \mathcal{T}_{b}P_{j},P_{k}\rangle_{b}+\langle Q_{j},P_{k}\rangle_{b})\\
			&+\mu_{k}\bigg\langle \sum_{j=0}^{k-1}\mu_{b,jk}P_{j}\mathcal{T}_{b}+\tilde{\psi},P_{k}\bigg\rangle_{b}+\mu_{k}\bigg\langle P_{k},\frac{P_{k}}{z}\bigg\rangle_{b}.
		\end{aligned}
\end{equation*}
Then combining with the relation (\ref{pr2.3:9}),  for $  0\leq{i}\leq{k-1} $,  set 
 \begin{equation*}
 	\begin{aligned}
 	g_{i}(\vec{\mu})=&\frac{\langle Q_{k},P_{i}\rangle_{b}}{2(i-k)\langle P_{i},P_{i}\rangle_{b}}\sqrt{b}+\frac{\sqrt{b}}{2(i-k)\langle P_{i},P_{i}\rangle_{b}}\bigg(be^{-\frac{b}{2}}\partial_{z}\tilde{\psi}(\sqrt{b})P_{j}(\sqrt{b})\\
     &+\sum_{\substack{0\leq{j}\leq{k-1}\\{j}\neq{i}}}\mu_{b,jk}\langle 2(j-k)\mathcal{T}_{b}P_{j},P_{i}\rangle_{b}+\sum_{j=0}^{k-1}\mu_{b,jk}\langle Q_{j},P_{i}\rangle_{b}\\
     &-\mu_{k}\langle T_{b,k},P_{i}\rangle_{b}+2(i-k)\mu_{b,ik}\bigg\langle \frac{1}{z}P_{i},P_{i}\bigg\rangle_{b}\bigg),
 \end{aligned}
 \end{equation*}
and for $ i=k $, set
\begin{equation*}
	\begin{aligned}
		g_{k}(\vec{\mu})=&-\frac{\langle Q_{k},P_{k}\rangle_{b}}{\langle P_{k},P_{k}\rangle_{b}}\sqrt{b}+\frac{\sqrt{b}}{\langle P_{k},P_{k}\rangle_{b}}\bigg(-be^{-\frac{b}{2}}\partial_{z}\tilde{\psi}(\sqrt{b})P_{k}(\sqrt{b})\\
		&-\sum_{j=0}^{k-1}\mu_{b,jk}(2(j-k)\langle \mathcal{T}_{b}P_{j},P_{k}\rangle_{b}+\langle Q_{j},P_{k}\rangle_{b})\\
        &+\mu_{k}\bigg\langle \sum_{j=0}^{k-1}\mu_{b,jk}P_{j}\mathcal{T}_{b},P_{k}\bigg\rangle_{b}+\mu_{k}\bigg\langle P_{k},\frac{P_{k}}{z}\bigg\rangle_{b}
		\bigg).
	\end{aligned}
\end{equation*}
Thus the relation (\ref{pr2.3:9}) is equivalent to $ \vec{g}(\vec{\mu})=\vec{\mu} $, for $ \vec{\mu}=(\mu_{b,0k},\mu_{b,1k},...,\mu_{b,k-1k},\mu_{k}) $ and $ \vec{g}(\vec{\mu})=(g_{0}(\vec{\mu}),g_{1}(\vec{\mu}),...,g_{k}(\vec{\mu})) $. As a result, in order to show the existence of  $ \vec{\mu} $ satisfying (\ref{pr2.3:9}), it suffices  to   show that
%Then we show for fixed $ \tilde{\psi}\in{B_{\alpha}} $, there exists a unique $ \vec{\mu}=(\mu_{b,0k},\mu_{b,1k},...,\mu_{b,k-1k},\mu_{k}) $ satisfy $ \vec{g}(\vec{\mu})=\vec{\mu} $, here $ \vec{g}(\vec{\mu})=(g_{0}(\vec{\mu}),g_{1}(\vec{\mu}),...,g_{k}(\vec{\mu})) $ and $ |\vec{\mu}|\leq{C}\sqrt{b} $ for some universal $ C>0 $ to be defined later and $ b>0 $ is small enough.\\
\begin{equation*}
	\vec{g}:\ E_b\rightarrow E_b
\end{equation*}
is a strict contraction for some universal constant $ C>0 $ and $ b>0 $ small enough to be determined later, where $E_b=\{|||\vec{\mu}|||\leq{C\sqrt{b}}\}$. \\
Firstly, note that  $ \langle \frac{1}{z}P_{i},P_{j}\rangle_{b} $, $ \langle Q_{i},P_{j}\rangle_{b} $, $ \langle P_{i},P_{i}\rangle_{b} $ are all uniformly bounded and that $ \tilde{\psi}\in{B_{\alpha}} $, so we have
\begin{equation*}
	|||\vec{g}(\vec{\mu})|||\leq{M}\sqrt{b}+\alpha{M}b+M\sqrt{b}|||\vec{\mu}|||,
\end{equation*}
where $ M $ depends only on $ P_{i}\ (0\leq{i}\leq{k}) $ and $ k\in\NN $.  Taking $ C=2M $ and $ b>0 $ small enough such that $ \alpha\sqrt{b}+2{M}\sqrt{b}\leq\frac{1}{2} $,  we get
\begin{equation*}
	|||\vec{g}(\vec{\mu})|||\leq{C\sqrt{b}}.
\end{equation*}
Then for any $\vec{\mu},\vec{\nu}\in E_b$,
\begin{equation*}
  \begin{aligned}
  	|||\vec{g}(\vec{\mu})-\vec{g}(\vec{\nu})|||&\leq{M}\alpha\sqrt{b}|||\vec{\mu}-\vec{\nu}|||+M\sup_{0\leq{j}\leq{k}}|\mu_{k}\mu_{jk}-\nu_{k}\nu_{jk}|\\
  	&\leq{M}\alpha\sqrt{b}|||\vec{\mu}-\vec{\nu}|||+M|||\vec{\mu}-\vec{\nu}|||(|||\vec{\mu}|||+|||\vec{\nu}|||)\\
  	&\leq(M\alpha\sqrt{b}+4M^{2}\sqrt{b})|||\vec{\mu}-\vec{\nu}|||\\
  	&\leq\frac{1}{2}|||\vec{\mu}-\vec{\nu}|||
  \end{aligned}
\end{equation*}
provided $ b>0 $ is small enough. Hence $ \vec{g} $ is a strict contraction and there exists a unique $ \vec{\mu}\in E_b $ satisfying (\ref{pr2.3:9}). Moreover, it holds that
\begin{equation*}
	\mu_{ik}=\frac{\langle Q_{k},P_{i}\rangle_{0}}{2(i-k)}\sqrt{b}+O(b+b|\sqrt{b}\partial_{z}\tilde{\psi}(\sqrt{b})|)
\end{equation*}
and
\begin{equation*}
	\mu_{k}=-\langle Q_{k},P_{k}\rangle_{0}\sqrt{b}+O(b+b|\sqrt{b}\partial_{z}\tilde{\psi}(\sqrt{b})|).
\end{equation*}
In particular, we obtain
\begin{equation}\label{pr2.3:16}
	|||\vec{\mu}|||\lesssim\sqrt{b}+b|\sqrt{b}\partial_{z}\tilde{\psi}(\sqrt{b})|.
\end{equation}
Additionally,  integration by parts yields 
\begin{equation*}
	\langle Q_{k},P_{k}\rangle_{0}=\int_{0}^{+\infty}\bigg(-\frac{1}{z}P_{k}+\frac{2}{z^{2}}P_{k}'\bigg)P_{k}\rho{z}^{2}dz=-P_{k}^{2}(0)=-C_{k}^{2}.
\end{equation*}
 
 Now we prove that $\tilde{H}^{-1}_{b,k}\circ{F}$
	is a strict contraction from $ B_{\alpha} $  to $ B_{\alpha} $ for some $ \alpha>0 $ universally large enough. We prove $\tilde{H}^{-1}_{b,k}\circ{F}$ maps $B_\alpha$ onto itself first. Define the norm 
\begin{equation*}
	\Arrowvert\tilde{\psi}\Arrowvert_{X}=\Arrowvert\Delta\tilde{\psi}\Arrowvert_{L^{2}_{\rho,b}}+\sqrt{b}|\partial_{z}\tilde{\psi}(\sqrt{b})|+\Arrowvert\tilde{\psi}\Arrowvert_{H^{1}_{\rho,b}}+\Arrowvert\Lambda\tilde{\psi}\Arrowvert_{L^{2}_{\rho,b}}+\Arrowvert z\tilde{\psi}\Arrowvert_{L^{2}_{\rho,b}}+\Arrowvert z^{2}\tilde{\psi}\Arrowvert_{L^{2}_{\rho,b}}.
\end{equation*}
Then by the definition of $ f(\tilde{\psi}) $ (see (\ref{pr2.3:8})) and (\ref{pr2.3:16}), for any $\tilde{\psi}\in{B_{\alpha}} $,
\begin{equation}\label{in2}
	\begin{aligned}
		\Arrowvert{F}(\tilde{\psi})\Arrowvert_{L^{2}_{\rho,b}}=&\Arrowvert{f(\tilde{\psi})}-\sum_{j=0}^{k}(M^{-1}_{b,k}(\langle f(\tilde{\psi}),\vec{P_{k}}\rangle_{b}))_{j}P_{j}\Arrowvert_{L^{2}_{\rho,b}}\lesssim\Arrowvert{f(\tilde{\psi})}\Arrowvert_{L^{2}_{\rho,b}}\\
		\leq& \Arrowvert Q_{k}\Arrowvert_{L^{2}_{\rho,b}}+\sum_{j=0}^{k-1}|\mu_{b,jk}|(\Arrowvert\mathcal{T}_{b}P_{j}\Arrowvert_{L^{2}_{\rho,b}}+\Arrowvert Q_{j}\Arrowvert_{L^{2}_{\rho,b}})\\
  &+|\mu_{k}|(\Arrowvert T_{b,k}\Arrowvert_{L^{2}_{\rho,b}}+\Arrowvert\tilde{\psi}\Arrowvert_{L^{2}_{\rho,b}})\\
		\leq& M+C(\sqrt{b}+b|\sqrt{b}\partial_{z}\tilde{\psi}(\sqrt{b})|)\cdot\frac{M}{\sqrt{b}}\\
  &+C(\sqrt{b}+b|\sqrt{b}\partial_{z}\tilde{\psi}(\sqrt{b})|)\bigg(\frac{M}{\sqrt{b}}+\Arrowvert\tilde{\psi}\Arrowvert_{L^{2}_{\rho,b}}\bigg)\\
		\leq& CM+C\sqrt{b}\alpha{M}+C\sqrt{b}\alpha+Cb\alpha^{2},
	\end{aligned}
\end{equation}
here $ C$ and $M $ depend only on  $ P_{i}\ (0\leq{i}\leq{k})$ and $ k\in\NN $. As a consequence of the near inversion property of $ H_{b}-2k $ (see (\ref{le1.3})), we have
\begin{equation*}
	\Arrowvert \tilde{H}^{-1}_{b,k}\circ F(\tilde{\psi})\Arrowvert_{X}\lesssim\Arrowvert{F}(\tilde{\psi})\Arrowvert_{L^{2}_{\rho,b}}\leq CM+C\sqrt{b}\alpha{M}+C\sqrt{b}\alpha+Cb\alpha^{2}.
\end{equation*}
Therefore, for $ b>0 $ small enough, there exists a universal constant $ \alpha>0 $, such that
\begin{equation*}
		\Arrowvert \tilde{H}^{-1}_{b,k}\circ F(\tilde{\psi})\Arrowvert_{X}\leq\alpha.
\end{equation*}
Next we prove $\tilde{H}^{-1}_{b,k}\circ{F}$ is a strict contradiction on $B_\alpha$. Indeed, $ \forall \phi,\psi\in {B_{\alpha}} $,
\begin{equation*}
	\begin{aligned}
		\Arrowvert F(\phi)-F(\psi)\Arrowvert_{L^{2}_{\rho,b}}=&\bigg\lVert f(\phi)-\sum_{j=0}^{k}(M^{-1}_{b,k}(\langle f(\phi),\vec{P_{k}}\rangle_{b}))_{j}P_{j}\\
  &-f(\psi)+\sum_{j=0}^{k}(M^{-1}_{b,k}(\langle f(\psi),\vec{P_{k}}\rangle_{b}))_{j}P_{j}\bigg\rVert_{L^{2}_{\rho,b}}\\
		\lesssim& \Arrowvert f(\phi)-f(\psi)\Arrowvert_{L^{2}_{\rho,b}}.
	\end{aligned}
\end{equation*}
By \eqref{pr2.3:8}, 
\begin{equation*}
	\begin{aligned}
		f(\phi)-f(\psi)=&-\sum_{j=0}^{k-1}(\mu_{b,jk}(\phi)-\mu_{b,jk}(\psi))[2(j-k)\mathcal{T}_{b}(z)P_{j}(z)
  +Q_{j}(z)]\\
  &+\mu_{k}(\phi)\phi-\mu_{k}(\psi)\psi
		+(\mu_{k}(\phi)-\mu_{k}(\psi))P_{k}\mathcal{T}_{b}\\
  &+\mu_{k}(\phi)\bigg(\sum_{j=0}^{k-1}\mu_{b,jk}(\phi)P_{j}\mathcal{T}_{b}\bigg)-\mu_{k}(\psi)\bigg(\sum_{j=0}^{k-1}\mu_{b,jk}(\psi)P_{j}\mathcal{T}_{b}\bigg).
	\end{aligned}
\end{equation*}
Hence, 
\begin{equation*}
	\begin{aligned}
	\langle f(\phi)-f(\psi),P_{k}\rangle_{b}=&-\frac{1}{\sqrt{b}}\langle P_{k},P_{k}\rangle_{b}(\mu_{k}(\phi)-\mu_{k}(\psi))\\
 &-\bigg\langle \sum_{j=0}^{k-1}(\mu_{b,jk}(\phi)-\mu_{b,jk}(\psi))[2(j-k)\mathcal{T}_{b}P_{j}(z)
 +Q_{j}(z)],P_{k}\bigg\rangle_{b}\\
	&+\langle \mu_{k}(\phi)\phi-\mu_{k}(\psi)\psi,P_{k}\rangle_{b}+(\mu_{k}(\phi)-\mu_{k}(\psi))\bigg\langle \frac{P_{k}}{z},P_{k}\bigg\rangle_{b}\\
	&+\sum_{j=0}^{k-1}(\mu_{b,jk}(\phi)\mu_{k}(\phi)-\mu_{b,jk}(\psi)\mu_{k}(\psi))\langle P_{j}\mathcal{T}_{b},P_{k}\rangle_{b},
\end{aligned}
\end{equation*}
which combined with  (\ref{pr2.3:9}) that
\begin{equation*}
	\langle f(\phi)-f(\psi),P_{k}\rangle_{b}=be^{-\frac{b}{2}}\partial_{z}(\phi-\psi)(\sqrt{b})P_{k}(\sqrt{b}),
\end{equation*}
yields 
	\begin{align}
	\mu_{k}(\phi)-\mu_{k}(\psi)=&\frac{-b^{\frac{3}{2}}e^{-\frac{b}{2}}\partial_{z}(\phi-\psi)(\sqrt{b})}{\langle P_{k},P_{k}\rangle_{b}}P_{k}(\sqrt{b})\notag\\
	&-\sqrt{b}\sum_{j=0}^{k-1}(\mu_{b,jk}(\phi)-\mu_{b,jk}(\psi))\frac{\langle 2(j-k)\mathcal{T}_{b}P_{j}+Q_{j},P_{k} \rangle}{\langle P_{k},P_{k}\rangle_{b}}\notag\\
	&+\sqrt{b}\frac{\langle\mu_{k}(\phi)\phi-\mu_{k}(\psi)\psi,P_{k}\rangle_{b}}{\langle P_{k},P_{k}\rangle_{b}}+\sqrt{b}(\mu_{k}(\phi)-\mu_{k}(\psi))\frac{\langle \frac{P_{k}}{z},P_{k}\rangle_{b}}{\langle P_{k},P_{k}\rangle_{b}}\notag\\
	&+\sqrt{b}\sum_{j=0}^{k-1}(\mu_{b,jk}(\phi)\mu_{k}(\phi)-\mu_{b,jk}(\psi)\mu_{k}(\psi))\frac{\langle P_{j}\mathcal{T}_{b},P_{k}\rangle_{b}}{\langle P_{k},P_{k}\rangle_{b}}.\label{pr2.3:19}
\end{align}

Since $ \phi,\psi\in{B_{\alpha}} $,  the third term of (\ref{pr2.3:19}) is cancelled and
\begin{equation*}
	\begin{aligned}
		|\mu_{k}(\phi)-\mu_{k}(\psi)|\lesssim & b|\sqrt{b}\partial_{z}(\phi-\psi)(\sqrt{b})|+\sqrt{b}\sum_{j=0}^{k-1}|\mu_{b,jk}(\phi)-\mu_{b,jk}(\psi)|\\
		&+\sqrt{b}|\mu_{k}(\phi)-\mu_{k}(\psi)|+\sqrt{b}\sum_{j=0}^{k-1}|\mu_{b,jk}(\phi)\mu_{k}(\phi)-\mu_{b,jk}(\psi)\mu_{k}(\psi)|.
	\end{aligned}
\end{equation*}
 By (\ref{pr2.3:16}), the last term is bounded by
\begin{equation*}
	\begin{aligned}
		&\sqrt{b}\sum_{j=0}^{k-1}|\mu_{b,jk}(\phi)\mu_{k}(\phi)-\mu_{b,jk}(\psi)\mu_{k}(\psi)|\\
        \leq&\sqrt{b}\sum_{j=0}^{k-1}|\mu_{b,jk}(\phi)-\mu_{b,jk}(\psi)||\mu_{k}(\phi)|+|\mu_{k}(\phi)-\mu_{k}(\psi)||\mu_{b,jk}(\psi)|\\
		\lesssim & \sqrt{b}(\sqrt{b}+b|\sqrt{b}\partial_{z}\psi(\sqrt{b})|+b|\sqrt{b}\partial_{z}\phi(\sqrt{b})|)\\
  &\cdot\left(\sum_{j=0}^{k-1}|\mu_{b,jk}(\phi)-\mu_{b,jk}(\psi)|+|\mu_{k}(\phi)-\mu_{k}(\psi)|\right)\\
		\lesssim &(b+\alpha b^{\frac{3}{2}})\left(\sum_{j=0}^{k-1}|\mu_{b,jk}(\phi)-\mu_{b,jk}(\psi)|+|\mu_{k}(\phi)-\mu_{k}(\psi)|\right).
	\end{aligned}
\end{equation*}
Thus, 
\begin{equation}\label{pr2.3:20}
\begin{aligned}
	|\mu_{k}(\phi)-\mu_{k}(\psi)|\lesssim &(\sqrt{b}+\alpha b^{\frac{3}{2}})\left(\sum_{j=0}^{k-1}|\mu_{b,jk}(\phi)-\mu_{b,jk}(\psi)|+|\mu_{k}(\phi)-\mu_{k}(\psi)|\right)\\
    &+b|\sqrt{b}\partial_{z}(\phi-\psi)(\sqrt{b})|.
\end{aligned}
\end{equation}
The same computation as before yields that  for any $0\leq{i}\leq{k-1} $,
\begin{equation}\label{pr2.3:21}
\begin{aligned}
	|\mu_{b,ik}(\phi)-\mu_{b,ik}(\psi)|\lesssim & (\sqrt{b}+\alpha b^{\frac{3}{2}})\left(\sum_{j=0}^{k-1}|\mu_{b,jk}(\phi)-\mu_{b,jk}(\psi)|+|\mu_{k}(\phi)-\mu_{k}(\psi)|\right)\\
    &+b|\sqrt{b}\partial_{z}(\phi-\psi)(\sqrt{b})|.
    \end{aligned}
\end{equation}
Combining (\ref{pr2.3:20}) and (\ref{pr2.3:21}), we get
\begin{equation*}
	\begin{aligned}
	&\sum_{j=0}^{k-1}|\mu_{b,ik}(\phi)-\mu_{b,ik}(\psi)|+|\mu_{k}(\phi)-\mu_{k}(\psi)|\\
    \lesssim& (\sqrt{b}+\alpha b^{\frac{3}{2}})\left(\sum_{j=0}^{k-1}|\mu_{b,jk}(\phi)-\mu_{b,jk}(\psi)|+|\mu_{k}(\phi)-\mu_{k}(\psi)|\right)\\
	&+b|\sqrt{b}\partial_{z}(\phi-\psi)(\sqrt{b})|.
\end{aligned}
\end{equation*}
Taking $ b>0 $ small enough yields 
\begin{equation}\label{pr2.3:22}
	\sum_{j=0}^{k-1}|\mu_{b,ik}(\phi)-\mu_{b,ik}(\psi)|+|\mu_{k}(\phi)-\mu_{k}(\psi)|\lesssim b|\sqrt{b}\partial_{z}(\phi-\psi)(\sqrt{b})|.
\end{equation}
Then by (\ref{pr2.3:16}) and (\ref{pr2.3:22}), we have
	\begin{align*}
		&\Arrowvert f(\phi)-f(\psi)\Arrowvert_{L^{2}_{\rho,b}}\\
        \leq&\sum_{j=0}^{k-1}|\mu_{b,jk}(\phi)-\mu_{b,jk}(\psi)|\Arrowvert 2(j-k)\mathcal{T}_{b}P_{j}+Q_{j}\Arrowvert_{L^{2}_{\rho,b}}+\Arrowvert\mu_{k}(\phi)\phi-\mu_{k}(\psi)\psi\Arrowvert_{L^{2}_{\rho,b}}\\
		&+|\mu_{k}(\phi)-\mu_{k}(\psi)|\Arrowvert P_{k}\mathcal{T}_{b}\Arrowvert_{L^{2}_{\rho,b}}+\sum_{j=0}^{k-1}|\mu_{k}(\phi)\mu_{b,jk}(\phi)-\mu_{k}(\psi)\mu_{b,jk}(\psi)|\Arrowvert P_{j}\mathcal{T}_{b}\Arrowvert_{L^{2}_{\rho,b}}\\
		\lesssim &\frac{1}{\sqrt{b}}\cdot {b}|\sqrt{b}\partial_{z}(\phi-\psi)(\sqrt{b})|+|\mu_{k}(\phi)-\mu_{k}(\psi)|\Arrowvert\phi\Arrowvert_{L^{2}_{\rho,b}}+|\mu_{k}(\psi)|\Arrowvert\phi-\psi\Arrowvert_{L^{2}_{\rho,b}}\\
		&+\frac{1}{\sqrt{b}}(\sqrt{b}+\alpha{b})\left(\sum_{j=0}^{k-1}|\mu_{b,jk}(\phi)-\mu_{b,jk}(\psi)|+|\mu_{k}(\phi)-\mu_{k}(\psi)|\right)\\
		\lesssim &\sqrt{b}|\sqrt{b}\partial_{z}(\phi-\psi)(\sqrt{b})|+\alpha b|\sqrt{b}\partial_{z}(\phi-\psi)(\sqrt{b})|+(\sqrt{b}+b|\sqrt{b}\partial_{z}\psi(\sqrt{b})|)\Arrowvert\phi-\psi\Arrowvert_{L^{2}_{\rho,b}}\\
		&+(1+\alpha\sqrt{b})b|\sqrt{b}\partial_{z}(\phi-\psi)(\sqrt{b})|\\
		\lesssim &(\sqrt{b}+\alpha{b})\Arrowvert \phi-\psi\Arrowvert_{X}.
	\end{align*}
Hence, the near inversion of $ H_{b}-2k $ gives
\begin{equation*}
	\begin{aligned}
		\Arrowvert H^{-1}_{b,k}\circ{F}(\phi)-H^{-1}_{b,k}\circ{F}(\psi)\Arrowvert_{X}&\lesssim \Arrowvert F(\phi)-F(\psi)\Arrowvert_{L^{2}_{\rho,b}}\\
		&\lesssim\Arrowvert f(\phi)-f(\psi)\Arrowvert_{L^{2}_{\rho,b}}\\
		&\lesssim(\sqrt{b}+\alpha{b})\Arrowvert \phi-\psi\Arrowvert_{X}.
	\end{aligned}
\end{equation*}
Thus for $ b>0 $ small enough, $ H_{b,k}^{-1}\circ{F} $ is a strict contraction. As a result, there exists a unique solution satisfying $ H_{b,k}^{-1}\circ{F}(u)=u $ in $ B_{\alpha} $. The Fr\"{e}chet differentiability of $ \psi_{b,k} $ with respect to $ b>0 $ then follows in the same manner.

\textbf{ Step 2. Spectral gap estimate.} We aim to prove
\begin{equation}\label{pr2.3:24}
\Arrowvert\partial_{z}{u}\Arrowvert^{2}_{L^{2}_{\rho,b}}\geq(2k+2)\Arrowvert{u}\Arrowvert^{2}
_{L^{2}_{\rho,b}}+O(\sqrt{b}\Arrowvert{u}\Arrowvert^{2}_{L^{2}_{\rho,b}}).
\end{equation}
for all $u\in{H}^{1}_{\rho,b} $ satisfying $ \langle u,\psi_{b,j}\rangle_{b}=0$,  $\forall\ 0\leq{j}\leq{k} $. Indeed, let
\begin{equation*}
	v=u\textbf{1}_{z\geq\sqrt{b}}-\sum_{j=0}^{k}\frac{\langle u,P_{j}\rangle_{b}}{\langle P_{j},P_{j}\rangle_{0}}P_{j}\in H^{1}_{\rho,0},
\end{equation*}
by spectral estimate on $ [0,+\infty) $ and the definition of $ v $, we deduce that
\begin{equation}\label{spe1}
	\Arrowvert\partial_{z}v\Arrowvert^{2}_{L^{2}_{\rho,0}}\geq(2k+2)\Arrowvert{v}\Arrowvert^{2}_{L^{2}_{\rho,0}}.
\end{equation}
Note also that
\begin{equation*}
	\begin{aligned}
		0=\langle u,\psi_{b,j}\rangle_{b}&=\bigg\langle u,P_{j}\mathcal{T}_{b}+\sum_{i=0}^{j-1}\mu_{ij}P_{i}\mathcal{T}_{b}+\tilde{\psi}_{b,j}\bigg\rangle_{b}\\
		&=-\frac{1}{\sqrt{b}}\langle u,P_{j}\rangle_{b}+\bigg\langle u,\frac{P_{j}}{z}\bigg\rangle_{b}+\bigg\langle \sum_{i=0}^{j-1}\mu_{b,ij}P_{i}\mathcal{T}_{b}+\tilde{\psi}_{b,j},u\bigg\rangle_{b},
	\end{aligned}
\end{equation*}
which yields
\begin{equation}\label{pr2.3:23}
\begin{aligned}
	|\langle u,P_{j}\rangle_{b}|\lesssim&\sqrt{b}[\Arrowvert u\Arrowvert_{L^{2}_{\rho,b}}+(|\mu_{b,ij}|\Arrowvert P_{i}\mathcal{T}_{b}\Arrowvert_{L^{2}_{\rho,b}}\Arrowvert{u}\Arrowvert_{L^{2}_{\rho,b}}\\
 &+\Arrowvert\tilde{\psi}_{b,j}\Arrowvert_{L^{2}_{\rho,b}}\Arrowvert{u}\Arrowvert_{L^{2}_{\rho,b}})]\lesssim\sqrt{b}\Arrowvert{u}\Arrowvert_{L^{2}_{\rho,b}}.
 \end{aligned}
\end{equation}
And since
\begin{equation*}
\begin{aligned}
&\Arrowvert\partial_{z}v\Arrowvert^{2}_{L^{2}_{\rho,0}}\\
=&\int_{0}^{+\infty}\bigg|\partial_{z}u\textbf{1}_{z\geq\sqrt{b}}-\sum_{j=0}^{k}\frac{\langle u,P_{j}\rangle_{b}}{\langle P_{j},P_{j}\rangle_{0}}\partial_{z}P_{j}\bigg|^{2}\rho{z}^{2}dz\\
=&\int_{\sqrt{b}}^{+\infty}|\partial_{z}u|^{2}\rho{z}^{2}+\sum_{j=0}^{k}\int_{0}^{+\infty}\bigg| \frac{\langle u,P_{j}\rangle_{b}}{\langle P_{j},P_{j}\rangle_{0}}\bigg|^{2}|\partial_{z}P_{j}|^{2}\rho{z}^{2}dz\\
&+\sum_{j\neq{i}}\frac{\langle u,P_{i}\rangle_{b}\langle u,P_{j}\rangle_{b}}{\langle P_{i},P_{i}\rangle_{0}\langle P_{j},P_{j}\rangle_{0}}\int_{0}^{+\infty}\partial_{z}P_{i}\partial_{z}P_{j}\rho{z}^{2}dz\\
&-\sum_{j=0}^{k}\frac{\langle u,P_{j}\rangle_{b}}{\langle P_{j},P_{j}\rangle_{0}}\int_{\sqrt{b}}^{+\infty}\partial_{z}u\partial_{z}P_{j}\rho{z}^{2}dz.
\end{aligned}
\end{equation*}

From (\ref{pr2.3:23}), we get that except the first term, all the other terms can be bounded by $ \sqrt{b}\Arrowvert{u}\Arrowvert^{2}_{L^{2}_{\rho,b}} $.
Hence,
\begin{equation*}
\Arrowvert\partial_{z}v\Arrowvert^{2}_{L^{2}_{\rho,0}}=\Arrowvert\partial_{z}u\Arrowvert^{2}_{L^{2}_{\rho,b}}+O(\sqrt{b}\Arrowvert{u}\Arrowvert^{2}_{L^{2}_{\rho,b}}).
\end{equation*}
Similarly, 
\begin{equation*}
\Arrowvert{v}\Arrowvert^{2}_{L^{2}_{\rho,0}}=\Arrowvert{u}\Arrowvert^{2}_{L^{2}_{\rho,b}}+O(\sqrt{b}\Arrowvert{u}\Arrowvert^{2}_{L^{2}_{\rho,b}}).
\end{equation*}

Injecting into \eqref{spe1} yields \eqref{pr2.3:24}.

{\it\bf
    Step 3. Estimate of $ \partial_{b}\lambda_{b,k} $.}  Differentiating the  both sides of (\ref{pr2.3:1}) with respect to $ b $ yields
\begin{equation}\label{pr2.3:25}
	H_{b}\partial_{b}\psi_{b,k}=\partial_{b}\lambda_{b,k}\psi_{b,k}+\lambda_{b,k}\partial_{b}\psi_{b,k}.
\end{equation}
Next, taking inner product with $ \psi_{b,k} $ on both sides of (\ref{pr2.3:25}), we obtain
\begin{equation}\label{pr2.3:26}
\partial_{b}\lambda_{b,k}\Arrowvert\psi_{b,k}\Arrowvert^{2}_{L^{2}_{\rho,b}}+\lambda_{b,k}\langle\partial_{b}\psi_{b,k},\psi_{b,k}\rangle_{b}=\langle H_{b}\partial_{b}\psi_{b,k},\psi_{b,k}\rangle_{b}.
\end{equation}
For the right hand-side term,  integration by parts and direct computation give
\begin{equation*}
	\begin{aligned}
		\langle H_{b}\partial_{b}\psi_{b,k},\psi_{b,k}\rangle_{b}&=-\int_{\sqrt{b}}^{+\infty}\frac{1}{\rho{z}^{2}}\partial_{z}(\rho{z}^{2}\partial_{z}(\partial_{b}\psi_{b,k}))\psi_{b,k}\rho{z}^{2}dz\\
		&=-\rho{z}^{2}\partial_{z}(\partial_{b}\psi_{b,k})\psi_{b,k}\bigg|^{+\infty}_{\sqrt{b}}+\int_{\sqrt{b}}^{+\infty}\rho{z}^{2}\partial_{z}\partial_{b}\psi_{b,k}\partial_{z}\psi_{b,k}dz\\
		&=\int_{\sqrt{b}}^{+\infty}\rho{z}^{2}\partial_{z}\psi_{b,k}d(\partial_{b}\psi_{b,k})\\
		&=\rho{z}^{2}\partial_{z}\psi_{b,k}\partial_{b}\psi_{b,k}\bigg|_{\sqrt{b}}^{+\infty}-\int_{\sqrt{b}}^{+\infty}\frac{1}{\rho{z}^{2}}\partial_{z}(\rho{z}^{2}\partial_{z}\psi_{b,k})\partial_{b}\psi_{b,k}	\rho{z}^{2}dz\\
		&=-e^{-\frac{b}{2}}b\partial_{z}\psi_{b,k}(\sqrt{b})\partial_{b}\psi_{b,k}(\sqrt{b})+\langle H_{b}\psi_{b,k},\partial_{b}\psi_{b,k}\rangle_{b}\\
		&=-e^{-\frac{b}{2}}b\partial_{z}\psi_{b,k}(\sqrt{b})\partial_{b}\psi_{b,k}(\sqrt{b})+\lambda_{b,k}\langle \psi_{b,k},\partial_{b}\psi_{b,k}\rangle_{b}.
	\end{aligned}
\end{equation*}
Then inserting the above formula into (\ref{pr2.3:26}) yields
\begin{equation}\label{pr2.3:27}
	\partial_{b}\lambda_{b,k}=-\frac{e^{-\frac{b}{2}}b\partial_{b}\psi_{b,k}(\sqrt{b})\partial_{z}\psi_{b,k}(\sqrt{b})}{\Arrowvert\psi_{b,k}\Arrowvert^{2}_{L^{2}_{\rho,b}}}.
\end{equation}
Differentiating the boundary condition $ \psi_{b,k}(\sqrt{b})=0 $ with respect to $ b $ on both sides, we have
\begin{equation}\label{pr2.3:28}
	\partial_{b}\psi_{b,k}(\sqrt{b})=-\frac{1}{2\sqrt{b}}\partial_{z}\psi_{b,k}(\sqrt{b}),
\end{equation}
 which combines with (\ref{pr2.3:27}) gives
\begin{equation}\label{pr2.3:29}
	\partial_{b}\lambda_{b,k}=\frac{\sqrt{b}e^{-{\frac{b}{2}}}}{2}\frac{|\partial_{z}\psi_{b,k}(\sqrt{b})|^{2}}{\Arrowvert\psi_{b,k}\Arrowvert^{2}_{L^{2}_{\rho,b}}}.
\end{equation}
By the expansion of $ \psi_{b,k} $ (see (\ref{pr2.3:3})) and estimates (\ref{pr2.3:5}), one sees that
\begin{equation*}
	|\partial_{z}\psi_{b,k}(\sqrt{b})|\leq|\partial_{z}\tilde{\psi}_{b,k}(\sqrt{b})|+|P_{k}(\sqrt{ b})|\cdot\frac{1}{b}+\sum_{j=0}^{k-1}|\mu_{b,jk}P_{j}(\sqrt{b})|\cdot\frac{1}{b}\lesssim\frac{1}{b},
\end{equation*}
and
\begin{equation*}
	\Arrowvert\psi_{b,k}\Arrowvert^{2}_{L^{2}_{\rho,b}}=\Arrowvert T_{b,k}(z)+\tilde{\psi}_{b,k}(z)
\Arrowvert^{2}_{L^{2}_{\rho,b}}\gtrsim \frac{1}{b}.
\end{equation*}
Thus by (\ref{pr2.3:29}), we obtain
\begin{equation*}
	|\partial_{b}\lambda_{b,k}|\lesssim \sqrt{b}\frac{\frac{1}{b^{2}}}{\frac{1}{b}}\lesssim\frac{1}{\sqrt{b}}.
\end{equation*}

\textbf{ Step 4. Estimate for $ |\partial_{b}\mu_{jk}|,\ \Arrowvert\partial_{b}\tilde{\psi}_{b,k}\Arrowvert_{L^{2}_{\rho,b}},\ \Arrowvert\partial_{z}\partial_{b}\tilde{\psi}_{b,k}\Arrowvert_{L^{2}_{\rho,b}},\ \sqrt{b}|\partial_{z}\partial_{b}\tilde{\psi}(\sqrt{b})| $.}
\hspace*{\fill}\\
Since $ \tilde{\psi}_{b,k}(\sqrt{b})=0 $ and $ \langle \tilde{\psi}_{b,k},P_{j}\rangle_{b}=0 $, for any $0\leq{j}\leq{k} $, we know
\begin{equation}\label{pr2.3:30}
	\langle \partial_{b}\tilde{\psi}_{b,k},P_{j}\rangle_{b}=0,\quad \forall\ 0\leq{j}\leq{k}.
\end{equation}
Consider the function $ u=\partial_{b}\tilde{\psi}_{b,k}\textbf{1}_{z\geq\sqrt{b}}+\partial_{b}\tilde{\psi}_{b,k}(\sqrt{b})\textbf{1}_{0\leq{z}<\sqrt{b}} $, and  set
\begin{equation}\label{pr2.3:31}
	v=u-\sum_{j=0}^{k}\frac{\langle u,P_{j}\rangle_{0}}{\langle P_{j},P_{j}\rangle_{0}}P_{j}.
\end{equation}
By spectral gap for harmonic oscillator (\ref{eq:12}), 
\begin{equation*}
		\Arrowvert\partial_{z}v\Arrowvert^{2}_{L^{2}_{\rho,0}}\geq(2k+2)\Arrowvert{v}\Arrowvert^{2}_{L^{2}_{\rho,0}},
	\end{equation*}
then inserting (\ref{pr2.3:31}) into the above estimate yields
\begin{equation}\label{pr2.3:32}
	\begin{aligned}
		\bigg\lVert& \partial_{z}\partial_{b}\tilde{\psi}_{b,k}\textbf{1}_{z\geq{\sqrt{b}}}-\sum_{j=0}^{k}\frac{\langle u,P_{j}\rangle_{0}}{\langle P_{j},P_{j}\rangle_{0}}\partial_{z}P_{j}\bigg\rVert^{2}_{L^{2}_{\rho,0}}\\
		\geq&(2k+2) \bigg\lVert \partial_{b}\tilde{\psi}_{b,k}\textbf{1}_{z\geq{\sqrt{b}}}+\partial_{b}\tilde{\psi}_{b,k}(\sqrt{b})\textbf{1}_{0\leq{z}<\sqrt{b}}-\sum_{j=0}^{k}\frac{\langle u,P_{j}\rangle_{0}}{\langle P_{j},P_{j}\rangle_{0}}P_{j}\bigg\rVert^{2}_{L^{2}_{\rho,0}}.
	\end{aligned}
\end{equation}
Note that by (\ref{pr2.3:30}),
\begin{equation*}
	\begin{aligned}
		\bigg|\frac{\langle u,P_{j}\rangle_{0}}{\langle P_{j},P_{j}\rangle_{0}}\bigg|&=|\langle \partial_{b}\tilde{\psi}_{b,k}\textbf{1}_{z\geq\sqrt{b}}+\partial_{b}\tilde{\psi}_{b,k}(\sqrt{b})\textbf{1}_{0\leq{z}\leq\sqrt{b}},P_{j}\rangle_{0}|\\
		&=|\langle \partial_{b}\tilde{\psi}_{b,k},P_{j}\rangle_{b}+\langle \partial_{b}\tilde{\psi}_{b,k}(\sqrt{b})\textbf{1}_{0\leq{z}<\sqrt{b}},P_{j}\rangle_{0}|\\
		&\lesssim b^{\frac{3}{2}}|\partial_{b}\tilde{\psi}_{b,k}(\sqrt{b})|.
	\end{aligned}
\end{equation*}
Thus we expand (\ref{pr2.3:32}) to get
\begin{equation}\label{pr2.3:33}
	(1+O(b^{\frac{3}{2}}))\Arrowvert\partial_{z}\partial_{b}\tilde{\psi}_{b,k}\Arrowvert^{2}_{L^{2}_{\rho,b}}\geq(2k+2)\Arrowvert\partial_{b}\tilde{\psi}_{b,k}\Arrowvert^{2}_{L^{2}_{\rho,b}}-O(b^{\frac{3}{2}})|\partial_{b}\tilde{\psi}_{b,k}(\sqrt{b})|^{2}.
\end{equation}
From (\ref{pr2.3:25}) and that $ \psi_{b,k}=T_{b,k}+\tilde{\psi}_{b,k} $,  $ \partial_{b}\tilde{\psi}_{b,k} $ satisfies
\begin{equation}\label{pr2.3:34}
	H_{b}\partial_{b}\tilde{\psi}_{b,k}=\lambda_{b,k}\partial_{b}\tilde{\psi}_{b,k}+F_{k},
\end{equation}
where
\begin{equation*}
	F_{k}=\partial_{b}\lambda_{b,k}(T_{b,k}+\tilde{\psi}_{b,k})+\lambda_{b,k}\partial_{b}T_{b,k}-H_{b}(\partial_{b}T_{b,k}).
\end{equation*}
Next, taking inner product with $ \partial_{b}\tilde{\psi}_{b,k} $ on both sides of (\ref{pr2.3:34}), we obtain
\begin{equation}\label{pr2.3:35}
	\langle H_{b}\partial_{b}\tilde{\psi}_{b,k},\partial_{b}\tilde{\psi}_{b,k}\rangle_{b}=\langle F_{k},\partial_{b}\tilde{\psi}_{b,k}\rangle_{b}+\lambda_{b,k}\Arrowvert\partial_{b}\tilde{\psi}_{b,k}\Arrowvert^{2}_{L^{2}_{\rho,b}}.
\end{equation}
Integration by parts for the left side yields
\begin{equation}\label{pr2.3:36}
	\begin{aligned}
		&\quad\langle  H_{b}\partial_{b}\tilde{\psi}_{b,k},\partial_{b}\tilde{\psi}_{b,k}\rangle_{b}\\
        &=\int_{\sqrt{b}}^{+\infty}\left(-\frac{1}{\rho{z}^{2}}\partial_{z}(\rho{z}^{2}\partial_{z}\partial_{b}\tilde{\psi}_{b,k})\right)\partial_{b}\tilde{\psi}_{b,k}\rho{z}^{2}dz\\
		&=-\rho{z}^{2}\partial_{z}\partial_{b}\tilde{\psi}_{b,k}\partial_{b}\tilde{\psi}_{b,k}\bigg|^{+\infty}_{\sqrt{b}}+\int_{\sqrt{b}}^{+\infty}(\partial_{z}\partial_{b}\tilde{\psi}_{b,k})^{2}\rho{z}^{2}dz\\
		&=e^{-\frac{b}{2}}b\partial_{z}\partial_{b}\tilde{\psi}_{b,k}(\sqrt{b})\partial_{b}\tilde{\psi}_{b,k}(\sqrt{b})+\Arrowvert\partial_{b}\partial_{z}\tilde{\psi}_{b,k}\Arrowvert^{2}_{L^{2}_{\rho,b}}.
	\end{aligned}
\end{equation}
Then injecting (\ref{pr2.3:36}) into (\ref{pr2.3:35}) and by Cauchy-Schwarz inequality, we have
\begin{equation}\label{pr2.3:37}
\begin{aligned}
	&e^{-\frac{b}{2}}b\partial_{z}\partial_{b}\tilde{\psi}_{b,k}(\sqrt{b})\partial_{b}\tilde{\psi}_{b,k}(\sqrt{b})+\Arrowvert\partial_{b}\partial_{z}\tilde{\psi}_{b,k}\Arrowvert^{2}_{L^{2}_{\rho,b}}\\
    \leq &\ C(\varepsilon)\Arrowvert F_{k}\Arrowvert^{2}_{L^{2}_{\rho,b}}+\varepsilon\Arrowvert\partial_{b}\tilde{\psi}_{b,k}\Arrowvert^{2}_{L^{2}_{\rho,b}}+\lambda_{b,k}\Arrowvert \partial_{b}\tilde{\psi}_{b,k}\Arrowvert^{2}_{L^{2}_{\rho,b}},
    \end{aligned}
\end{equation}
for some $ \varepsilon>0 $ small enough to be determined later. Note also that by (\ref{pr2.3:33})
\begin{equation*}
	\lambda_{b,k}\Arrowvert\partial_{b}\tilde{\psi}_{b,k}\Arrowvert^{2}_{L^{2}_{\rho,b}}\leq\frac{\lambda_{b,k}(1+O(b^{\frac{3}{2}}))}{2k+2}\Arrowvert\partial_{z}\partial_{b}\tilde{\psi}_{b,k}\Arrowvert^{2}_{L^{2}_{\rho,b}}+O(b^{\frac{3}{2}})|\partial_{b}\tilde{\psi}_{b,k}(\sqrt{b})|^{2},
\end{equation*}
where $ \lambda_{b,k}=2k+O(\sqrt{b}) $,  hence inserting the above inequality into (\ref{pr2.3:37}) gives 
\begin{equation}\label{pr2.3:38}
	\Arrowvert\partial_{z}\partial_{b}\tilde{\psi}_{b,k}\Arrowvert^{2}_{L^{2}_{\rho,b}}\lesssim\Arrowvert{F_{k}}\Arrowvert^{2}_{L^{2}_{\rho,b}}+b|\partial_{z}\partial_{b}\tilde{\psi}_{b,k}(\sqrt{b})|\cdot|\partial_{b}\tilde{\psi}_{b,k}(\sqrt{b})|+b^{\frac{3}{2}}|\partial_{b}\tilde{\psi}_{b,k}(\sqrt{b})|^{2}.
\end{equation}
From (\ref{pr2.3:5}) and (\ref{pr2.3:28}), we know
\begin{equation}\label{pr2.3:39}
	|\partial_{b}\tilde{\psi}_{b,k}(\sqrt{b})|\lesssim\frac{1}{\sqrt{b}}|\partial_{z}\tilde{\psi}_{b,k}(\sqrt{b})|\lesssim\frac{1}{b}.
\end{equation}
Note that by direct computation,
\begin{equation*}
	H_{b}\bigg(\frac{1}{z}\bigg)=-\frac{1}{z}.
\end{equation*}
Hence
\begin{equation*}
	\begin{aligned}
	\bigg\langle \partial_{b}\tilde{\psi}_{b,k},-\frac{1}{z}\bigg\rangle_{b}=&\bigg\langle \partial_{b}\tilde{\psi}_{b,k},H_{b}\bigg(\frac{1}{z}\bigg)\bigg\rangle_{b}=-\int_{\sqrt{b}}^{+\infty}\partial_{b}\tilde{\psi}_{b,k}\frac{1}{\rho{z}^{2}}\partial_{z}\bigg(\rho{z}^{2}\partial_{z}\bigg(\frac{1}{z}\bigg)\bigg)\rho{z}^{2}dz\\
	=&-\bigg(\partial_{b}\tilde{\psi}_{b,k}\rho{z}^{2}\partial_{z}\bigg(\frac{1}{z}\bigg)\bigg|^{+\infty}_{\sqrt{b}}-\int_{\sqrt{b}}^{+\infty}\partial_{z}\partial_{b}\tilde{\psi}_{b,k}\rho{z}^{2}\partial_{z}\bigg(\frac{1}{z}\bigg)dz\bigg)\\
	=&-e^{-\frac{b}{2}}\partial_{b}\tilde{\psi}_{b,k}(\sqrt{b})+\bigg(\frac{1}{z}(\partial_{z}\partial_{b}\tilde{\psi}_{b,k}\rho{z}^{2})\bigg|^{+\infty}_{\sqrt{b}}\\
 &-\int_{\sqrt{b}}^{+\infty}\partial_{z}(\partial_{b}\partial_{z}\tilde{\psi}_{b,k}\rho{z}^{2})\frac{1}{z}dz\bigg)\\
	=&-e^{-\frac{b}{2}}\partial_{b}\tilde{\psi}_{b,k}(\sqrt{b})-\sqrt{b}e^{-\frac{b}{2}}\partial_{z}\partial_{b}\tilde{\psi}_{b,k}+\bigg\langle H_{b}(\partial_{b}\tilde{\psi}_{b.k}),\frac{1}{z}\bigg\rangle_{b}.
		\end{aligned}
\end{equation*}
As a result, by (\ref{pr2.3:34}) and (\ref{pr2.3:39}), we have
\begin{equation}\label{pr2.3:40}
	\begin{aligned}
		\sqrt{b}|\partial_{z}\partial_{b}\tilde{\psi}_{b,k}(\sqrt{b})|&\lesssim|\partial_{b}\tilde{\psi}_{b,k}(\sqrt{b})|+\bigg|\bigg\langle \partial_{b}\tilde{\psi}_{b,k},\frac{1}{z}\bigg\rangle_{b}\bigg|+\bigg|\bigg\langle H_{b}(\partial_{b}\tilde{\psi}_{b,k}),\frac{1}{z}\bigg\rangle_{b}\bigg|\\
		&\lesssim\frac{1}{b}+\Arrowvert\partial_{b}\tilde{\psi}_{b,k}\Arrowvert_{L^{2}_{\rho,b}}+\bigg|\bigg\langle \lambda_{b,k}\partial_{b}\tilde{\psi}_{b,k}+F_{k},\frac{1}{z}\bigg\rangle_{b}\bigg|\\
		&\lesssim \frac{1}{b}+\Arrowvert\partial_{b}\tilde{\psi}_{b,k}\Arrowvert_{L^{2}_{\rho,b}}+\Arrowvert{F_{k}}\Arrowvert_{L^{2}_{\rho,b}}.
	\end{aligned}
\end{equation}
Injecting  (\ref{pr2.3:40}) into (\ref{pr2.3:38}) yields
\begin{equation*}
	\begin{aligned}
		\Arrowvert \partial_{z}\partial_{b}\tilde{\psi}_{b,k}\Arrowvert^{2}_{L^{2}_{\rho,b}}&\lesssim\Arrowvert{F}_{k}\Arrowvert^{2}_{L^{2}_{\rho,b}}+b|\partial_{z}\partial_{b}\tilde{\psi}_{b,k}(\sqrt{b})|\cdot|\partial_{b}\tilde{\psi}_{b,k}(\sqrt{b})|+b^{\frac{3}{2}}|\partial_{b}\tilde{\psi}_{b,k}(\sqrt{b})|^{2}\\
		&\lesssim\Arrowvert F_{k}\Arrowvert^{2}_{L^{2}_{\rho,b}}+b^{2}|\partial_{z}\partial_{b}\tilde{\psi}_{b,k}(\sqrt{b})|^{2}+|\partial_{b}\tilde{\psi}_{b,k}(\sqrt{b})|^{2}+b^{\frac{3}{2}}|\partial_{b}\tilde{\psi}_{b,k}(\sqrt{b})|^{2}\\
		&\lesssim\Arrowvert F_{k}\Arrowvert^{2}_{L^{2}_{\rho,b}}+b\cdot\bigg(\frac{1}{b}+\Arrowvert\partial_{b}\tilde{\psi}_{b,k}\Arrowvert_{L^{2}_{\rho,b}}+\Arrowvert{F_{k}}\Arrowvert_{L^{2}_{\rho,b}}\bigg)^{2}+\frac{1}{b^{2}}\\
		&\lesssim\Arrowvert{F}_{k}\Arrowvert^{2}_{L^{2}_{\rho,b}}+\frac{1}{b^{2}}+b\Arrowvert\partial_{z}\partial_{b}\tilde{\psi}_{b,k}\Arrowvert^{2}_{L^{2}_{\rho,b}},
	\end{aligned}
\end{equation*}
where in the last line we have used (\ref{pr2.3:33}).\\
Hence
\begin{equation}\label{pr2.3:41}
	\Arrowvert\partial_{z}\partial_{b}\tilde{\psi}_{b,k}\Arrowvert_{L^{2}_{\rho,b}}\lesssim\Arrowvert{F}_{k}\Arrowvert_{L^{2}_{\rho,b}}+\frac{1}{b}.
	\end{equation}
Applying $ \partial_{b} $ to (\ref{pr2.3:6}) gives
\begin{equation*}
	(H_{b}-2k)\partial_{b}T_{b,k}=\sum_{j=0}^{k-1}\partial_{b}\mu_{b,jk}[2(j-k)\mathcal{T}_{b}(z)P_{j}(z)+Q_{j}(z)]+\sum_{j=0}^{k-1}\mu_{b,jk}(j-k)\frac{1}{b^{\frac{3}{2}}}P_{j}(z),
\end{equation*}
moreover, since $\lambda_{b,k}=2k+O(\sqrt{b})  $, we have
\begin{equation*}
\begin{aligned}
	(H_{b}-\lambda_{b,k})(\partial_{b}T_{b,k})=&O(\sqrt{b})\partial_{b}T_{b,k}+\sum_{j=0}^{k-1}\partial_{b}\mu_{b,jk}[2(j-k)\mathcal{T}_{b}(z)P_{j}(z)\\
 &+Q_{j}(z)]+\sum_{j=0}^{k-1}\mu_{b,jk}(j-k)\frac{1}{b^{\frac{3}{2}}}P_{j}(z).
 \end{aligned}
\end{equation*}
By the definition of $ T_{b,k} $, we  directly compute that
\begin{equation*}
	\partial_{b}T_{b,k}=\frac{1}{2}b^{-\frac{3}{2}}P_{k}+\frac{1}{2}b^{-\frac{3}{2}}\sum_{j=0}^{k-1}\mu_{b,jk}P_{j}+\sum_{j=0}^{k-1}\partial_{b}\mu_{b,jk}P_{j}\mathcal{T}_{b}.
\end{equation*}
As a result, 
\begin{equation*}
	\begin{aligned}
		\Arrowvert H_{b}(\partial_{b}T_{b,k})-\lambda_{b,k}\partial_{b}T_{b,k}\Arrowvert_{L^{2}_{\rho,b}}\lesssim&\sqrt{b}\Arrowvert\partial_{b}T_{b,k}\Arrowvert_{L^{2}_{\rho,b}}+\frac{1}{\sqrt{b}}\sum_{j=0}^{k-1}|\partial_{b}\mu_{b,jk}|+\frac{1}{b^{\frac{3}{2}}}\sqrt{b}\\
		\lesssim&\sqrt{b}\bigg(b^{-\frac{3}{2}}+b^{-1}+\frac{1}{\sqrt{b}}\bigg(\sum_{j=0}^{k-1}|\partial_{b}\mu_{b,jk}|\bigg)\bigg)\\
  &+\frac{1}{\sqrt{b}}\sum_{j=0}^{k-1}|\partial_{b}\mu_{b,jk}|+\frac{1}{\sqrt{b}}\\
		\lesssim&\frac{1}{b}+\frac{1}{\sqrt{b}}\bigg(\sum_{j=0}^{k-1}|\partial_{b}\mu_{b,jk}|\bigg),
	\end{aligned}
\end{equation*}
and then
\begin{equation}\label{pr2.3:42}
	\begin{aligned}
		\Arrowvert F_{k}\Arrowvert_{L^{2}_{\rho,b}}&=\Arrowvert\partial_{b}\lambda_{b,k}(T_{b,k}+\tilde{\psi}_{b,k})+\lambda_{b,k}\partial_{b}T_{b,k}-H_{b}(\partial_{b}T_{b,k})\Arrowvert_{L^{2}_{\rho,b}}\\
		&\leq |\partial_{b}\lambda_{b,k}|(\Arrowvert T_{b,k}\Arrowvert_{L^{2}_{\rho,b}}+\Arrowvert\tilde{\psi}_{b,k}\Arrowvert_{L^{2}_{\rho,b}})+\Arrowvert\lambda_{b,k}\partial_{b}T_{b,k}-H_{b}(\partial_{b}T_{b,k})\Arrowvert_{L^{2}_{\rho,b}}\\
	&\lesssim\frac{1}{b}+\frac{1}{\sqrt{b}}\bigg(\sum_{j=0}^{k-1}|\partial_{b}\mu_{b,jk}|\bigg).
	\end{aligned}
\end{equation}
Injecting (\ref{pr2.3:42}) into (\ref{pr2.3:33}), (\ref{pr2.3:40}) and (\ref{pr2.3:41}) yields
\begin{equation}\label{pr2.3:43}
	\begin{aligned}
		\Arrowvert \partial_{z}\partial_{b}\tilde{\psi}_{b,k}\Arrowvert_{L^{2}_{\rho,b}}+\Arrowvert\partial_{b}\tilde{\psi}_{b,k}\Arrowvert_{L^{2}_{\rho,b}}+\sqrt{b}|\partial_{z}\partial_{b}\tilde{\psi}_{b,k}(\sqrt{b})|&\lesssim\frac{1}{b}+\Arrowvert{F}_{k}\Arrowvert_{L^{2}_{\rho,b}}\\
		&\lesssim\frac{1}{b}+\frac{1}{\sqrt{b}}\sum_{j=0}^{k-1}|\partial_{b}\mu_{b,jk}|.
	\end{aligned}
\end{equation}
It remains to estimate $|\partial_{b}\mu_{b,jk}| $. 
Taking inner product again on both sides of (\ref{pr2.3:25}) with $ P_{j} $ $ (0\leq{j}\leq{k-1}) $ and combining with (\ref{pr2.3:30}), we have
\begin{equation}\label{pr2.3:44}
	\langle H_{b}\partial_{b}\psi_{b,k},P_{j}\rangle_{b}=\partial_{b}{\lambda}_{b,k}\langle \psi_{b,k},P_{j}\rangle_{b}+\lambda_{b,k}\langle \partial_{b}T_{b,k},P_{j}\rangle_{b}.
\end{equation}
For the left term, integration by parts gives
	\begin{align*}
	\langle H_{b}\partial_{b}\psi_{b,k},P_{j}\rangle_{b}=&\int_{\sqrt{b}}^{+\infty}-\frac{1}{\rho{z}^{2}}\partial_{z}(\rho{z}^{2}\partial_{z}\partial_{b}\psi_{b,k})P_{j}\rho{z}^{2}dz\\
	=&-\bigg(\rho{z}^{2}\partial_{z}\partial_{b}\psi_{b,k}P_{j}\bigg|^{+\infty}_{\sqrt{b}}-\int_{\sqrt{b}}^{+\infty}\partial_{z}\partial_{b}\psi_{b,k}\partial_{z}P_{j}\rho{z}^{2}dz\bigg)\\
	=&be^{-\frac{b}{2}}\partial_{z}\partial_{b}\psi_{b,k}(\sqrt{b})P_{j}(\sqrt{b})+\int_{\sqrt{b}}^{+\infty}\rho{z}^{2}\partial_{z}\partial_{b}\psi_{b,k}\partial_{z}P_{j}dz\\
	=&be^{-\frac{b}{2}}\partial_{z}\partial_{b}\psi_{b,k}(\sqrt{b})P_{j}(\sqrt{b})+\rho{z}^{2}\partial_{b}\psi_{b,k}\partial_{z}P_{j}\bigg|^{+\infty}_{\sqrt{b}}\\
 &-\int_{\sqrt{b}}^{+\infty}\partial_{b}\psi_{b,k}\partial_{z}(\rho{z}^{2}\partial_{z}P_{j})dz\\
	=&be^{-\frac{b}{2}}\partial_{z}\partial_{b}\psi_{b,k}(\sqrt{b})P_{j}(\sqrt{b})-be^{-\frac{b}{2}}\partial_{b}\psi_{b,k}(\sqrt{b})\partial_{z}P_{j}(\sqrt{b})\\
 &+\langle H_{b}P_{j},\partial_{b}\psi_{b,k}\rangle\\
	=&be^{-\frac{b}{2}}\partial_{z}\partial_{b}\psi_{b,k}(\sqrt{b})P_{j}(\sqrt{b})-be^{-\frac{b}{2}}\partial_{b}\psi_{b,k}(\sqrt{b})\partial_{z}P_{j}(\sqrt{b})\\
 &+2j\langle P_{j},\partial_{b}T_{b,k}\rangle.
		\end{align*}
Then  injecting into (\ref{pr2.3:44}), we get for any ${0}\leq j\leq k-1 $,
\begin{equation}\label{pr2.3:45}
\begin{aligned}
	(2j-\lambda_{b,k})\langle\partial_{b}T_{b,k},P_{j}\rangle_{b}=&be^{-\frac{b}{2}}\partial_{b}\psi_{b,k}(\sqrt{b})\partial_{z}P_{j}(\sqrt{b})\\
 &-be^{-\frac{b}{2}}\partial_{z}\partial_{b}\psi_{b,k}(\sqrt{b})P_{j}(\sqrt{b})+\partial_{b}\lambda_{b,k}\langle T_{b,k},P_{j}\rangle_{b}.
 \end{aligned}
\end{equation}
For the first term of the right side, using the expression of $ \partial_{b}T_{b,k} $ and (\ref{pr2.3:5}), we get
\begin{equation*}
	\begin{aligned}
		|be^{-\frac{b}{2}}\partial_{b}\psi_{b,k}(\sqrt{b})\partial_{z}P_{j}(\sqrt{b})|&\lesssim b^{\frac{3}{2}}(|\partial_{b}\tilde{\psi}_{b,k}(\sqrt{b})|+|\partial_{b}T_{b,k}(\sqrt{b})|)\\
		&\lesssim b^{\frac{3}{2}}(b^{-\frac{1}{2}}+b^{-\frac{3}{2}})\\
		&\lesssim{1},
	\end{aligned}
\end{equation*}
for the second term of the right side, using the expression of $ \partial_{b}T_{b,k} $ again and (\ref{pr2.3:43}),
\begin{equation*}
	\begin{aligned}
	|be^{-\frac{b}{2}}\partial_{z}\partial_{b}\psi_{b,k}(\sqrt{b})P_{j}(\sqrt{b})|&\lesssim b|\partial_{z}\partial_{b}T_{b,k}(\sqrt{b})|+b|\partial_{z}\partial_{b}\tilde{\psi}_{b,k}(\sqrt{b})|\\
	&\lesssim1+\sum_{j=0}^{k-1}|\partial_{b}\mu_{b,jk}|+\frac{1}{\sqrt{b}}+\sum_{j=0}^{k-1}|\partial_{b}\mu_{b,jk}|\\
	&\lesssim\frac{1}{\sqrt{b}}+\sum_{j=0}^{k-1}|\partial_{b}\mu_{b,jk}|,
\end{aligned}
\end{equation*}
for the last term of the right side, using the expression of $ T_{b,k} $ and (\ref{pr2.3:4}),
\begin{equation*}
	|\partial_{b}\lambda_{b,k}\langle T_{b,k},P_{j}\rangle_{b}|\lesssim|\partial_{b}\lambda_{b,k}|\Arrowvert T_{b,k}\Arrowvert_{L^{2}_{\rho,b}}\lesssim \frac{1}{\sqrt{b}}\cdot\frac{1}{\sqrt{b}}=\frac{1}{b}.
\end{equation*}
For the left side of (\ref{pr2.3:45}), using the expression of $ \partial_{b}T_{b,k} $, (\ref{eq:13}) and (\ref{eq:15}), we have
\begin{equation*}
	\begin{aligned}
		\langle \partial_{b}T_{b,k},P_{j}\rangle_{b}=&\bigg\langle \frac{1}{2}b^{-\frac{3}{2}}P_{k}+\frac{1}{2}b^{-\frac{3}{2}}\sum_{i=0}^{k-1}\mu_{b,ik}P_{i}+\sum_{i=0}^{k-1}\partial_{b}\mu_{b,ik}P_{i}\mathcal{T}_{b},P_{j}\bigg\rangle_{b}\\
		=&\frac{1}{2}b^{-\frac{3}{2}}(M_{b,k})_{kj}+\frac{1}{2}b^{-\frac{3}{2}}\sum_{i=0}^{k-1}\mu_{b,ik}(M_{b,k})_{ij}+\sum_{i=0}^{k-1}\partial_{b}\mu_{b,ik}\bigg(\bigg\langle \frac{P_{i}}{z},P_{j}\bigg\rangle_{b}-\frac{1}{\sqrt{b}}(M_{b,k})_{ij}\bigg)\\
		=&\frac{1}{2}(b^{-\frac{3}{2}}\delta_{kj}+O(1))+\frac{1}{2}(b^{-\frac{3}{2}}\delta_{ij}+O(1))\mu_{b,ik}\\
  &-\partial_{b}\mu_{b,ik}\frac{1}{\sqrt{b}}(\delta_{ij}+O(b^{\frac{3}{2}}))+\sum_{i=0}^{k-1}\partial_{b}\mu_{b,ik}\bigg\langle \frac{P_{i}}{z},P_{j}\bigg\rangle_{b}.
	\end{aligned}
\end{equation*}
Injecting the above four estimates into (\ref{pr2.3:45}) yields 
\begin{equation*}
	\frac{1}{\sqrt{b}}|\partial_{b}\mu_{b,jk}|\lesssim\frac{1}{b}+\sum_{j=0}^{k-1}|\partial_{b}\mu_{b,jk}|,\quad  \forall\ 0\leq{j}\leq{k-1}.
\end{equation*}
Suming the above estimates for $ j=0,1,2,\cdot\cdot\cdot,k-1 $, and note that $ b>0 $ is small enough, we arrive at
\begin{equation}\label{pr2.3:46}
	\sum_{j=0}^{k-1}|\partial_{b}\mu_{b,jk}|\lesssim\frac{1}{\sqrt{b}}.
\end{equation}
Finally  inserting (\ref{pr2.3:46}) into (\ref{pr2.3:43}), we conclude that
\begin{equation*}
	\Arrowvert \partial_{z}\partial_{b}\tilde{\psi}_{b,k}\Arrowvert_{L^{2}_{\rho,b}}+\Arrowvert\partial_{b}\tilde{\psi}_{b,k}\Arrowvert_{L^{2}_{\rho,b}}+\sqrt{b}|\partial_{z}\partial_{b}\tilde{\psi}_{b,k}(\sqrt{b})|\lesssim \frac{1}{b}.
\end{equation*}
\vspace{0.3cm}
{\it\bf Step 5. Some further estimates:}
\hspace*{\fill}\\
\textbf{ Proof of (\ref{pr2.3.8}).} From (\ref{pr2.3:3}), (\ref{pr2.3:4}) and (\ref{pr2.3:5}) we estimate
\begin{equation*}
	\begin{aligned}
		\langle \psi_{b,k},\psi_{b,i}\rangle_{b}&=\bigg\langle P_{k}\bigg(\frac{1}{z}-\frac{1}{\sqrt{b}}\bigg)+\sum_{j=0}^{k-1}\mu_{b,jk}P_{j}\bigg(\frac{1}{z}-\frac{1}{\sqrt{b}}\bigg)+\tilde{\psi}_{b,k},\\
  &P_{i}\bigg(\frac{1}{z}-\frac{1}{\sqrt{b}}\bigg)+\sum_{j=0}^{i-1}\mu_{b,ji}P_{j}\bigg(\frac{1}{z}-\frac{1}{\sqrt{b}}\bigg)+\tilde{\psi}_{b,i}\bigg\rangle_{b}\\
		&=\frac{1}{b}(\langle P_{k},P_{i}\rangle_{0}+O(\sqrt{b}))=\frac{1}{b}(\delta_{ik}+O(\sqrt{b})),
	\end{aligned}
\end{equation*}
and  (\ref{pr2.3.8}) follows from (\ref{eq:10}).\\
\textbf{Proof of (\ref{pr2.3.9}).}   From (\ref{pr2.3:3}), (\ref{pr2.3:4}) , (\ref{pr2.3:5}) and (\ref{pr2.3.6}), we have
\begin{equation*}
	\begin{aligned}
		\langle b\partial_{b}\psi_{b,k},\psi_{b,i}\rangle_{b}=&\bigg\langle \frac{1}{2}b^{-\frac{1}{2}}P_{k}+\frac{1}{2}b^{-\frac{1}{2}}\sum_{j=0}^{k-1}\mu_{b,jk}P_{j}+b\sum_{j=0}^{k-1}\partial_{b}\mu_{b,jk}P_{j}\bigg(\frac{1}{z}-\frac{1}{\sqrt{b}}\bigg)\\
  &+b\partial_{b}\tilde{\psi}_{b,k},
		P_{i}\bigg(\frac{1}{z}-\frac{1}{\sqrt{b}}\bigg)+\sum_{j=0}^{i-1}\mu_{b,ji}P_{j}\bigg(\frac{1}{z}-\frac{1}{\sqrt{b}}\bigg)+\tilde{\psi}_{b,i}\bigg\rangle_{b}\\
		&=-\frac{1}{2b}(\langle P_{i},P_{k}\rangle_{0}+O(\sqrt{b}))=-\frac{1}{2b}(\delta_{ik}+O(\sqrt{b})),
	\end{aligned}
\end{equation*}
which is (\ref{pr2.3.9}).\\
\textbf{Proof of (\ref{pr2.3.10}).}  Since 
\begin{equation*}
	\partial_{b}\psi_{b,k}=\frac{1}{2}b^{-\frac{3}{2}}P_{k}+\frac{1}{2}b^{-\frac{3}{2}}\sum_{j=0}^{k-1}\mu_{b,jk}P_{j}+\sum_{j=0}^{k-1}\partial_{b}\mu_{b,jk}P_{j}\mathcal{T}_{b}+\partial_{b}\tilde{\psi}_{b,k}.
\end{equation*}
Hence 
\begin{equation}\label{pr2.3:47}
	\begin{aligned}
		\bigg\lVert b^{\frac{3}{2}}\partial_{b}\psi_{b,k}-\frac{1}{2}P_{k}\bigg\rVert_{L^{2}_{\rho,b}}&=\bigg\lVert \frac{1}{2}\sum_{j=0}^{k-1}\mu_{b,jk}P_{j}+b^{\frac{3}{2}}\sum_{j=0}^{k-1}\partial_{b}\mu_{b,jk}P_{j}\mathcal{T}_{b}+b^{\frac{3}{2}}\partial_{b}\tilde{\psi}_{b,k}\bigg\rVert_{L^{2}_{\rho,b}}\\
		&\lesssim \sqrt{b}+b^{\frac{3}{2}}\cdot\frac{1}{\sqrt{b}}\cdot\frac{1}{\sqrt{b}}+b^{\frac{3}{2}}\cdot\frac{1}{b}\lesssim\sqrt{b},
	\end{aligned}
\end{equation}
and also by the expression of $ \psi_{b,k} $, we obtain
\begin{equation}\label{pr2.3:48}
	\Arrowvert P_{k}+\sqrt{b}\psi_{b,k}\Arrowvert_{L^{2}_{\rho,b}}=\sqrt{b}\bigg\lVert\frac{P_{k}}{z}+\sum_{j=0}^{k-1}\mu_{b,jk}P_{j}\bigg(\frac{1}{z}-\frac{1}{\sqrt{b}}\bigg)+\tilde{\psi}_{b,k}\bigg\rVert_{L^{2}_{\rho,b}}\lesssim \sqrt{b}.
\end{equation}
Combining (\ref{pr2.3:47}) and (\ref{pr2.3:48}), we arrive at
\begin{equation}\label{pr2.3:49}
	\Arrowvert 2b\partial_{b}\psi_{b,k}+\psi_{b,k}\Arrowvert_{L^{2}_{\rho,b}}\lesssim{1}.
\end{equation}
Then from (\ref{pr2.3:1}), (\ref{pr2.3:2}), (\ref{pr2.3:25}) and (\ref{pr2.3:49}), direct computation yields
\begin{equation*}
	\begin{aligned}
		\Arrowvert H_{b}(2b\partial_{b}\psi_{b,k}+\psi_{b,k})\Arrowvert_{L^{2}_{\rho,b}}&=\Arrowvert 2b\partial_{b}\lambda_{b,k}\psi_{b,k}+2b\lambda_{b,k}\partial_{b}\psi_{b,k}+\lambda_{b,k}\psi_{b,k}\Arrowvert_{L^{2}_{\rho,b}}\\
	&\lesssim b|\partial_{b}\lambda_{b,k}|\Arrowvert\psi_{b,k}\Arrowvert_{L^{2}_{\rho,b}}+	\Arrowvert 2b\partial_{b}\psi_{b,k}+\psi_{b,k}\Arrowvert_{L^{2}_{\rho,b}}\\
	&\lesssim b\cdot\frac{1}{\sqrt{b}}\cdot\frac{1}{\sqrt{b}}+1\lesssim{1},	
	\end{aligned}
\end{equation*}
which is (\ref{pr2.3.10}).\\
\textbf{Proof of (\ref{pr2.3.11}).}   Note firstly that
\begin{equation}\label{pr2.3:50}
	\Lambda\psi_{b,k}=\Lambda{P_{k}}\bigg(\frac{1}{z}-\frac{1}{\sqrt{b}}\bigg)-\frac{P_{k}}{z}+\sum_{j=0}^{k-1}\mu_{b,jk}\Lambda P_{j}\bigg(\frac{1}{z}-\frac{1}{\sqrt{b}}\bigg)-\sum_{j=0}^{k-1}\mu_{b,jk}\frac{P_{j}}{z}+\Lambda\tilde{\psi}_{b,k},
\end{equation}
then by the definition of $ P_{k} $ and combining with (\ref{pr2.3:4}) and (\ref{pr2.3:5}), we have
\begin{equation*}
		\Arrowvert\Lambda\psi_{b,0}\Arrowvert_{L^{2}_{\rho,b}}\lesssim{1},\quad \Arrowvert\Lambda\psi_{b,k}\Arrowvert_{L^{2}_{\rho,b}}\lesssim\frac{1}{\sqrt{b}},\quad \forall\ k\geq{1},
	\end{equation*}
which is (\ref{pr2.3.11}).\\
\textbf{Proof of (\ref{pr2.3.12}).} It follows from (\ref{pr2.3.8}) and the Cauchy-Schwarz inequality.\\
\textbf {Proof of (\ref{pr2.3.13}).} Firstly recall (\ref{eq:11}) and combine with (\ref{pr2.3:50}),
\begin{equation*}
	\begin{aligned}
		\Lambda \psi_{b,k}=&\Lambda P_{k}\mathcal{T}_{b}-\frac{P_{k}}{z}+\sum_{j=0}^{k-1}\mu_{b,jk}\Lambda P_{j}\mathcal{T}_{b}-\sum_{j=0}^{k-1}\mu_{b,jk}\frac{P_{j}}{z}+\Lambda\tilde{\psi}_{b,k}\\
		=&2k\bigg(P_{k}-\frac{2k+1}{2k}\frac{A_{k-1}}{A_{k}}P_{k-1}\bigg)\mathcal{T}_{b}-\frac{P_{k}}{z}\\
  &+\sum_{j=0}^{k-1}\mu_{b,jk}\Lambda P_{j}\mathcal{T}_{b}-\sum_{j=0}^{k-1}\mu_{b,jk}\frac{P_{j}}{z}+\Lambda\tilde{\psi}_{b,k}\\
		=&2k\bigg(\psi_{b,k}-\frac{2k+1}{2k}\frac{A_{k-1}}{A_{k}}\psi_{b,k-1}\bigg)\\
  &-2k\bigg(\sum_{j=0}^{k-1}\mu_{b,jk}P_{j}\mathcal{T}_{b}+\tilde{\psi}_{b,k}\bigg)
  +(2k+1)\frac{A_{k-1}}{A_{k}}\\&\cdot\bigg(\sum_{j=0}^{k-2}\mu_{b,jk-1}P_{j}\mathcal{T}_{b}
  +\tilde{\psi}_{b,k-1}\bigg)
		-\frac{P_{k}}{z}\\
  &+\sum_{j=0}^{k-1}\mu_{b,jk}\Lambda P_{j}\mathcal{T}_{b}-\sum_{j=0}^{k-1}\mu_{b,jk}\frac{P_{j}}{z}+\Lambda\tilde{\psi}_{b,k}.
	\end{aligned}
\end{equation*}
Hence from (\ref{pr2.3:4}) and (\ref{pr2.3:5}) again, we obtain
\begin{equation*}
	\begin{aligned}
		\bigg\lVert\Lambda\psi_{b,k}&-2k\bigg(\psi_{b,k}-\frac{2k+1}{2k}\frac{A_{k-1}}{A_{k}}\psi_{b,k-1}\bigg)\bigg\rVert_{L^{2}_{\rho,b}}=\bigg\lVert-2k\bigg(\sum_{j=0}^{k-1}\mu_{b,jk}P_{j}\mathcal{T}_{b}+\tilde{\psi}_{b,k}\bigg)
		\\
  &+(2k+1)\frac{A_{k-1}}{A_{k}}\bigg(\sum_{j=0}^{k-2}\mu_{b,jk-1}P_{j}\mathcal{T}_{b}
		+\tilde{\psi}_{b,k-1}\bigg)
		-\frac{P_{k}}{z}+\sum_{j=0}^{k-1}\mu_{b,jk}\Lambda P_{j}\mathcal{T}_{b}-
  \\ &\sum_{j=0}^{k-1}\mu_{b,jk}\frac{P_{j}}{z}+\Lambda\tilde{\psi}_{b,k}\bigg\rVert_{L^{2}_{\rho,b}}
		\lesssim 1.
		\end{aligned}
\end{equation*}
For the second part, by direct computation,
\begin{equation}\label{pr2.3:53}
	[\Delta,\Lambda]=2\Delta,
\end{equation}
where $ [A,B]=AB-BA $ is the commutator.\\
As a result,
\begin{equation*}
	H_{b}(\Lambda\tilde{\psi}_{b,k})=(-\Delta+\Lambda)\Lambda\tilde{\psi}_{b,k}=(-\Lambda\Delta-[\Delta,\Lambda]+\Lambda\Lambda)\tilde{\psi}_{b,k}=\Lambda H_{b}(\tilde{\psi}_{b,k})-2\Delta\tilde{\psi}_{b,k}.
\end{equation*}
It follows that
\begin{equation*}
	\begin{aligned}
		H_{b}&\bigg(\Lambda\psi_{b,k}-2k\bigg(\psi_{b,k}-\frac{2k+1}{2k}\frac{A_{k-1}}{A_{k}}\psi_{b,k-1}\bigg)\bigg)=-2k\bigg(\sum_{j=0}^{k-1}\mu_{b,jk}H_{b}(P_{j}\mathcal{T}_{b})+H_{b}(\tilde{\psi}_{b,k})\bigg)\\
		&+(2k+1)\frac{A_{k-1}}{A_{k}}\bigg(\sum_{j=0}^{k-2}\mu_{b,jk-1}H_{b}(P_{j}\mathcal{T}_{b})+H_{b}(\tilde{\psi}_{b,k-1})\bigg)+H_{b}\bigg(-\frac{P_{k}}{z}\bigg)\\
		&+\sum_{j=0}^{k-1}\mu_{b,jk}H_{b}(\Lambda P_{j}\mathcal{T}_{b})+\sum_{j=0}^{k-1}\mu_{b,jk}H_{b}\bigg(-\frac{P_{j}}{z}\bigg)+\Lambda H_{b}(\tilde{\psi}_{b,k})-2\Delta\tilde{\psi}_{b,k}.
	\end{aligned}
\end{equation*}
Then  similar computations  as before and combining with (\ref{pr2.3:3}), (\ref{pr2.3:4}), (\ref{pr2.3:5}), (\ref{pr2.3:6}) and (\ref{pr2.3:8}) yields
\begin{equation*}
	\bigg\lVert 	H_{b}\bigg(\Lambda\psi_{b,k}-2k\bigg(\psi_{b,k}-\frac{2k+1}{2k}\frac{A_{k-1}}{A_{k}}\psi_{b,k-1}\bigg)\bigg)\bigg\rVert_{L^{2}_{\rho,b}}\lesssim{1},
\end{equation*}
which is (\ref{pr2.3.13}).

       	\end{proof}
       \subsection{Diagonalisation of $ \mathcal{H}_{b} $}
       We are now in a position to derive the eigenvalues of $ \mathcal{H}_{b} $ and the related estimates.
    \begin{proposition}[Diagonalisation of $ \mathcal{H}_{b} $]
    	Let $ K\in\NN $. Then for all $ 0<b<b^{*}(K) $ small enough, the renormalised operator
    	\begin{equation*}
    		\mathcal{H}_{b}=-\Delta+b\Lambda\quad with\ the\ boundary\ condition\quad  u(1)=0 
    	\end{equation*}
       	have a family of eigenstates $ \eta_{b,k} $ satisfying
       	\begin{equation}\label{pr2.4:1}
       		\mathcal{H}_{b}\eta_{b,k}=b\lambda_{b,k}\eta_{b,k},\quad \forall\ 0\leq{k}\leq{K},
       	\end{equation}
       with the following properties:
       \begin{itemize}
       	\item Expansion of the eigenfunctions:\\
       	\begin{equation}\label{pr2.4:2}
       		\left\{\begin{aligned}
&\eta_{b,k}=S_{b,k}(y)+\tilde{\eta}_{b,k}(y)\\
&S_{b,k }=\frac{1}{\sqrt{b}}\bigg(P_{k}(\sqrt{b}y)\bigg(\frac{1}{y}-1\bigg)+\sum_{j=0}^{k-1}\mu_{b,jk}P_{j}(\sqrt{b}y)\bigg(\frac{1}{y}-1\bigg)\bigg)\\
&\tilde{\eta}_{b,k}(y)=\tilde{\psi}_{b,k}(\sqrt{b}y),
       		\end{aligned}\right.
       	\end{equation}
with the following estimates: 
\begin{equation}\label{pr2.4:3}
\begin{aligned}
	b^{\frac{3}{4}}&\Arrowvert\tilde{\eta}_{b,k}\Arrowvert_{b}+\frac{1}{b^{\frac{1}{4}}}\Arrowvert \Delta\tilde{\eta}_{b,k}\Arrowvert_{b}+b^{\frac{1}{4}}\Arrowvert \partial_{y}\tilde{\eta}_{b,k}\Arrowvert_{b}+b^{\frac{3}{4}}\Arrowvert\Lambda\tilde{\eta}_{b,k}\Arrowvert_{b}\\
 &+b^{\frac{5}{4}}\Arrowvert y\tilde{\eta}_{b,k}\Arrowvert_{b}+b^{\frac{7}{4}}\Arrowvert y^{2}\tilde{\eta}_{b,k}\Arrowvert_{b}+|\partial_{y}\tilde{\eta}_{b,k}(1)|\lesssim {1}.
 \end{aligned}
	\end{equation}
\item Some further estimates:
\begin{equation}\label{pr2.4:4}
	(\eta_{b,k},\eta_{b,i})_{b}=\frac{1}{b^{\frac{5}{2}}}(\delta_{ik}+O(\sqrt{b})),
\end{equation}
\begin{equation}\label{pr2.4:5}
	(2b\partial_{b}\eta_{b,k}-\Lambda\eta_{b,k},\eta_{b,i})_{b}=-\frac{1}{b^{\frac{5}{2}}}(\delta_{ik}+O(\sqrt{b})),
	\end{equation}
\begin{equation}\label{pr2.4:6}
	\Arrowvert\Lambda\eta_{b,0}\Arrowvert_{b}\lesssim\frac{1}{b^{\frac{3}{4}}},\quad \Arrowvert\Lambda\eta_{b,k}\Arrowvert_{b}\lesssim\frac{1}{b^{\frac{5}{4}}}, \quad \forall k\geq{1},
\end{equation}
\begin{equation}\label{pr2.4:7}
	\Arrowvert y\eta_{b,k}\Arrowvert_{b}\lesssim\frac{1}{b^{\frac{7}{4}}},\quad \Arrowvert y^{2}\eta_{b,k}\Arrowvert_{b}\lesssim\frac{1}{b^{\frac{9}{4}}},
\end{equation}
\begin{equation}\label{pr2.4:8}
	\Arrowvert 2b\partial_{b}\eta_{b,k}-\Lambda\eta_{b,k}+\eta_{b,k}\Arrowvert_{b}+\frac{1}{b}\Arrowvert\mathcal{H}_{b}(2b\partial_{b}\eta_{b,k}-\Lambda\eta_{b,k}+\eta_{b,k})\Arrowvert_{b}\lesssim \frac{1}{b^{\frac{3}{4}}},
\end{equation}
\begin{equation}\label{pr2.4:9}
\begin{aligned}
	\bigg\lVert& \Lambda\eta_{b,k}-2k\bigg(\eta_{b,k}-\frac{2k+1}{2k}\frac{A_{k-1}}{A_{k}}\eta_{b,k-1}\bigg)\bigg\rVert_{b}\\
 &+\frac{1}{b}\bigg\lVert\mathcal{H}_{b}\bigg(\Lambda\eta_{b,k}-2k\bigg(\eta_{b,k}-\frac{2k+1}{2k}\frac{A_{k-1}}{A_{k}}\eta_{b,k-1}\bigg)\bigg)\bigg\rVert_{b}\lesssim \frac{1}{b^{\frac{3}{4}}}.
 \end{aligned}
\end{equation}
       \end{itemize}
       \end{proposition}
   \begin{proof}
   	Given $ v: \Omega(1)\rightarrow \RR $, set $ v(y)=u(\sqrt{b}y) $ and $ z=\sqrt{b}y $. It is straightforward to check that
   	\begin{equation}\label{pr2.4:10}
   		\begin{aligned}
   			&\mathcal{H}_{b}v=b(H_{b}u)(\sqrt{b}y),\quad \lVert v\rVert_{b}=\frac{1}{b^{\frac{3}{4}}}\lVert u\rVert_{L^{2}_{\rho,b}},\quad \lVert \partial_{y}v\rVert_{b}=\frac{1}{b^{\frac{1}{4}}}\lVert \partial_{z}u\rVert_{L^{2}_{\rho,b}},\\
      &\partial_{y}v(1)=\sqrt{b}\partial_{z}u(\sqrt{b}),\quad
   			\lVert \Delta v\rVert_{b}=b^{\frac{1}{4}}\lVert \Delta u\rVert_{L^{2}_{\rho,b}},\quad \lVert\Lambda v\rVert_{b}=\frac{1}{b^{\frac{3}{4}}}\lVert\Lambda u\rVert_{L^{2}_{\rho,b}},\\
      &\lVert yv\rVert_{b}=\frac{1}{b^{\frac{5}{4}}}\lVert zu\rVert_{L^{2}_{\rho,b}},\quad
   			\lVert y^{2}v\rVert_{b}=\frac{1}{b^{\frac{7}{4}}}\lVert z^{2}u\rVert_{L^{2}_{\rho,b}}.
   		\end{aligned}
   	\end{equation}
   Therefore let 
   \begin{equation}\label{pr2.4:11}
   	\eta_{b,k}(y)=\psi_{b,k}(z),\quad z=\sqrt{b}y,
   \end{equation}
and (\ref{pr2.4:1}) follows. Now the decomposition (\ref{pr2.4:2}) follows from (\ref{pr2.3:3}), and (\ref{pr2.3:5}) implies (\ref{pr2.4:3}). The estimates (\ref{pr2.4:4}),  (\ref{pr2.4:6}), (\ref{pr2.4:7}) and (\ref{pr2.4:9}) follows by rescaling (\ref{pr2.3.8}),  (\ref{pr2.3.11}), (\ref{pr2.3.12}) and (\ref{pr2.3.13}) respectively.\\
Next, directly from the definition (\ref{pr2.4:11}), we compute
\begin{equation*}
	b\partial_{b}\eta_{b,k}=\bigg[\bigg(\frac{\Lambda \psi_{b,k}}{2}+b\partial_{b}\psi_{b,k}\bigg)(\sqrt{b}y)\bigg], \quad \Lambda\psi_{b,k}(\sqrt{b}y)=\Lambda\eta_{b,k}(y),
\end{equation*}
which together with (\ref{pr2.3.10}), (\ref{pr2.3.9}) and by rescaling yields (\ref{pr2.4:5}) and (\ref{pr2.4:8}) .
   	\end{proof}
   \section{Renormaliesd equations and initialisation}
   We start with the classical modulated nonlinear decomposition of the flow. Firstly,  set
   \begin{equation}\label{3.1}
   	u(t,x)=v(s,y),\quad y=\frac{r}{\lambda(t)}.
   	\end{equation}
   Without loss of generality, we may assume that $ \lambda(0)=1 $ which thanks to the scaling symmetry of the equation. Consider the renormalised time
   \begin{equation*}
   	s(t)=s_{0}+\int_{0}^{t}\frac{d\tau}{\lambda^{2}(\tau)},
   	\end{equation*}
   here $ s_{0}\gg{1} $ is to be chosen later, and then injecting (\ref{3.1}) into (\ref{eq:2}) yields the renormalised equation
   \begin{equation}\label{3.3}
   	\left\{\begin{aligned}
&\partial_{s}v+\mathcal{H}_{a}v=0,\quad a=-\frac{\lambda_{s}}{\lambda},\\
&v(s,1)=0,\quad \partial_{y}v(s,1)=a.
   	\end{aligned}\right.
   \end{equation}
We now choose our initial data in the following way:\\
	 \textbf{Case {\boldmath{$ k=0 $}}}: We pick an initial data of the form
	\begin{equation*}
		v_{0}=-\frac{1}{C_{0}}{b^{*}(0)}^{\frac{3}{2}}\eta_{b^{*}(0),0}+\varepsilon^{*}_{0},\quad \lVert \varepsilon\rVert_{b^{*}(0)}\ll 1\ {\rm{and}}\ 0<b^{*}(0)\ll 1.
	\end{equation*}
By the standard  modulation argument (see for example \cite{CFP}, \cite{MP}, \cite{MR1}), there exists  a unique $ b(0)>0 $ with $ |b(0)-b^{*}(0)|\ll 1 $ and $ \lVert \varepsilon_{0}\rVert_{b(0)}\ll 1 $ satisfying
\begin{equation*}
	v_{0}=-\frac{1}{C_{0}}b^{\frac{3}{2}}(0)\eta_{b(0),0}+\varepsilon_{0}\quad {\rm{with}} \quad (\varepsilon_{0},\eta_{b(0),0})_{b(0)}=0
	\end{equation*}
and the corresponding solution of (\ref{3.3}) admits a unique time dependent decomposition 
\begin{equation}\label{3.5}
	v(s,y)=-\frac{1}{C_{0}}b^{\frac{3}{2}}(s)\eta_{b(s),0}+\varepsilon(s,y)\quad {\rm{with}} \quad (\varepsilon,\eta_{b,0})_{b}=0,
	\end{equation}
which makes sense as long as $ \varepsilon(s,y) $ is small enough in the weighted $ L^{2} $ space. Set
\begin{equation}\label{3.6}
	\varepsilon_{2}=\mathcal{H}_{b}\varepsilon
\end{equation}
and define the energy
\begin{equation*}
	E:=\lVert \mathcal{H}_{b}\varepsilon\rVert^{2}_{b},
\end{equation*}
which is  coercive  under the  orthogonality condition  (\ref{3.5}) (see Appendix A).
We assume the initial smallness of energy
\begin{equation}\label{3.8}
	 E(0)\leq b^{4}(0),
\end{equation}
here $ b(0)\ll 1 $ is small enough.
\begin{remark}
The choice of  coefficients $-\frac{1}{C_{0}}b^{\frac{3}{2}}(s)$ in the decomposition (\ref{3.5}) ensures that $-\frac{\lambda_{s}}{\lambda}\sim b $ as  expected which comes from the boundary condition (see (\ref{3.27})).
\end{remark}
\begin{remark}
	Note that the set of initial data satisfies (\ref{3.5}) and (\ref{3.8}) is non-empty and contains compactly supported initial data which has already been constructed in \cite{MP} in 2D. The construction in 3D is similar to the case in 2D, so we omit the details.
\end{remark}
 \textbf{Case {\boldmath{$ k\geq 1 $}}}: We firstly freeze the explicit value
   \begin{equation}\label{3.9}
	\left\{\begin{aligned}
	&b(s)=\frac{1}{2ks},\\
	&a^{e}(s)=\frac{k+1}{2ks},\\
	&b^{e}_{k}(s)=\left(-\frac{1}{C_{k}}\right)\frac{k+1}{(2ks)^{\frac{3}{2}}},\\
	& b^{e}_{j}(s)=0\quad  \forall\ 0\leq{j}\leq{k-1}.
	\end{aligned}\right.
\end{equation}
Next one can  check that $ b(s),a^{e}(s),b^{e}_{k}(s),b^{e}_{j}(s) $ satisfy the following equation
\begin{equation}\label{3.10}
	\left\{\begin{aligned}
&(b^{e}_{k})_{s}+bb_{k}^{e}\lambda_{b,k}+(a^{e}-b)b^{e}_{k}=O(s^{-3}),\\
&b_{s}+2b(a^{e}-b)=0,\\
&a^{e}=-C_{k}\frac{b^{e}_{k}}{\sqrt{b}}.
\end{aligned}\right.
\end{equation}
\begin{remark}
Note that the first and second equations above come from  the leading order term of the modulation equations (see (\ref{3.37})) and the third equation comes from the main term of $a$ which follows by the boundary condition (see (\ref{3.27})).
\end{remark}
Consider now the approximate solution of the form
\begin{equation}\label{in3}
	Q_{\beta}(y)=\sum_{j=0}^{k}b_{j}\eta_{b,j}(y)
\end{equation}
and then  the corresponding solution amdits the decomposition
\begin{equation}\label{3.12}
	v(s,y)=Q_{\beta(s)}(y)+\varepsilon(s,y), \quad (\eta_{b(s),j},\varepsilon)_{b}=0,\quad \forall\ 0\leq{j}\leq{k}.
\end{equation}
Set again
\begin{equation*}
	\varepsilon_{2}=\mathcal{H}_{b}\varepsilon,\quad E:=\lVert\mathcal{H}_{b}\varepsilon\rVert^{2}_{b},
\end{equation*}
which is also a coercive norm under the orthogonality conditions (\ref{3.12}) (see Appendix A).
We assume the initial smallness of energy
\begin{equation}\label{3.13}
	E(0)\leq\left\{\begin{aligned}
&b^{3}(0)\quad 1\leq{k}\leq{3},\\
&b^{3}(0)|\log b(0)|\quad k=4,\\
&b^{\frac{5}{2}+\frac{2}{k}}(0)\quad k\geq 5,
	\end{aligned}\right.
\end{equation}
where $ b(0)>0 $ is small enough to be determined later. For $ k\geq{1} $, the set of initial data will be built as a codimension $ k $ manifold. In order to prepare the initial data, we consider the decomposition
\begin{equation}\label{3.14}
	b_{j}(s)=b^{e}_{j}(s)+\tilde{b}_{j}(s),\quad \tilde{b}_{j}(s)=\frac{U_{j}(s)}{s^{\frac{3}{2}+\alpha_{k}}},\quad \forall\ 0\leq{j}\leq{k},
\end{equation}
here
\begin{equation}\label{3.15}
	\alpha_{k}=\left\{\begin{aligned}
		& \frac{1}{8}\quad 1\leq{k}\leq 4,\\
		&\frac{1}{2k}\quad k\geq 5.
	\end{aligned}\right.
\end{equation}
We introduce the $ (2\times 2) $ matrix $ B_{k} $\footnote{$B_{k}$ is the matrix of the equations for $U_{k-1}$ and $U_{k}$ (see (\ref{3.5:8}) below) .}
\begin{equation}\label{3.16}
	B_{k}=\begin{pmatrix}
		-a_{k} & -b_{k}\\
		c_{k} &  d_{k}\\
	\end{pmatrix}=
	\begin{pmatrix}
	-(k+1-\alpha_{k}) &-(k+1)\frac{C_{k-1}}{C_{k}}\\
	\frac{(k+1)(2k+1)}{2k}\frac{A_{k-1}}{A_{k}} &\frac{1}{k}+\alpha_{k}+\frac{(2k+1)(k+1)B_{k-1}}{2kB_{k}}
	\end{pmatrix}.
\end{equation}
The matrix $ B_{k} $ admits one positive eigenvalue $ \mu^{k}_{1}>0 $ and one negative eigenvalue $ \mu^{k}_{2}<0 $. In fact, the characteristic polynomial of $ B_{k} $ is 
\begin{equation*}
	f_{k}(\lambda)=\lambda^{2}+(a_{k}-d_{k})\lambda+c_{k}b_{k}-a_{k}d_{k},
\end{equation*}
and note that
\begin{equation*}
	\begin{aligned}
		a_{k}-d_{k}&=k+1-\alpha_{k}-\left(\frac{1}{k}+\alpha_{k}+\frac{(2k+1)(k+1)B_{k-1}}{2kB_{k}}\right)\\
  &=k+1-\alpha_{k}-\left(\frac{1}{k}+\alpha_{k}+\frac{(2k+1)(k+1)}{2k}\frac{\frac{\Gamma\left(\frac{1}{2}+k-1+1\right)}{(k-1)!\Gamma\left(\frac{1}{2}+1\right)}}{\frac{\Gamma\left(\frac{1}{2}+k+1\right)}{k!\Gamma\left(\frac{1}{2}+1\right)}}\right)\\
  &=k+1-\alpha_{k}-\left(\frac{1}{k}+\alpha_{k}+\frac{(2k+1)(k+1)}{2k}\cdot \frac{2k}{2k+1}\right)\\
  &=-2\alpha_{k}-\frac{1}{k}<0,
	\end{aligned}
\end{equation*}
and
\begin{equation*}
	\begin{aligned}
		c_{k}b_{k}-a_{k}d_{k}=&\frac{(k+1)(2k+1)}{2k}\frac{A_{k-1}}{A_{k}}(k+1)\frac{C_{k-1}}{C_{k}}\\
  &-(k+1-\alpha_{k})\left(\frac{1}{k}+\alpha_{k}+\frac{(2k+1)(k+1)B_{k-1}}{2kB_{k}}\right)\\
  =&\frac{(k+1)(2k+1)}{2k}(k+1)\frac{2k}{2k+1}\\
  &-(k+1-\alpha_{k})\left(\frac{1}{k}+\alpha_{k}+\frac{(2k+1)(k+1)}{2k}\frac{2k}{2k+1}\right)\\
  =&-1-\frac{1}{k}+\alpha_{k}\left(\frac{1}{k}+\alpha_{k}\right)<0,
	\end{aligned}
\end{equation*}
hence, direct computation yields
\begin{equation*}
	\mu^{k}_{1}=\alpha_{k}+\frac{1}{k}+1>0,\quad \mu^{k}_{2}=-1+\alpha_{k}<0.
\end{equation*}
As a result, $ B_{k} $ is diagonaliable and there exists a $ P_{k}\in GL_{2}(\RR) $, such that
\begin{equation*}
	P_{k}B_{k}P_{k}^{-1}=\begin{pmatrix}
		\mu^{k}_{2} & 0\\
		0 & \mu^{k}_{1}
	\end{pmatrix}.
\end{equation*}
Define 
\begin{equation}\label{3.18}
	\begin{pmatrix}
		W_{k}\\
		W_{k-1}
	\end{pmatrix}=P_{k}\begin{pmatrix}
	U_{k}\\
	U_{k-1},
\end{pmatrix}
	\end{equation}
and assume the initial bound
\begin{equation}\label{3.19}
	|W_{k}(0)|\leq 1,
\end{equation}
\begin{equation}\label{3.20}
	|W_{k-1}(0)|^{2}+\sum_{j=0}^{k-2}\frac{|U_{j}(0)|^{2}}{\delta^{2}}\leq 4,
\end{equation}
for some $ 0<\delta\ll 1 $ to be chosen later.\\
We now consider the bootstrap bounds:
\begin{itemize}
	\item $ k=0 $:
	\begin{equation}\label{3.21}
		E(s)\leq Kb(s)^{4},
	\end{equation}
for some $ K>0 $ large enough to be chosen later, and 
\begin{equation}\label{3.22}
	0<b(s)<b^{*},
\end{equation}
here $ b^{*}>0 $ is a universal small constant to be chosen later.\\
\item $ k\geq 1 $:
\begin{equation}\label{3.23}
	E(s)\leq\left\{\begin{aligned}
& Kb(s)^{3}\quad 1\leq{k}\leq 3,\\
& Kb(s)^{3}|\log b(s)|\quad k=4,\\
& Kb(s)^{\frac{5}{2}+\frac{2}{k}}\quad k\geq 5,
\end{aligned}\right.
\end{equation}
for some $ K>0 $ universally large to be chosen later, and
\begin{equation}\label{3.24}
	|W_{k}(s)|\leq 2,
	\end{equation}
\begin{equation}\label{3.25}
	|W_{k-1}(s)|^{2}+\sum_{j=0}^{k-2}\bigg|\frac{U_{j}(s)}{\delta}\bigg|^{2}\leq 4.
\end{equation}
\end{itemize}
Define
\begin{equation*}
	s^{*}=\left\{\begin{aligned}
&{\rm sup}\{s\geq s_{0}|(\ref{3.21}),(\ref{3.22})\ {\rm hold}\ {\rm on}\ [s_{0},s]\}\quad {\rm for}\ k=0,\\
&{\rm sup}\{s\geq s_{0}|(\ref{3.23}),(\ref{3.24}),(\ref{3.25})\ {\rm hold}\ {\rm on}\ [s_{0},s]\}\quad {\rm for}\ k\geq 1.
	\end{aligned}\right.
\end{equation*}
\begin{remark}
	Note in particular that under the above assumption, we have the rough bounds
	\begin{equation}\label{3.26}
		E(s)\lesssim\left\{\begin{aligned}
			&b(s)^{\frac{7}{2}}\quad k=0,\\
			&b(s)^{\frac{47}{16}}\quad 1\leq k \leq 4,\\
			&b(s)^{\frac{5}{2}+\frac{7}{4k}}\quad k\geq 5,
		\end{aligned}\right. \quad |b_{j}(s)|\lesssim b^{\frac{3}{2}}\quad \forall\ 0\leq{j}\leq{k}.
	\end{equation}
for any	 $s\in\left[ s_{0},s^{*}\right)  $ for some $ s_{0}>0 $ large enough.
\end{remark}
The heart of our analysis lines in the following bootstrap proposition.
\begin{proposition}[Bootstrap estimates on $ b $ and $ \varepsilon $]\label{pr3.3}
	Let $ s^{*} $ be defined above, then the following statements hold:
	\begin{itemize}
		\item Stable regime: for $ k=0 $, $ s^{*}=+\infty $.
		\item Unstable regime: for $ k\geq 1 $, there exist $ 0<\delta\ll 1 $ and $( U_{0}(s_{0}),U_{1}(s_{0}),...,\\U_{k-2}(s_{0}),W_{k-1}(s_{0}) )$ depending on $ \varepsilon(0) $ and satisfying $ (\ref{3.13}) $, $ (\ref{3.19}) $, $ (\ref{3.20}) $, such that $ s^{*}=+\infty $.
	\end{itemize}
\end{proposition}
In what follows, we study the flow in the bootstrap regime $ s\in[s_{0},s^{*}) $ and finally close the bootstrap bounds. 
\section{Modulation  equations}
In this section we derive the  modulation equations for the parameters $ (a,\ b,\  (\tilde{b}_{j})_{0\leq{j}\leq{k}}) $. We start with the constraint induced by the boundary conditions.
\begin{lemma}[Boundary conditions]\label{le4.1} Let $ v $ be a solution of $(\ref{3.3})$ with the decomposition $ (\ref{3.5}), (\ref{3.12}) $ , then we have
	\begin{align}
		&a=\left\{\begin{aligned}
			&b+O\bigg(b^{\frac{3}{2}}+\frac{\sqrt{E}}{b^{\frac{1}{4}}}\bigg)\quad k=0,\\
			&-\sum_{j=0}^{k}\frac{C_{j}b_{j}}{\sqrt{b}}+O\bigg(\sum_{j=0}^{k}|b_{j}|+\frac{\sqrt{E}}{b^{\frac{1}{4}}}\bigg)\quad k\geq 1,
		\end{aligned}\right.\label{3.27} \\ 
		&\varepsilon_{2}(1)=-a(a-b),\label{3.28} \\
		&\partial_{y}\varepsilon_{2}(1)=\left\{\begin{aligned}
			&-a_{s}-\lambda_{b,0}b^{2}-(a-b)(a^{2}-a)+O(b^{3})\quad k=0,\\
			&-a_{s}+\sqrt{b}\sum_{j=0}^{k}C_{j}\lambda_{b,j}b_{j}-(a-b)(a^{2}-a)+O\bigg(b\sum_{j=0}^{k}|b_{j}|\bigg)\quad k\geq{1}.
		\end{aligned}\right.\label{3.29}
	\end{align}
\end{lemma}
\begin{remark}
	Note that from (\ref{3.26}) and (\ref{3.27}), for $ k=0 $, we arrive at  the further estimate
	\begin{equation}\label{3.30}
		|a-b|\lesssim b^{\frac{3}{2}}.
	\end{equation}
\end{remark}
\begin{proof}
	Firstly, using (\ref{pr2.4:2}), (\ref{pr2.4:3}) and (\ref{pr2.3:4}), one has
	\begin{equation}\label{3.31}
		\begin{aligned}
			\partial_{y}\eta_{b,j}(1)&=\frac{1}{\sqrt{b}}\bigg(-P_{j}(\sqrt{b})-\sum_{i=0}^{j-1}\mu_{b,ij}P_{i}(\sqrt{b})\bigg)+\partial_{y}\tilde{\eta}_{b,j}(1)\\
			&=-\frac{C_{j}}{\sqrt{b}}+O(1),
		\end{aligned}
	\end{equation}
where $ C_{j}=P_{j}(0) $. Then for $ k=0 $, by the decomposition (\ref{3.5}), the boundary condition $ \partial_{y}v(s,1)=a $ and the coercivity estimate (\ref{leCH_{b}:14}), we conclude that
\begin{equation*}
\begin{aligned}
	a=\partial_{y}v(s,1)&=\bigg(-\frac{1}{C_{0}}b^{\frac{3}{2}}(s)\partial_{y}\eta_{b(s),0}+\partial_{y}\varepsilon\bigg)(1)\\
	&=-\frac{1}{C_{0}}b^{\frac{3}{2}}\bigg(-\frac{C_{0}}{\sqrt{b}}+O(1)\bigg)+\partial_{y}\varepsilon(1)\\
	&=b+O\bigg(b^{\frac{3}{2}}+\frac{\sqrt{E}}{b^{\frac{1}{4}}}\bigg).
\end{aligned}	
\end{equation*}
Similar computations applied for the case $ k\geq 1 $ yield
\begin{equation*}
	a=\partial_{y}v(s,1)=\sum_{j=0}^{k}b_{j}\partial_{y}\eta_{b,j}(1)+\partial_{y}\varepsilon(1)=-\sum_{j=0}^{k}\frac{C_{j}b_{j}}{\sqrt{b}}+O\bigg(\sum_{j=0}^{k}|b_{j}|+\frac{\sqrt{E}}{b^{\frac{1}{4}}}\bigg),
	\end{equation*}
which is (\ref{3.27}).\\
Inserting $ v(s,1)=0 $ into the equation (\ref{3.3}) yields $ \mathcal{H}_{a}v(1)=0 $, hence by (\ref{pr2.4:1}) together with that $ \partial_{y}v(s,1)=a $, we get
\begin{equation*}
		0=\mathcal{H}_{a}v(1)=(\mathcal{H}_{b}+(a-b)y\partial_{y})v(1)=\varepsilon_{2}(1)+(a-b)a,
\end{equation*}
which is (\ref{3.28}).\\
From $ \partial_{y}v(s,1)=a $ again, 
\begin{equation*}
	\partial_{y}\partial_{s}v(s,1)=a_{s},
\end{equation*}
Next, differentiating with respect to $ y $ on both sides of (\ref{3.3}) and taking $ y=1 $, we obtain
\begin{equation}\label{3.32}
	\begin{aligned}
		0&=\partial_{y}\partial_{s}v(s,1)+\partial_{y}(\mathcal{H}_{b}v)(1)+(a-b)\partial_{y}(\Lambda v)(1)\\
		&=a_{s}+\partial_{y}\varepsilon_{2}(1)+b\sum_{j=0}^{k}b_{j}\lambda_{b,j}\partial_{y}\eta_{b,j}(1)+(a-b)(\partial_{y}v+y\partial_{yy}v)(1).
	\end{aligned}
\end{equation}
Note also that $ \mathcal{H}_{a}v(1)=0 $ and $ \partial_{y}v(s,1)=a $  which imply
\begin{equation*}
	\partial_{yy}v(1)=\bigg(-\frac{2}{y}\partial_{y}v\bigg)(1)+a(y\partial_{y}v)(1)=a^{2}-2a.
\end{equation*}
Then taking into (\ref{3.32}) and combining with (\ref{3.31}) yield
\begin{equation*}
 a_{s}+\partial_{y}\varepsilon_{2}(1)+\lambda_{b,0}b^{2}+(a-b)(a^{2}-a)=O(b^{3})\quad k=0,
\end{equation*}
and
\begin{equation*}
	a_{s}+\partial_{y}\varepsilon_{2}(1)-\sqrt{b}\sum_{j=0}^{k}C_{j}\lambda_{b,j}b_{j}+(a-b)(a^{2}-a)=O\bigg(b\sum_{j=0}^{k}|b_{j}|\bigg)\quad k\geq 1,
\end{equation*}
which is (\ref{3.29}). 
	\end{proof}
Before we move on to the derivation of modulation equations, let us firstly compute the equation for $\varepsilon$.
\begin{lemma}[Equations for $\varepsilon$]From the decomposition (\ref{3.5}) and 
 (\ref{3.12}), the remainder term $ \varepsilon $ satisfies
\begin{equation}\label{3.40}
	\partial_{s}\varepsilon+\mathcal{H}_{a}\varepsilon=\mathcal{F},
\end{equation}
where 
\begin{equation}\label{3.41}
	\mathcal{F}=-{\rm{Mod}}-\Psi,
\end{equation}
and
	\begin{itemize}
		\item for $ k=0 $:
	\begin{equation}\label{3.34}
		{\rm{Mod}}:=\bigg((b_{0})_{s}-\frac{1}{2}\frac{b_{0}b_{s}}{b}+b_{0}b\lambda_{b,0}\bigg)\eta_{b,0},
	\end{equation}
with the error term satisfying the bound
\begin{equation}\label{3.35}
	\lVert \Psi\rVert_{b}+\frac{1}{b}\lVert\mathcal{H}_{b}\Psi\rVert_{b}\lesssim b^{\frac{9}{4}}+\sqrt{b}\sqrt{E}+\frac{|b_{s}|}{b^{\frac{1}{4}}},
\end{equation}
 here for simplity we denote $ b_{0}(s)=-\frac{1}{C_{0}}b^{\frac{3}{2}}(s) $.
\item for $ k\geq{1} $:
\begin{equation}\label{3.37}
	\begin{aligned}
		{\rm{Mod}}:&=\bigg[(b_{k})_{s}+bb_{k}\lambda_{b,k}+(a-b)b_{k}+\frac{(2k-1)b_{k}}{2b}\Phi\bigg]\eta_{b,k}\\
		&+\sum_{j=0}^{k-1}\bigg[(b_{j})_{s}+bb_{j}\lambda_{b,j}+(a-b)b_{j}+\frac{(2j-1)b_{j}}{2b}\Phi-(2j+3)\frac{A_{j}}{A_{j+1}}\frac{\Phi}{2b}b_{j+1}\bigg]\eta_{b,j},
	\end{aligned}
\end{equation}
here
\begin{equation*}
	\Phi:=b_{s}+2(a-b)b
\end{equation*}
and the error term satisfies the bound
\begin{equation}\label{3.39}
	\lVert \Psi\rVert_{b}+\frac{1}{b}\lVert\mathcal{H}_{b}\Psi\rVert_{b}\lesssim b^{\frac{7}{4}}+\frac{|\Phi|}{b^{\frac{1}{4}}}.
\end{equation}
	\end{itemize}
\end{lemma}
\begin{proof}
 Injecting the decomposition  (\ref{3.5}) and (\ref{3.12}) into equation (\ref{3.3}) yields
 \begin{equation*}
 \partial_{s}\varepsilon+\mathcal{H}_{a}\varepsilon=-\left\{\begin{aligned}
			&\partial_{s}(b_{0}\eta_{b,0})+\mathcal{H}_{a}(b_{0}\eta_{b,0}),\quad k=0,\\
			&\partial_{s}Q_{\beta}+\mathcal{H}_{a}Q_{\beta},\quad k\geq 1.
		\end{aligned}\right.
 \end{equation*}
It remains to estimate the terms on the right side.\\
\textbf{Case {\boldmath{$ k=0 $}}}.
	From (\ref{pr2.4:1}) and by direct computation,
	\begin{align*}
			\partial_{s}(b_{0}\eta_{b,0})+\mathcal{H}_{a}(b_{0}\eta_{b,0})=&(b_{0})_{s}\eta_{b,0}+b_{0}b_{s}\partial_{b}\eta_{b,0}+\mathcal{H}_{b}(b_{0}\eta_{b,0})+(a-b)\Lambda(b_{0}\eta_{b,0})\\
			=&(b_{0})_{s}\eta_{b,0}+b_{0}b_{s}\partial_{b}\eta_{b,0}+b_{0}b\lambda_{b,0}\eta_{b,0}+(a-b)\Lambda(b_{0}\eta_{b,0})\\
			=&\bigg((b_{0})_{s}-\frac{1}{2}\frac{b_{0}b_{s}}{b}+b_{0}b\lambda_{b,0}\bigg)\eta_{b,0}\\
   &+(a-b)b_{0}\Lambda\eta_{b,0}+\frac{b_{0}b_{s}}{b}\bigg(b\partial_{b}\eta_{b,0}+\frac{1}{2}\eta_{b,0}\bigg).
			\end{align*}
Then from (\ref{pr2.4:6}), (\ref{pr2.4:8}), (\ref{3.26}), (\ref{3.27}) and the definition of $ \Psi $, we obtain
\begin{equation*}
	\lVert \Psi\rVert_{b}+\frac{1}{b}\lVert \mathcal{H}_{b}\Psi\rVert_{b}\lesssim|a-b|\cdot b^{\frac{3}{2}}\cdot \frac{1}{b^{\frac{3}{4}}}+\sqrt{b}\cdot|b_{s}|\cdot\frac{1}{b^{\frac{3}{4}}}\lesssim b^{\frac{9}{4}}+\sqrt{b}\sqrt{E}+\frac{|b_{s}|}{b^{\frac{1}{4}}},
\end{equation*}
which is (\ref{3.35}).\\
\textbf{Case {\boldmath{$ k\geq 1 $}}}. Carrying out the exact same calculations as Case  $ k=0 $ yields
\begin{equation*}
	\begin{aligned}
	\partial_{s}Q_{\beta}+\mathcal{H}_{a}Q_{\beta}=&\sum_{j=0}^{k}(b_{j})_{s}\eta_{b,j}
+\sum_{j=0}^{k}b_{j}b_{s}\partial_{b}\eta_{b,j}\\
&+\sum_{j=0}^{k}b_{j}\lambda_{b,j}b\eta_{b,j}+(a-b)\Lambda\bigg(\sum_{j=0}^{k}b_{j}\eta_{b,j}\bigg)\\
	=&\sum_{j=0}^{k}(b_{j})_{s}\eta_{b,j}+\sum_{j=0}^{k}b_{j}\lambda_{b,j}b\eta_{b,j}+\sum_{j=0}^{k}(a-b)b_{j}\eta_{b,j}\\
 &+\sum_{j=0}^{k}\bigg[\frac{\Phi}{2b}2j\bigg(\eta_{b,j}-\frac{2j+1}{2j}\frac{A_{j-1}}{A_{j}}\eta_{b,j-1}\bigg)-\frac{\Phi}{2b}\eta_{b,j}\bigg]b_{j}\\
	&+\sum_{j=0}^{k}b_{j}\bigg(\frac{\Phi-2b(a-b)}{2b}\bigg)(2b\partial_{b}\eta_{b,j}-\Lambda\eta_{b,j}+\eta_{b,j})\\
 &+\frac{b_{j}\Phi}{2b}\bigg(\Lambda\eta_{b,j}-2j\bigg(\eta_{b,j}-\frac{2j+1}{2j}\frac{A_{j-1}}{A_{j}}\eta_{b,j-1}\bigg)\bigg)\\
	=&{\rm{Mod}}+\Psi,
\end{aligned}
\end{equation*}
where $ {\rm{Mod}} $ is defined as (\ref{3.37}) and
\begin{align*}
	\Psi=&\sum_{j=0}^{k}b_{j}\bigg(\frac{\Phi-2b(a-b)}{2b}\bigg)(2b\partial_{b}\eta_{b,j}-\Lambda\eta_{b,j}+\eta_{b,j})\\
 &+\frac{b_{j}\Phi}{2b}\bigg(\Lambda\eta_{b,j}-2j\bigg(\eta_{b,j}-\frac{2j+1}{2j}\frac{A_{j-1}}{A_{j}}\eta_{b,j-1}\bigg)\bigg).
\end{align*}
Moreover, from (\ref{pr2.3:8}), (\ref{pr2.4:9}) and (\ref{pr2.4:10}), one has
\begin{equation*}
	\lVert \Psi\rVert_{b}+\frac{1}{b}\lVert\mathcal{H}_{b}\Psi\rVert_{b}\lesssim\sum_{j=0}^{k}\bigg(|b_{j}|\frac{|\Phi|}{b}+|b_{j}||a-b|\bigg)\frac{1}{b^{\frac{3}{4}}}\lesssim b^{\frac{7}{4}}+\frac{|\Phi|}{b^{\frac{1}{4}}},
\end{equation*}
which is (\ref{3.39}).
	\end{proof}
We now derive the modulation equations for $ (b,(\tilde{b}_{j})_{0\leq{j}\leq{k}}) $ which is a consequence of the orthogonality conditions $ (\ref{3.5}) $, $ (\ref{3.12}) $ and the equations for $ \varepsilon$.
\begin{lemma}[Modulation equations for $ (b,(\tilde{b}_{j})_{0\leq{j}\leq{k}}) $]There hold the bounds for the modulation parameters
	\begin{itemize}
		\item Case $ k=0 $:
		\begin{equation}\label{3.42}
			\bigg|b_{s}+\sqrt{\frac{2}{\pi}}b^{\frac{5}{2}}\bigg|\lesssim b^{3}.
		\end{equation}
	\item Case $ k\geq 1 $:
	\begin{equation}\label{3.43}
		\begin{aligned}
			\bigg|&(\tilde{b}_{k})_{s}+\frac{3}{2s}\tilde{b}_{k}+\frac{k+1}{C_{k}s}\sum_{j=0}^{k}C_{j}\tilde{b}_{j}\bigg|\\
   &+\bigg|(\tilde{b}_{k-1})_{s}+\bigg(\frac{3}{2s}-\frac{1}{ks}\bigg)\tilde{b}_{k-1}-\frac{(k+1)(2k+1)}{2kC_{k}s}\frac{A_{k-1}}{A_{k}}\sum_{j=0}^{k}C_{j}\tilde{b}_{j}\bigg|\\
			&+\bigg|(\tilde{b}_{j})_{s}+\bigg(\frac{j}{ks}+\frac{1}{2s}\bigg)\tilde{b}_{j}\bigg|
   \lesssim b^{\frac{5}{4}}\sqrt{E}+b^{\frac{5}{2}+2\alpha_{k}}.
		\end{aligned}
		\end{equation}
	\end{itemize}
\end{lemma}
\begin{proof} 
		  From the orthogonality conditions (\ref{3.5}) and (\ref{3.12}),  we get for any $0\leq j\leq k$,
		\begin{align*}
			0=\frac{d}{ds}(\varepsilon,\eta_{b,j})_{b}=&\frac{d}{ds}\bigg(\int_{1}^{+\infty}\varepsilon\eta_{b,j}e^{-\frac{by^{2}}{2}}y^{2}dy\bigg)\\
   =&(\partial_{s}\varepsilon,\eta_{b,j})_{b}-\frac{b_{s}}{2}(\varepsilon,y^{2}\eta_{b,j})_{b}+(\varepsilon,b_{s}\partial_{b}\eta_{b,j})_{b}.
		\end{align*}
		Combining with (\ref{3.40}) and (\ref{3.41}) yields
		\begin{equation}\label{3.44}
		 ({\rm{Mod}},\eta_{b,j})_{b}=-(a-b)(\Lambda\varepsilon,\eta_{b,j})_{b}-\frac{b_{s}}{2}(\varepsilon,y^{2}\eta_{b,j})_{b}+(\varepsilon,b_{s}\partial_{b}\eta_{b,j})_{b}-(\Psi,\eta_{b,j})_{b}.
		\end{equation}
\textbf{ Case {\boldmath{$ k=0 $}}}. From (\ref{pr2.4:4}), (\ref{pr2.4:6}), (\ref{pr2.4:7}), (\ref{pr2.4:8}), (\ref{3.27}), (\ref{3.44}) and together with the coercivity estimates of $ \mathcal{H}_{b} $ (see Lemma \ref{leCH{b}}), we have
\begin{equation*}
	\begin{aligned}
		\bigg|\frac{({\rm{Mod}},\eta_{b,0})_{b}}{(\eta_{b,0},\eta_{b,0})_{b}}\bigg|\lesssim& b^{\frac{5}{2}}(|a-b|\lVert\Lambda\varepsilon\rVert_{b}\lVert\eta_{b,0}\rVert_{b}+|b_{s}|\lVert\varepsilon\rVert_{b}\lVert y^{2}\eta_{b,0}\rVert_{b}\\
  &+\lVert\varepsilon\rVert_{b}|b_{s}|\lVert \partial_{b}\eta_{b,0}\rVert_{b}+\lVert\Psi\rVert_{b}\lVert\eta_{b,0}\rVert_{b})\\
		\lesssim& b^{\frac{5}{2}}\bigg[\bigg(b^{\frac{3}{2}}+\frac{\sqrt{E}}{b^{\frac{1}{4}}}\bigg)\frac{\sqrt{E}}{b}\frac{1}{b^{\frac{5}{4}}}+|b_{s}|\frac{\sqrt{E}}{b}\frac{1}{b^{\frac{9}{4}}}\\
  &+\frac{\sqrt{E}}{b}\frac{|b_{s}|}{b}\frac{1}{b^{\frac{5}{4}}}+\bigg(b^{\frac{9}{4}}+\sqrt{b}\sqrt{E}+\frac{|b_{s}|}{b^{\frac{1}{4}}}\bigg)\frac{1}{b^{\frac{5}{4}}}\bigg]\\
		\lesssim& b^{\frac{7}{4}}\sqrt{E}+\frac{|b_{s}|}{b^{\frac{3}{4}}}\sqrt{E}+E+b^{\frac{7}{2}}+b|b_{s}|.			
	\end{aligned}
\end{equation*}
In view of  (\ref{3.34}) and together with the rough estimate (\ref{3.26}), we arrive at
\begin{equation*}
	|b_{s}+\lambda_{b,0}b^{2}|\lesssim b^{3}+\sqrt{b}|b_{s}|.
\end{equation*}
Moreover, for $ b>0 $ small enough,
\begin{equation*}
	|b_{s}|\lesssim b^{\frac{5}{2}}.
\end{equation*}
Hence we conclude that
\begin{equation}\label{3.45}
	|b_{s}+\lambda_{b,0}b^{2}|\lesssim b^{3}.
	\end{equation}
Note also that from (\ref{pr2.3:2}),
\begin{equation*}
	\lambda_{b,0}=C^{2}_{0}\sqrt{b}+O(b)=\frac{\Gamma(\frac{3}{2})}{\sqrt{2}\Gamma(\frac{3}{2})^{2}}\sqrt{b}+O(b)=\sqrt{\frac{2}{\pi}}\sqrt{b}+O(b).
\end{equation*}
Then injecting into (\ref{3.45}) yields (\ref{3.42}).\\
 \textbf{ Case {\boldmath{$ k\geq 1 $}}}. We perform integration by parts with respect to the first term on the right side of (\ref{3.44}) and obtain
\begin{equation*}
	\begin{aligned}
		({\rm{Mod}},\eta_{b,j})_{b}=&(a-b)(\varepsilon,\Lambda\eta_{b,j})_{b}-b(a-b)(\varepsilon,y^{2}\eta_{b,j})_{b}\\
  &-\frac{b_{s}}{2}(\varepsilon,y^{2}\eta_{b,j})_{b}+(\varepsilon,b_{s}\partial_{b}\eta_{b,j})_{b}-(\Psi,\eta_{b,j})_{b}\\
		=&-\frac{\Phi}{2}(\varepsilon,y^{2}\eta_{b,j})_{b}+\frac{\Phi}{2b}(\varepsilon,\Lambda\eta_{b,j})_{b}\\
  &-\frac{b_{s}}{b}\bigg(\varepsilon,\frac{1}{2}\Lambda\eta_{b,j}-b\partial_{b}\eta_{b,j}\bigg)_{b}-(\Psi,\eta_{b,j})_{b}.
	\end{aligned}
\end{equation*}
As a result, from (\ref{pr2.4:4}), (\ref{pr2.4:6}), (\ref{pr2.4:7}), (\ref{pr2.4:8}), (\ref{3.39}) and (\ref{leCH_{b}:14}),
\begin{equation}\label{le3.6:1.2}
	\begin{aligned}
		\bigg|\frac{({\rm{Mod}},\eta_{b,j})_{b}}{(\eta_{b,j},\eta_{b,j})_{b}}\bigg|\lesssim& b^{\frac{5}{2}}\bigg(|\Phi|\lVert\varepsilon\rVert_{b}\lVert y^{2}\eta_{b,j}\rVert_{b}+\frac{|\Phi|}{b}\lVert\varepsilon\rVert_{b}\lVert\Lambda\eta_{b,j}\rVert_{b}\\
  &+\frac{|b_{s}|}{b}\lVert\varepsilon\rVert_{b}\lVert\Lambda\eta_{b,j}-2b\partial_{b}\eta_{b,j}\rVert_{b}+\lVert\Psi\rVert_{b}\lVert\eta_{b,j}\rVert_{b}\bigg)\\
		\lesssim& b^{\frac{5}{2}}\bigg(|\Phi|\frac{\sqrt{E}}{b}\frac{1}{b^{\frac{9}{4}}}+\frac{|\Phi|}{b}\frac{\sqrt{E}}{b}\frac{1}{b^{\frac{5}{4}}}+b\frac{\sqrt{E}}{b}\frac{1}{b^{\frac{5}{4}}}+\frac{1}{b^{\frac{5}{4}}}\bigg(b^{\frac{7}{4}}+\frac{|\Phi|}{b^{\frac{1}{4}}}\bigg)\bigg)\\
		\lesssim& \frac{|\Phi|}{b^{\frac{3}{4}}}\sqrt{E}+b^{\frac{5}{4}}\sqrt{E}+b|\Phi|+b^{3}.
	\end{aligned}
\end{equation}
 It remains to estimate  $ \Phi $. From the boundary condition (\ref{3.27}) and together with (\ref{3.10}), (\ref{3.14}), we have
\begin{equation}\label{le3.6:1.3}
	\begin{aligned}
		a&=-\sum_{j=0}^{k}\frac{C_{j}b_{j}}{\sqrt{b}}+O\bigg(\sum_{j=0}^{k}|b_{j}|+\frac{\sqrt{E}}{b^{\frac{1}{4}}}\bigg)\\
		&=-\frac{C_{k}b^{e}_{k}}{\sqrt{b}}-\sum_{j=0}^{k}\frac{C_{j}\tilde{b}_{j}}{\sqrt{b}}+O\bigg(b^{\frac{3}{2}}+\frac{\sqrt{E}}{b^{\frac{1}{4}}}\bigg)\\
		&=a^{e}-\sum_{j=0}^{k}\frac{C_{j}\tilde{b}_{j}}{\sqrt{b}}+O\bigg(b^{\frac{3}{2}}+\frac{\sqrt{E}}{b^{\frac{1}{4}}}\bigg).
	\end{aligned}
\end{equation}
Hence, from (\ref{3.10}) again
\begin{equation}\label{le3.6:1.4}
	\begin{aligned}
		\Phi=b_{s}+2b(a-b)&=b_{s}+2b\bigg(a^{e}-\sum_{j=0}^{k}\frac{C_{j}\tilde{b}_{j}}{\sqrt{b}}+O\bigg(b^{\frac{3}{2}}+\frac{\sqrt{E}}{b^{\frac{1}{4}}}\bigg)\bigg)-2b^{2}\\
		&=b_{s}+2b(a^{e}-b)-2\sqrt{b}\sum_{j=0}^{k}C_{j}\tilde{b}_{j}+O(b^{\frac{5}{2}}+b^{\frac{3}{4}}\sqrt{E})\\
		&=-2\sqrt{b}\sum_{j=0}^{k}C_{j}\tilde{b}_{j}+O(b^{\frac{5}{2}}+b^{\frac{3}{4}}\sqrt{E}).
		\end{aligned}
\end{equation}
In particular, from (\ref{3.14}), (\ref{3.15}), (\ref{3.26}) and the bootstrap assumption, we obtain the rough bound
\begin{equation}\label{le3.6:1.5}
		|\Phi|\lesssim\left\{\begin{aligned}
    &\sqrt{b}\cdot b^{\frac{3}{2}+\alpha_{k}}+b^{\frac{5}{2}}+b^{\frac{3}{4}}\cdot b^{\frac{47}{32}}\lesssim b^{\frac{17}{8}}\quad 1\leq{k}\leq 4,\\
    &\sqrt{b}\cdot b^{\frac{3}{2}+\alpha_{k}}+b^{\frac{5}{2}}+b^{2}\cdot b^{\frac{7}{8k}}\lesssim b^{2+\frac{1}{2k}}\quad k\geq 5.
	\end{aligned}\right.
\end{equation}
Inserting into (\ref{le3.6:1.2}) yields
\begin{equation}\label{le3.6:1.6}
	\bigg|\frac{({\rm{Mod}},\eta_{b,j})_{b}}{(\eta_{b,j},\eta_{b,j})_{b}}\bigg|\lesssim b^{\frac{5}{4}}\sqrt{E}+b^{3}.
\end{equation}
Next using (\ref{3.10}), (\ref{3.14}), (\ref{le3.6:1.3}) and (\ref{le3.6:1.4}), we compute for 
 $ j=k $,
	\begin{align*}
		&(b_{k})_{s}+bb_{k}\lambda_{b,k}+(a-b)b_{k}+\frac{(2k-1)b_{k}}{2b}\Phi\notag\\
		=&(b^{e}_{k}+\tilde{b}_{k})_{s}+b(b^{e}_{k}+\tilde{b}_{k})(2k+\mu_{k})+(a-a^{e})(b^{e}_{k}+\tilde{b}_{k})+(a^{e}-b)(b^{e}_{k}+\tilde{b}_{k})\notag\\
		&+\frac{(2k-1)(b^{e}_{k}+\tilde{b}_{k})}{2b}\bigg(-2\sqrt{b}\sum_{j=0}^{k}C_{j}\tilde{b}_{j}+O(b^{\frac{5}{2}}+b^{\frac{3}{4}}\sqrt{E})\bigg)\notag\\
		=&(b^{e}_{k})_{s}+2kbb^{e}_{k}+(a^{e}-b)b^{e}_{k}+(\tilde{b}_{k})_{s}+2kb\tilde{b}_{k}-b^{e}_{k}\sum_{j=0}^{k}\frac{C_{j}\tilde{b}_{j}}{\sqrt{b}}+(a^{e}-b)\tilde{b}_{k}\notag\\
		&-\frac{(2k-1)b^{e}_{k}}{2b}\bigg(2\sqrt{b}\sum_{j=0}^{k}C_{j}\tilde{b}_{j}\bigg)+O\bigg(b^{3}+b^{\frac{5}{4}}\sqrt{E}+\frac{1}{\sqrt{b}}\sum_{j=0}^{k}|\tilde{b}_{j}||\tilde{b}_{k}|\bigg)\notag\\
		=&(\tilde{b}_{k})_{s}+(2kb+a^{e}-b)\tilde{b}_{k}+\bigg(-\frac{b^{e}_{k}}{\sqrt{b}}-\frac{(2k-1)b^{e}_{k}}{\sqrt{b}}\bigg)\sum_{j=0}^{k}C_{j}\tilde{b}_{j}\notag\\
  &+O\bigg(b^{3}+b^{\frac{5}{4}}\sqrt{E}+\frac{1}{\sqrt{b}}\sum_{j=0}^{k}|\tilde{b}_{j}||\tilde{b}_{k}|\bigg)\notag\\
		=&(\tilde{b}_{k})_{s}+\frac{3}{2s}\tilde{b}_{k}+\frac{k+1}{C_{k}s}\sum_{j=0}^{k}C_{j}\tilde{b}_{j}+O\bigg(b^{3}+b^{\frac{5}{4}}\sqrt{E}+\frac{1}{\sqrt{b}}\sum_{j=0}^{k}|\tilde{b}_{j}||\tilde{b}_{k}|\bigg),
	\end{align*}
for $ j=k-1 $,
	\begin{align*}
		&(b_{k-1})_{s}+bb_{k-1}\lambda_{b,k-1}+(a-b)b_{k-1}+\frac{(2k-3)b_{k-1}}{2b}\Phi-(2k+1)\frac{A_{k-1}}{A_{k}}\frac{\Phi}{2b}b_{k}\notag\\
	=&(\tilde{b}_{k-1})_{s}+b\tilde{b}_{k-1}(2k-2+\mu_{k-1})+(a-a^{e})\tilde{b}_{k-1}+(a^{e}-b)\tilde{b}_{k-1}\notag\\
	&+\frac{2k-3}{2b}\tilde{b}_{k-1}\bigg(-2\sqrt{b}\sum_{j=0}^{k}C_{j}\tilde{b}_{j}+O(b^{\frac{5}{2}}+b^{\frac{3}{4}}\sqrt{E})\bigg)\notag\\
 &-\bigg(\frac{1}{2}+k\bigg)\frac{A_{k-1}}{A_{k}}\frac{-2\sqrt{b}\sum_{j=0}^{k}C_{j}\tilde{b}_{j}+O(b^{\frac{5}{2}}+b^{\frac{3}{4}}\sqrt{E})}{b}(b^{e}_{k}+\tilde{b}_{k})\notag\\
	=&(\tilde{b}_{k-1})_{s}+(2k-2)b\tilde{b}_{k-1}+(a^{e}-b)\tilde{b}_{k-1}+(2k+1)\frac{A_{k-1}}{A_{k}}\frac{b^{e}_{k}}{\sqrt{b}}\sum_{j=0}^{k}C_{j}\tilde{b}_{j}\notag\\
 &+O\bigg(b^{3}+b^{\frac{5}{4}}\sqrt{E}+\sum_{i,j=0}^{k}\frac{|\tilde{b}_{i}||\tilde{b}_{j}|}{\sqrt{b}}\bigg)\notag\\
	=&(\tilde{b}_{k-1})_{s}+\bigg(\frac{3}{2s}-\frac{1}{ks}\bigg)\tilde{b}_{k-1}-\frac{(k+1)(2k+1)}{2kC_{k}s}\frac{A_{k-1}}{A_{k}}\sum_{j=0}^{k}C_{j}\tilde{b}_{j}\notag\\
 &+O\bigg(b^{3}+b^{\frac{5}{4}}\sqrt{E}+\sum_{i,j=0}^{k}\frac{|\tilde{b}_{i}||\tilde{b}_{j}|}{\sqrt{b}}\bigg),
	\end{align*}

and for $ 0\leq{j}\leq{k-2} $,
	\begin{align*}
		&(b_{j})_{s}+bb_{j}\lambda_{b,j}+(a-b)b_{j}+\frac{(2j-1)b_{j}}{2b}\Phi-(2j+3)\frac{A_{j}}{A_{j+1}}\frac{\Phi}{2b}b_{j+1}\notag\\
		=&(\tilde{b}_{j})_{s}+b\tilde{b}_{j}(2j+\mu_{j})+(a-a^{e})\tilde{b}_{j}\notag\\
  &+(a^{e}-b)\tilde{b}_{j}+\frac{(2j-1)\tilde{b}_{j}}{2b}\bigg(-2\sqrt{b}\sum_{j=0}^{k}C_{j}\tilde{b}_{j}+O(b^{\frac{5}{2}}+b^{\frac{3}{4}}\sqrt{E})\bigg)\notag\\
		&-(2j+3)\frac{A_{j}}{A_{j-1}}\frac{-2\sqrt{b}\sum_{j=0}^{k}C_{j}\tilde{b}_{j}+O(b^{\frac{5}{2}}+b^{\frac{3}{4}}\sqrt{E})}{2b}\tilde{b}_{j+1}\notag\\
		=&(\tilde{b}_{j})_{s}+\bigg(\frac{j}{ks}+\frac{1}{2s}\bigg)\tilde{b}_{j}+O\bigg(b^{3}+b^{\frac{5}{4}}\sqrt{E}+\sum_{i,j=0}^{k}\frac{|\tilde{b}_{i}||\tilde{b}_{j}|}{\sqrt{b}}\bigg).
	\end{align*}
Finally together with (\ref{3.14}), (\ref{3.15}), (\ref{le3.6:1.6}) and the bootstrap assumption (\ref{3.24}), (\ref{3.25}), we obtain
\begin{equation*}
	\begin{aligned}
		&\bigg|(\tilde{b}_{j})_{s}+\frac{3}{2s}\tilde{b}_{k}+\frac{k+1}{C_{k}s}\sum_{j=0}^{k}C_{j}\tilde{b}_{j}\bigg|\\
  &+\bigg|(\tilde{b}_{k-1})_{s}+\bigg(\frac{3}{2s}-\frac{1}{ks}\bigg)\tilde{b}_{k-1}-\frac{(k+1)(2k+1)}{2kC_{k}s}\frac{A_{k-1}}{A_{k}}\sum_{j=0}^{k}C_{j}\tilde{b}_{j}\bigg|\\
		&+\bigg|(\tilde{b}_{j})_{s}+\bigg(\frac{j}{ks}+\frac{1}{2s}\bigg)\tilde{b}_{j}\bigg|\lesssim b^{3}+b^{\frac{5}{4}}\sqrt{E}+\sum_{i,j=0}^{k}\frac{|\tilde{b}_{i}||\tilde{b}_{j}|}{\sqrt{b}}\lesssim b^{\frac{5}{2}+2\alpha_{k}}+b^{\frac{5}{4}}\sqrt{E},
	\end{aligned}
\end{equation*}
which is (\ref{3.43}).
	\end{proof}
\section{Energy estimates}
\ 
\newline
In this section, we derive the energy estimate which is crucial in closing the bootstrap.
\begin{proposition}[Energy estimates]\label{Ee}
	Under the bootstrap assumption, there hold the pointwise estimates:
	\begin{itemize}
		\item Case $ k=0 $:
		\begin{equation}\label{3.7:0}
			\frac{1}{2}\frac{d}{ds}(E+\varphi)+2bE\lesssim b^{5},
		\end{equation}
	and here
	\begin{equation}\label{3.7:1}
		|\varphi(s)|\lesssim b^{4}.
	\end{equation}
		\item Case $ k\geq 1 $:
		\begin{equation}\label{pr3.7:2}
			\frac{d}{ds}(E+\varphi)+(5k+4)bE\lesssim \left\{\begin{aligned}
& b^{4}\quad 1\leq{k}\leq 4,\\
& b^{\frac{7}{2}+\frac{9}{4k}}\quad {k}\geq 5.
			\end{aligned}\right.
		\end{equation}
	and here 
	\begin{equation}\label{pr3.7:3}
		|\varphi(s)|\lesssim \left\{\begin{aligned}
			& b^{\frac{103}{32}} \quad 1\leq{k}\leq 4,\\
			& b^{\frac{7}{2}+\frac{9}{4k}} \quad k\geq 5.
		\end{aligned}\right.
	\end{equation}
	\end{itemize}
\end{proposition}
\begin{remark}
 The function $\varphi(s)$  comes from the boundary term in the energy identity and depends only on the parameters $(a,b,(\tilde{b}_{j})_{0\leq{j}\leq{k}})$ as we will see in the computation below. 
\end{remark}
\begin{remark}
	The sharp coercivity constant $ 5k+4 $ in (\ref{pr3.7:2}) is essential to close the energy bounds below and also the main reason why the energy is divided as $ k $ changes.
\end{remark}
\begin{proof}
	We compute the energy identity firstly and then estimate all the terms.\\
\textbf{Step 1: Energy identity.} 
		  Recalling  (\ref{3.6}), it follows from (\ref{3.40}) that
		\begin{equation*}
			\partial_{s}\varepsilon_{2}+\mathcal{H}_{a}\varepsilon_{2}=[\partial_{s},\mathcal{H}_{b}]\varepsilon+[\mathcal{H}_{a},\mathcal{H}_{b}]\varepsilon+\mathcal{H}_{b}\mathcal{F}.
		\end{equation*}
	Using (\ref{pr2.3:53}) and the identity
	\begin{equation*}
		[\partial_{s},\Delta]=0,
	\end{equation*}
we obtain 
\begin{equation*}
	\begin{aligned}
		[\partial_{s},\mathcal{H}_{b}]\varepsilon+[\mathcal{H}_{a},\mathcal{H}_{b}]\varepsilon&=[\partial_{s},-\Delta+b\Lambda]\varepsilon+[\mathcal{H}_{b}+(a-b)\Lambda,\mathcal{H}_{b}]\varepsilon\\
		&=\partial_{s}(b\Lambda\varepsilon)-b\Lambda(\partial_{s}\varepsilon)+(a-b)[\Lambda,-\Delta]\varepsilon\\
		&=\partial_{s}b\Lambda\varepsilon+2(a-b)\Delta\varepsilon\\
  &=(b_{s}+2(a-b)b)\Lambda\varepsilon-2(a-b)(-\Delta\varepsilon+b\Lambda\varepsilon)\\
		&=\Phi\Lambda\varepsilon-2(a-b)\varepsilon_{2}.
		\end{aligned}
\end{equation*}
As a result, we get the equation for $ \varepsilon_{2} $
\begin{equation*}
	\partial_{s}\varepsilon_{2}+\mathcal{H}_{a}\varepsilon_{2}=\Phi\Lambda\varepsilon-2(a-b)\varepsilon_{2}+\mathcal{H}_{b}\mathcal{F}.
\end{equation*}
We now compute the energy identity for $ \varepsilon_{2} $:
\begin{equation}\label{pr3.7:5}
	\begin{aligned}
		\frac{1}{2}\frac{d}{ds}E&=\frac{1}{2}\frac{d}{ds}\bigg(\int_{1}^{+\infty}\varepsilon_{2}^{2}e^{-\frac{by^{2}}{2}}y^{2}dy\bigg)=-\frac{b_{s}}{4}\int_{1}^{+\infty}y^{4}\varepsilon_{2}^{2}e^{-\frac{by^{2}}{2}}dy+(\partial_{s}\varepsilon_{2},\varepsilon_{2})_{b}\\
		&=-\frac{b_{s}}{4}\lVert y\varepsilon_{2}\rVert^{2}_{b}+\Phi(\Lambda\varepsilon,\varepsilon_{2})_{b}-2(a-b)\lVert\varepsilon_{2}\rVert^{2}_{b}+(\mathcal{H}_{b}\mathcal{F},\varepsilon_{2})_{b}-(\mathcal{H}_{a}\varepsilon_{2},\varepsilon_{2})_{b}.
	\end{aligned}
	\end{equation}
The term $-(\mathcal{H}_{a}\varepsilon_{2},\varepsilon_{2})_{b}$ in (\ref{pr3.7:5}) needs to be dealt with further. Note that
\begin{equation}\label{pr3.7:6}
		-(\mathcal{H}_{a}\varepsilon_{2},\varepsilon_{2})_{b}=-\int_{1}^{+\infty}(\mathcal{H}_{b}\varepsilon_{2})\varepsilon_{2}e^{-\frac{by^{2}}{2}}y^{2}dy-(a-b)\int_{1}^{+\infty}(\Lambda\varepsilon_{2})\varepsilon_{2}e^{-\frac{by^{2}}{2}}y^{2}dy.
\end{equation}
On one hand,
\begin{equation}\label{pr3.7:7}
	\begin{aligned}
		\int_{1}^{+\infty}(\mathcal{H}_{b}\varepsilon_{2})\varepsilon_{2}e^{-\frac{by^{2}}{2}}y^{2}dy&=\int_{1}^{+\infty}-\frac{1}{\rho_{b}y^{2}}\partial_{y}(\rho_{b}y^{2}\partial_{y}\varepsilon_{2})\varepsilon_{2}\rho_{b}y^{2}dy\\
		&=-\int_{1}^{+\infty}\partial_{y}(\rho_{b}y^{2}\partial_{y}\varepsilon_{2})\varepsilon_{2}dy=-\int_{1}^{+\infty}\varepsilon_{2}d(\rho_{b}y^{2}\partial_{y}\varepsilon_{2})\\
		&=-\bigg(\varepsilon_{2}\rho_{b}y^{2}\partial_{y}\varepsilon_{2}\bigg|^{+\infty}_{1}-\int_{1}^{+\infty}\partial_{y}\varepsilon_{2}\rho_{b}y^{2}\partial_{y}\varepsilon_{2}dy\bigg)\\
		&=\varepsilon_{2}(1)e^{-\frac{b}{2}}\partial_{y}\varepsilon_{2}(1)+\lVert\partial_{y}\varepsilon_{2}\rVert^{2}_{b}.
	\end{aligned}
\end{equation}
On the other hand,
\begin{equation}\label{pr3.7:8}
	\begin{aligned}
		(\Lambda\varepsilon_{2},\varepsilon_{2})_{b}&=\int_{1}^{+\infty}(\Lambda\varepsilon_{2})\varepsilon_{2}e^{-\frac{by^{2}}{2}}y^{2}dy=\int_{1}^{+\infty}y^{3}\varepsilon_{2}e^{-\frac{by^{2}}{2}}d\varepsilon_{2}\\
		&=-\frac{1}{2}e^{-\frac{b}{2}}\varepsilon^{2}_{2}(1)-\frac{3}{2}\lVert\varepsilon_{2}\rVert^{2}_{b}+\frac{b}{2}\lVert y\varepsilon_{2}\rVert^{2}_{b}.
		\end{aligned}
\end{equation}
Inserting  (\ref{pr3.7:6}), (\ref{pr3.7:7}) and  (\ref{pr3.7:8}) into (\ref{pr3.7:5}) yields
\begin{equation}\label{pr3.7:9}
	\begin{aligned}
		\frac{1}{2}\frac{d}{ds}E=-\frac{\Phi}{4}\lVert y\varepsilon_{2}\rVert^{2}_{b}+&\Phi(\Lambda\varepsilon,\varepsilon_{2})_{b}-\frac{1}{2}(a-b)\lVert\varepsilon_{2}\rVert_{b}^{2}-\lVert\partial_{y}\varepsilon_{2}\rVert^{2}_{b}+(\mathcal{H}_{b}\mathcal{F},\varepsilon_{2})_{b}\\
		&+\frac{1}{2}(a-b)e^{-\frac{b}{2}}\varepsilon^{2}_{2}(1)-\varepsilon_{2}(1)\partial_{y}\varepsilon_{2}(1)e^{-\frac{b}{2}}.
	\end{aligned}
\end{equation}
We now estimate  the terms on the right side of (\ref{pr3.7:9}) respectively.\\
 \textbf{Step 2: Nonlinear estimates.} From (\ref{3.28}), (\ref{3.30}), (\ref{le3.6:1.5}) and the coercivity estimate (\ref{leCH_{b}:16}),
\begin{equation}\label{pr3.7:10}
	\begin{aligned}
		|\Phi|\lVert y\varepsilon_{2}\rVert^{2}_{b}&\lesssim|\Phi|\bigg(\frac{1}{b^{2}}\lVert\partial_{y}\varepsilon_{2}\rVert^{2}_{b}+\frac{1}{b}(\varepsilon_{2}(1))^{2}\bigg)\\
		&=\frac{|\Phi|}{b^{2}}\lVert\partial_{y}\varepsilon_{2}\rVert^{2}_{b}+\frac{|\Phi|}{b}a^{2}(a-b)^{2}\lesssim \left\{\begin{aligned}
			&\sqrt{b}\lVert\partial_{y}\varepsilon_{2}\rVert^{2}_{b}+b^{\frac{13}{2}}\quad k=0,\\
			& b^{\frac{1}{8}}\lVert\partial_{y}\varepsilon_{2}\rVert^{2}_{b}+b^{5}\quad 1\leq{k}\leq 4,\\
			& b^{\frac{1}{2k}}\lVert\partial_{y}\varepsilon_{2}\rVert^{2}_{b}+b^{5}\quad {k}\geq 5.
		\end{aligned}\right.
	\end{aligned}
\end{equation}
From (\ref{3.26}), (\ref{le3.6:1.5}) and the coercivity estimate (\ref{leCH_{b}:14}), 
\begin{equation}\label{pr3.7:11}
	\begin{aligned}
		|\Phi(\Lambda\varepsilon,\varepsilon_{2})_{b}|\lesssim|\Phi|\lVert\Lambda\varepsilon\rVert_{b}\lVert\varepsilon_{2}\rVert_{b}\lesssim|\Phi|\frac{\sqrt{E}}{b}\sqrt{E}\lesssim \left\{\begin{aligned}
			&b^{\frac{3}{2}}\cdot b^{\frac{7}{2}}=b^{5} \quad k=0,\\
			& b^{\frac{9}{8}}\cdot b^{\frac{47}{16}}=b^{\frac{65}{16}}\quad 1\leq{k}\leq 4,\\
			& b^{\frac{5}{2}+\frac{7}{4k}}\cdot b^{1+\frac{1}{2k}}=b^{\frac{7}{2}+\frac{9}{4k}}\quad {k}\geq 5.
		\end{aligned}\right.
	\end{aligned}
\end{equation}
Using (\ref{pr2.4:1}) and the orthogonality condition (\ref{3.12}), one has 
\begin{equation*}
(\mathcal{H}_{b}{\rm{Mod}},\varepsilon_{2})_{b}=0.
\end{equation*}  
Hence together with (\ref{3.35}), (\ref{3.39}) and (\ref{le3.6:1.5}), we obtain
\begin{equation}\label{pr3.7:12}
	\begin{aligned}
		&|(\mathcal{H}_{b}\mathcal{F},\varepsilon_{2})_{b}|=|(\mathcal{H}_{b}(\Psi+{\rm{Mod}}),\varepsilon_{2})_{b}|=|(\mathcal{H}_{b}\Psi,\varepsilon_{2})_{b}|\lesssim \lVert\mathcal{H}_{b}\Psi\rVert_{b}\lVert\varepsilon_{2}\rVert_{b}\\
		&\lesssim\left\{\begin{aligned}
		&b\cdot\bigg(b^{\frac{9}{4}}+\sqrt{b}\sqrt{E}+\frac{|b_{s}|}{b^{\frac{1}{4}}}\bigg)\sqrt{E}\lesssim b^{\frac{13}{4}}\sqrt{E} \quad k=0,\\ &b\cdot\bigg(b^{\frac{7}{4}}+\frac{|\Phi|}{b^{\frac{1}{4}}}\bigg)\sqrt{E}\lesssim b^{\frac{11}{4}}\sqrt{E} \quad k\geq{1}.
				\end{aligned}\right.
		\end{aligned}
\end{equation}
From (\ref{3.26}) and (\ref{3.30}), for $ k=0 $,
\begin{equation}\label{pr3.7:13}
	|a-b|\lVert \varepsilon_{2}\rVert_{b}^{2}\lesssim b^{\frac{3}{2}}\cdot b^{\frac{7}{2}}= b^{5}.
\end{equation}
From (\ref{3.9}) and (\ref{le3.6:1.3}), for $ k\geq 1 $,
\begin{equation}\label{pr3.7:14}
	\frac{1}{2}(a-b)\lVert \varepsilon_{2}\rVert^{2}_{b}=\frac{1}{2}\bigg(a^{e}-\sum_{j=0}^{k}\frac{C_{j}\tilde{b}_{j}}{\sqrt{b}}+O\bigg(b^{\frac{3}{2}}+\frac{\sqrt{E}}{b^{\frac{1}{4}}}\bigg)-b\bigg)E=\frac{1}{4s}E+O\bigg(b^{1+\alpha_{k}}E+\frac{E^{\frac{3}{2}}}{b^{\frac{1}{4}}}\bigg).
\end{equation}
\textbf{Step 3: Boundary terms.} 
It remains to estimate the boundary term of (\ref{pr3.7:9}) and we argue differently depending on $ k $.\\
 \textbf{Case {\boldmath{$ k=0 $}}:} Firstly, from (\ref{3.26}), (\ref{3.29}), (\ref{3.30}) and (\ref{3.42}), we have
\begin{equation*}
	\begin{aligned}
		\partial_{y}\varepsilon_{2}(1)&=-a_{s}-\lambda_{b,0}b^{2}-(a-b)(a^{2}-a)+O(b^{3})\\
		&=-(a-b)_{s}-b_{s}-\lambda_{b,0}b^{2}+O(b^{\frac{5}{2}}+b^{\frac{3}{4}}\sqrt{E}+b^{3})\\
		&=-(a-b)_{s}+O(b^{\frac{5}{2}}).
	\end{aligned}
\end{equation*}
Hence together with (\ref{3.28}), (\ref{3.30})
	\begin{align}
		&e^{-\frac{b}{2}}\varepsilon_{2}(1)\bigg[\frac{1}{2}(a-b)\varepsilon_{2}(1)-\partial_{y}\varepsilon_{2}(1)\bigg]\notag\\
		=&-e^{-\frac{b}{2}}a(a-b)\bigg[-\frac{1}{2}(a-b)a(a-b)+(a-b)_{s}+O(b^{\frac{5}{2}})\bigg]\notag\\
		=&-e^{-\frac{b}{2}}a(a-b)(a-b)_{s}+O(b^{5})\notag\\
		=&-\frac{d}{ds}\bigg[e^{-\frac{b}{2}}\bigg(\frac{(a-b)^{3}}{3}+\frac{b(a-b)^{2}}{2}\bigg)\bigg]\label{pr3.7.3:1}\\
  &-\frac{b_{s}}{2}e^{-\frac{b}{2}}\frac{(a-b)^{3}}{3}+e^{-\frac{b}{2}}b_{s}\frac{(a-b)^{2}}{2}-\frac{b_{s}}{2}e^{-\frac{b}{2}}b\frac{(a-b)^{2}}{2}+O(b^{5})\notag\\
		=&-\frac{d}{ds}\bigg[e^{-\frac{b}{2}}\bigg(\frac{(a-b)^{3}}{3}+\frac{b(a-b)^{2}}{2}\bigg)\bigg]+O(b^{5})\notag.
	\end{align}

Set 
\begin{equation*}
	\varphi(s)=2e^{-\frac{b}{2}}\bigg(\frac{(a-b)^{3}}{3}+\frac{b(a-b)^{2}}{2}\bigg),
	\end{equation*}
and from (\ref{3.30}),
\begin{equation*}
	|\varphi(s)|\lesssim b^{4}.
\end{equation*}
Then  injecting (\ref{pr3.7:10}), (\ref{pr3.7:11}), (\ref{pr3.7:12}), (\ref{pr3.7:13}), (\ref{pr3.7.3:1}) into (\ref{pr3.7:9}) and combining with the coercivity estimate (\ref{leCH_{b}:15}) with $k=0 $  yields
\begin{equation*}
	\begin{aligned}
	\frac{1}{2}\frac{d}{ds}(E+\varphi)&\leq (-1+O(\sqrt{b}))(2+O(\sqrt{b}))bE+O(b^{5})
	=-2bE+O(b^{5})
	\end{aligned}
\end{equation*}
\textbf{Case {\boldmath{$ k\geq 1 $}}:}
Firstly  set 
\begin{equation*}
	\tilde{a}=a^{e}-\sum_{j=0}^{k}\frac{C_{j}}{\sqrt{b}}\tilde{b}_{j},
\end{equation*}
which is the leading order term of $ a $ (see (\ref{le3.6:1.3})). By direct computation,
\begin{equation*}
	\begin{aligned}
		&(\tilde{a})_{s}-\sqrt{b}\sum_{j=0}^{k}C_{j}\lambda_{b,j}b_{j}\\
  =&(a^{e})_{s}-\sum_{j=0}^{k}\frac{C_{j}}{\sqrt{b}}(\tilde{b}_{j})_{s}+\frac{1}{2}\sum_{j=0}^{k}C_{j}b^{-\frac{3}{2}}\tilde{b}_{j}b_{s}-\sqrt{b}\sum_{j=0}^{k}C_{j}\lambda_{b,j}b_{j}\\
	=&(a^{e})_{s}-2kC_{k}\sqrt{b}b^{e}_{k}-\mu_{k}C_{k}\sqrt{b}b^{e}_{k}-\sqrt{b}\sum_{j=0}^{k}C_{j}\lambda_{b,j}\tilde{b}_{j}\\
 &-\sum_{j=0}^{k}\frac{C_{j}}{\sqrt{b}}(\tilde{b}_{j})_{s}+\frac{1}{2}\sum_{j=0}^{k}C_{j}b^{-\frac{3}{2}}\tilde{b}_{j}b_{s}.
	\end{aligned}
\end{equation*}
From (\ref{3.9}), (\ref{3.14}), (\ref{3.43}) and the bootstrap assumption (\ref{3.24}), (\ref{3.25})
\begin{equation*}
		|(\tilde{b}_{j})_{s}|\lesssim b\sum_{j=0}^{k}|\tilde{b}_{j}|+b^{\frac{5}{2}+2\alpha_{k}}+b^{\frac{5}{4}}\sqrt{E}\lesssim b^{\frac{5}{2}+\alpha_{k}}+b^{\frac{5}{4}}\sqrt{E},\quad \forall\ 0\leq{j}\leq{k}.
\end{equation*}
Combining with (\ref{3.9}) and that $ |\mu_{k}|\lesssim\sqrt{b} $, we get
\begin{equation*}
	\begin{aligned}
	&(\tilde{a})_{s}-\sqrt{b}\sum_{j=0}^{k}C_{j}\lambda_{b,j}b_{j}\\
	=&-\frac{k+1}{2ks^{2}}-2kC_{k}\frac{1}{\sqrt{2ks}}\left(-\frac{1}{C_{k}}\right)\frac{k+1}{(2ks)^{\frac{3}{2}}}+O\bigg(b^{\frac{5}{2}}+b^{2+\alpha_{k}}+\frac{1}{\sqrt{b}}\cdot(b^{\frac{5}{2}+\alpha_{k}}+b^{\frac{5}{4}}\sqrt{E})
	\bigg)\\
	=&O(b^{2+\alpha_{k}}+b^{\frac{3}{4}}\sqrt{E}).\quad  \bigg({\rm{note\ that\ }} 0<\alpha_{k}<\frac{1}{2}\bigg)
		\end{aligned}
\end{equation*}
Hence, from (\ref{3.28}) and (\ref{3.29}), the boundary term is 
\begin{align}\label{pr3.7.4:4}
		&e^{-\frac{b}{2}}\varepsilon_{2}(1)\bigg(\frac{1}{2}(a-b)\varepsilon_{2}(1)-\partial_{y}\varepsilon_{2}(1)\bigg)\notag\\
		=&-e^{-\frac{b}{2}}a(a-b)\bigg(-\frac{1}{2}a(a-b)^{2}+a_{s}\notag\\
  &-\sqrt{b}\sum_{j=0}^{k}C_{j}\lambda_{b,j}b_{j}+(a-b)(a^{2}-a)+O(b^{\frac{5}{2}})\bigg)\notag\\
		=&-e^{-\frac{b}{2}}a(a-b)\bigg((a-\tilde{a})_{s}+\tilde{a}_{s}-\sqrt{b}\sum_{j=0}^{k}C_{j}\lambda_{b,j}b_{j}\notag\\
  &-\frac{1}{2}a(a-b)^{2}+(a-b)(a^{2}-a)+O(b^{\frac{5}{2}})\bigg)\notag\\
		=&-e^{-\frac{b}{2}}a(a-b)\bigg((a-\tilde{a})_{s}+O(b^{2+\alpha_{k}}+b^{\frac{3}{4}}\sqrt{E})+O(b^{2})\bigg)\notag\\
		=&-e^{-\frac{b}{2}}a(a-b)(a-\tilde{a})_{s}+O(b^{2}+b^{\frac{11}{4}}\sqrt{E})\notag\\
		=&-e^{-\frac{b}{2}}(a-\tilde{a})_{s}[(a-\tilde{a})^{2}+(a-\tilde{a})(2\tilde{a}-b)+\tilde{a}(\tilde{a}-b)]+O(b^{4}+b^{\frac{11}{4}}\sqrt{E})\notag\\
		=&-\frac{d}{ds}\biggl\{e^{-\frac{b}{2}}\bigg[\frac{(a-\tilde{a})^{3}}{3}+\frac{(a-\tilde{a})^{2}}{2}(2\tilde{a}-b)\notag\\
  &+(a-\tilde{a})\tilde{a}(\tilde{a}-b)\bigg]\biggr\}+O(b^{4}+b^{\frac{11}{4}}\sqrt{E}).
	\end{align}
Notice that the last identity is implied by the following estimates:
\begin{equation*}
	\begin{aligned}
		&|b_{s}||a-\tilde{a}|^{3}+|b_{s}||2\tilde{a}-b||a-\tilde{a}|^{2}+(a-\tilde{a})^{2}|2\tilde{a}_{s}-b_{s}|+|b_{s}||a-\tilde{a}||\tilde{a}||\tilde{a}-b|+\\
		&|a-\tilde{a}||\tilde{a}_{s}||\tilde{a}-b|+|a-\tilde{a}||\tilde{a}||(\tilde{a}-b)_{s}|\lesssim b^{4}+b^{\frac{11}{4}}\sqrt{E},
	\end{aligned}
\end{equation*}
which is a consequence of (\ref{3.9}), (\ref{le3.6:1.3}) and (\ref{3.43}).\\
Set 
\begin{equation*}
	\varphi(s)=2e^{-\frac{b}{2}}\bigg[\frac{(a-\tilde{a})^{3}}{3}+\frac{(a-\tilde{a})^{2}}{2}(2\tilde{a}-b)+(a-\tilde{a})\tilde{a}(\tilde{a}-b)\bigg].
\end{equation*}
	
From (\ref{3.9}), (\ref{le3.6:1.3}) and (\ref{3.26}), we have
\begin{equation*}
	|\varphi(s)|\lesssim b^{2}\cdot\bigg(b^{\frac{3}{2}}+\frac{\sqrt{E}}{b^{\frac{1}{4}}}\bigg)\lesssim \left\{\begin{aligned}
		& b^{2}\cdot\frac{(b^{\frac{47}{16}})^{\frac{1}{2}}}{b^{\frac{1}{4}}}=b^{\frac{103}{32}}\quad 1\leq{k}\leq{4},\\
		&b^{2}\cdot\frac{b^{\frac{1}{2}(\frac{5}{2}+\frac{7}{4k})}}{b^{\frac{1}{4}}}=b^{3+\frac{7}{8k}}\quad k\geq 5.
	\end{aligned}\right.
\end{equation*}
To proceed, injecting (\ref{3.26}), (\ref{pr3.7:10}), (\ref{pr3.7:11}), (\ref{pr3.7:12}), (\ref{pr3.7:14}), (\ref{pr3.7.4:4}) into (\ref{pr3.7:9}),\\
when $ 1\leq{k}\leq 4 $,
\begin{equation*}
	\begin{aligned}
		\frac{1}{2}\frac{d}{ds}(E+\varphi)&=O(b^{\frac{1}{8}}\lVert\partial_{y}\varepsilon_{2}\rVert^{2}_{b}+b^{5})+O(b^{\frac{65}{16}})-\frac{1}{4s}E+O\bigg(b^{1+\alpha_{k}}E+\frac{E^{\frac{3}{2}}}{b^{\frac{1}{4}}}\bigg)\\
		&+O(b^{4}+b^{\frac{11}{4}}\sqrt{E})-\lVert\partial_{y}\varepsilon_{2}\rVert_{b}^{2}\\
		&=(-1+O(b^{\frac{1}{8}}))\lVert\partial_{y}\varepsilon_{2}\rVert^{2}_{b}-\frac{1}{4s}E+O\bigg(b^{4}+b^{1+\frac{1}{8}}\cdot b^{\frac{47}{16}}+\frac{b^{\frac{47}{16}\cdot\frac{3}{2}}}{b^{\frac{1}{4}}}\bigg)\\
		&=(-1+O(b^{\frac{1}{8}}))\lVert\partial_{y}\varepsilon_{2}\rVert^{2}_{b}-\frac{1}{4s}E+O(b^{4}),
	\end{aligned}
\end{equation*}
then together with the coercivity estimate (\ref{leCH_{b}:15}) and that $ b(s)=\frac{1}{2ks} $, we get
\begin{equation*}
\begin{aligned}
	\frac{1}{2}\frac{d}{ds}(E+\varphi)&\leq (-1+O(b^{\frac{1}{8}}))(2k+2+O(\sqrt{b}))bE-c_{k}b^{2}|\varepsilon_{2}(1)|^{2})-\frac{1}{4s}E+O(b^{4})\\
	&\leq-\bigg(\frac{5}{2}k+2\bigg)bE+O(b^{4}).
\end{aligned}
\end{equation*}
When $ k\geq 5 $, similar computations as before yield 
\begin{equation*}
	\frac{d}{ds}(E+\varphi)+(5k+4)bE\lesssim b^{\frac{7}{2}+\frac{9}{4k}},
\end{equation*}
which is (\ref{pr3.7:2}).
	\end{proof}
 \section{Closing the bootstrap}
We are now in a position to  close the bootstrap bounds of Proposition \ref{pr3.3} and conclude the proof of Theorem \ref{th1}.\\
\subsection{Proof of Proposition \ref{pr3.3}.} Our goal in this subsection is to improve the bounds (\ref{3.21}) and (\ref{3.22}) for $ k=0 $ and similarly for the bounds (\ref{3.23}), (\ref{3.24}) and (\ref{3.25}) in the case $ k\geq 1 $. The improvement of the energy bounds follows from Proposition \ref{Ee}, while the unstable modes $ W_{k-1} $ and $ (U_{j})_{0\leq j\leq k-2} $ will be controlled  by a topological argument. We will distinguish the case $ k=0 $ and $ k\geq 1 $ in the following discussion.\\
\textbf{Step 1:} 
	\textbf{Case {\boldmath{$ k=0 $}}.} We claim that  for any $s\in(s_{0},s^{*}) $,
	\begin{equation}\label{3.5:1}
	0<e^{-Cb^{*}(s-s_{0})}b(s_{0})<b(s)<b(s_{0})<b^{*},
	\end{equation}
\begin{equation}\label{3.5:2}
	E(s)\leq C b^{4}(s),
\end{equation}
\begin{equation}\label{3.5:3}
	\lambda(s)\geq e^{-Cb^{*}(s-s_{0})},
\end{equation}
for some universal constant $ C>0 $ and $ 0<b^{*}\ll 1 $  which improve the bootstrap estimate (\ref{3.21}) and (\ref{3.22}). Firstly, since 
\begin{equation*}
	\bigg|b_{s}+\sqrt{\frac{2}{\pi}}b^{\frac{5}{2}}\bigg|\lesssim b^{3},
\end{equation*}
then taking $ b^{*}>0 $ small enough and by the bootstrap bound (\ref{3.22}), we have 
\begin{equation*}
	b_{s}<0,\quad \forall\ s\in[s_{0},s^{*}).
\end{equation*}
Hence \begin{equation*}
	b(s)<b(s_{0})<b^{*},\quad \forall\ s\in\left(s_{0},s^{*}\right).
\end{equation*}
By the definition of $ a $ and (\ref{3.30}), one sees that  for any $s\in[s_{0},s^{*}) $,
\begin{equation*}
	|\log \lambda(s)|=\bigg|\int_{s_{0}}^{s}a d\sigma\bigg|=\bigg|\int_{s_{0}}^{s}(a-b+b)d\sigma\bigg|=\bigg|\int_{s_{0}}^{s}b(1+O(\sqrt{b}))d\sigma\bigg|\leq Cb^{*}(s-s_{0}),
	\end{equation*}
which is (\ref{3.5:3}). From (\ref{3.42}) again, we estimate the term
\begin{equation*}
	\begin{aligned}
	\frac{d}{ds}\bigg(\frac{b^{4}}{\lambda^{4}}\bigg)=4b_{s}b^{3}\frac{1}{\lambda^{4}}-4\lambda_{s}b^{4}\frac{1}{\lambda^{5}}&=4\frac{b^{4}}{\lambda^{4}}\bigg(\frac{b_{s}}{b}-\frac{\lambda_{s}}{\lambda}\bigg)\\
	&=4\frac{b^{4}}{\lambda^{4}}\left(\frac{-\sqrt{\frac{2}{\pi}}b^{\frac{5}{2}}+O(b^{3})}{b}+b+a-b\right)\\
	&=4\frac{b^{4}}{\lambda^{4}}(b+O(b^\frac{3}{2}))=4\frac{b^{5}}{\lambda^{4}}(1+O(\sqrt{b}))>0,
	\end{aligned}
	\end{equation*}
where the last inequality holds for any $ b^{*}>0 $ small enough and $s\in [s_{0},s^{*})$.
Hence
\begin{equation}\label{3.5:5}
	\frac{b^{4}(s)}{\lambda^{4}(s)}\geq\frac{b^{4}(s_{0})}{\lambda^{4}(s_{0})}=b^{4}(s_{0})>0,
\end{equation}
which together with (\ref{3.5:3}) yields (\ref{3.5:1}). Finally, we estimate the energy, as the consequence of (\ref{3.7:0}) and
\begin{equation*}
	|b-a||E+\varphi|\lesssim b^{\frac{3}{2}}(b^{\frac{7}{2}}+b^{4})\lesssim b^{5},
\end{equation*}
which follows from (\ref{3.26}), (\ref{3.30}) and (\ref{3.7:1}), we have
\begin{equation*}
	\frac{d}{ds}(E+\varphi)+4a(E+\varphi)\leq Cb^{5}.
\end{equation*}
We integrate this inequality in time using (\ref{3.3}), (\ref{3.5:5}) and the initial bound (\ref{3.8}) to derive	
\begin{equation}\label{3.5:6}
	\begin{aligned}
		(E+\varphi)(s)&\leq e^{-\int_{s_{0}}^{s}4a(\tau)d\tau}(E+\varphi)(s_{0})+Ce^{-\int_{s_{0}}^{s}4a(\tau)d\tau}\bigg(\int_{s_{0}}^{s}e^{\int_{s_{0}}^{t}4a(\tau)d\tau}b^{5}(t)dt\bigg)\\
		&\leq C\lambda^{4}(s)b^{4}(s_{0})+C\lambda^{4}(s)\int_{s_{0}}^{s}\frac{b^{5}(t)}{\lambda^{4}(t)}dt\\
		&\leq Cb^{4}(s)+C\lambda^{4}(s)\int_{s_{0}}^{s}\frac{b^{5}(t)}{\lambda^{4}(t)}dt.
	\end{aligned}
\end{equation}
In order to estimate the second term on the right side,  using (\ref{3.42}),
\begin{equation*}
	\begin{aligned}
		&\int_{s_{0}}^{s}\frac{b^{5}(\sigma)}{\lambda^{4}(\sigma)}d\sigma=\int_{s_{0}}^{s}\frac{b^{4}}{\lambda^{4}}(a+b-a)d\sigma\leq \int_{s_{0}}^{s}\frac{b^{4}}{\lambda^{4}}\bigg(-\frac{\lambda_{\sigma}}{\lambda}\bigg)d\sigma+C\int_{s_{0}}^{s}\frac{b^{\frac{11}{2}}}{\lambda^{4}}d\sigma\\
		=&\frac{1}{4}\frac{b^{4}}{\lambda^{4}}\bigg|^{s}_{s_{0}}-\int_{s_{0}}^{s}b^{3}b_{\sigma}\lambda^{-4}d\sigma+C\int_{s_{0}}^{s}\frac{b^{\frac{11}{2}}}{\lambda^{4}}d\sigma\leq\frac{1}{4}\frac{b^{4}(s)}{\lambda^{4}(s)}+C\int_{s_{0}}^{s}\frac{b^{\frac{11}{2}}(\sigma)}{\lambda^{4}(\sigma)}d\sigma,
	\end{aligned}
\end{equation*}
notice that $ 0<b<b^{*} $ is small enough, and thus furthermore
\begin{equation*}
	\int_{s_{0}}^{s}\frac{b^{5}(\sigma)}{\lambda^{4}(\sigma)}d\sigma\leq C \frac{b^{4}(s)}{\lambda^{4}(s)}.
\end{equation*}
Then injecting into (\ref{3.5:6}) and combining with (\ref{3.7:1}) yields
\begin{equation*}
	E(s)\leq Cb^{4}(s),
	\end{equation*}
which is (\ref{3.5:2}).\\
\textbf{Step 2:} \textbf{Case {\boldmath{$ k\geq 1 $}}.}  
In this case, we firstly improve the energy bounds (\ref{3.23}). Since the energy estimates in (\ref{pr3.7:2}) depend on $ k $,  we distinguish $ k $ in the following argument.\\
For $ 1\leq k\leq 3 $, from (\ref{pr3.7:2}) and  $ b(s)=\frac{1}{2ks} $, we have
\begin{equation*}
	\frac{d}{ds}(E+\varphi)+\bigg(\frac{5}{2}+\frac{2}{k}\bigg)\frac{1}{s}(E+\varphi)\lesssim s^{-4},
\end{equation*}
which together with the initial assumptation (\ref{3.13}) yieids
\begin{equation*}
	\begin{aligned}
		(E+\varphi)(s)&\leq\frac{s_{0}^{\frac{5}{2}+\frac{2}{k}}}{s^{\frac{5}{2}+\frac{2}{k}}}(E+\varphi)(s_{0})+\frac{C\int_{s_{0}}^{s}t^{-\frac{3}{2}+\frac{2}{k}}dt}{s^{\frac{5}{2}+\frac{2}{k}}}\\
		&\lesssim \frac{s_{0}^{\frac{5}{2}+\frac{2}{k}}\cdot s_{0}^{-3}}{s^{\frac{5}{2}+\frac{2}{k}}}+\frac{s^{-\frac{1}{2}+\frac{2}{k}}}{s^{\frac{5}{2}+\frac{2}{k}}}\\
		&\lesssim \frac{s^{\frac{5}{2}+\frac{2}{k}-3}}{s^{\frac{5}{2}+\frac{2}{k}}}+s^{-3}\lesssim s^{-3},
	\end{aligned}
	\end{equation*}
then combining with (\ref{pr3.7:3}) that $ |\varphi(s)|\lesssim s^{-\frac{103}{32}} $, we know 
\begin{equation*}
	E(s)\lesssim s^{-3}\lesssim b^{3}.
\end{equation*}
For $ k=4 $, similar computaion as before yields 
\begin{equation*}
	\begin{aligned}
		(E+\varphi)(s)&\leq \frac{s_{0}^{3}}{s^{3}}(E+\varphi)(s_{0})+\frac{C\int_{s_{0}}^{s}t^{-1}dt}{s^{3}}\\
		&\lesssim \frac{s_{0}^{3}s_{0}^{-3}\log s_{0}}{s^{3}}+\frac{C\log s}{s^{3}}\\
		&\lesssim \frac{\log s}{s^{3}},
	\end{aligned}
	\end{equation*}
then from (\ref{pr3.7:3}), we get
\begin{equation*}
	E(s)\lesssim\frac{\log s}{s^{3}}\lesssim b^{3}|\log b|.
\end{equation*}
For $ k\geq 5 $,
\begin{equation*}
	\begin{aligned}
	(E+\varphi)(s)&\leq \frac{s_{0}^{\frac{5}{2}+\frac{2}{k}}}{s^{\frac{5}{2}+\frac{2}{k}}}(E+\varphi)(s_{0})+\frac{C\int_{s_{0}}^{s}t^{-1-\frac{1}{4k}}dt}{s^{\frac{5}{2}+\frac{2}{k}}}\\
	&\lesssim \frac{1}{s^{\frac{5}{2}+\frac{2}{k}}},
	\end{aligned}
	\end{equation*}
which together with (\ref{pr3.7:3}) again yields
\begin{equation*}
	E(s)\lesssim \frac{1}{s^{\frac{5}{2}+\frac{2}{k}}}\lesssim b^{\frac{5}{2}+\frac{2}{k}}.
	\end{equation*}
To summarize, the above argument improve the bootstrap bound (\ref{3.23}) since here we can find a universal constant $ C $ to replace the large $ K $ to be determined. We now inject these estimates into (\ref{3.43}) and conclude
\begin{equation}\label{3.5:7}
	\begin{aligned}
		&\bigg|(\tilde{b}_{k})_{s}+\frac{3}{2s}\tilde{b}_{k}+\frac{k+1}{C_{k}s}\sum_{j=0}^{k}C_{j}\tilde{b}_{j}\bigg|\\
  &+\bigg|(\tilde{b}_{k-1})_{s}+\bigg(\frac{3}{2s}-\frac{1}{ks}\bigg)\tilde{b}_{k-1}-\frac{(k+1)(2k+1)}{2kC_{k}s}\frac{A_{k-1}}{A_{k}}\sum_{j=0}^{k}C_{j}\tilde{b}_{j}\bigg|\\
		&+\bigg|(\tilde{b}_{j})_{s}+\bigg(\frac{j}{ks}+\frac{1}{2s}\bigg)\tilde{b}_{j}\bigg|\\
  \lesssim& b^{\frac{5}{4}}\sqrt{E}+b^{\frac{5}{2}+2\alpha_{k}}\lesssim\left\{\begin{aligned}
			&b^{\frac{5}{4}}\cdot b^{\frac{3}{2}}+b^{\frac{5}{2}+2\alpha_{k}}\lesssim b^{\frac{11}{4}}\quad 1\leq k \leq 3,\\
			&b^{\frac{5}{4}}\cdot b^{\frac{3}{2}}\sqrt{|\log b|}+b^{\frac{5}{2}+2\alpha_{k}}\lesssim b^{\frac{11}{4}}\sqrt{|\log b|}\quad k=4,\\
			&b^{\frac{5}{4}}\cdot b^{\frac{5}{4}+\frac{1}{k}}++b^{\frac{5}{2}+2\alpha_{k}}\lesssim b^{\frac{5}{2}+\frac{1}{k}}\quad k\geq 5,
		\end{aligned}\right.\\
	\lesssim& b^{\frac{5}{2}+\beta_{k}},
	\end{aligned}
	\end{equation}
where 
\begin{equation*}
	\beta_{k}=\left\{\begin{aligned}
		&\frac{1}{4}\quad 1\leq k \leq 3,\\
		&\frac{3}{16}\quad k=4,\\
		&\frac{1}{k} \quad k\geq 5.
	\end{aligned}\right.
\end{equation*}
Next, injecting $ \tilde{b}_{j}=\frac{U_{j}}{s^{\frac{3}{2}+\alpha_{k}}} $ into (\ref{3.5:7}), 
%and note that $ (\tilde{b}_{j})_{s}=\frac{(U_{j})_{s}}{s^{\frac{3}{2}+\alpha_{k}}}-(\frac{3}{2}+\alpha_{k})\frac{U_{j}}{s^{\frac{5}{2}+\alpha_{k}}} $, 
we obtain
\begin{equation}\label{3.5:8}
	\begin{aligned}
		&\bigg|(U_{k})_{s}-\frac{\alpha_{k}}{s}U_{k}+\frac{k+1}{C_{k}s}\sum_{j=0}^{k}C_{j}U_{j}\bigg|\\
  &+\bigg|(U_{k-1})_{s}-\bigg(\frac{1}{k}+\alpha_{k}\bigg)\frac{1}{s}U_{k-1}-\frac{(k+1)(2k+1)}{2kC_{k}s}\frac{A_{k-1}}{A_{k}}\sum_{j=0}^{k}C_{j}U_{j}\bigg|\\
		&+\bigg|(U_{j})_{s}+\bigg(\frac{j}{k}-1-\alpha_{k}\bigg)\frac{1}{s}U_{j}\bigg|\\
  \lesssim& s^{\frac{3}{2}+\alpha_{k}}\cdot s^{-\frac{5}{2}-\beta_{k}}=s^{-1+\alpha_{k}-\beta_{k}}=s^{-1-\gamma_{k}},
	\end{aligned}
\end{equation}
where
\begin{equation*}
	\gamma_{k}=\left\{\begin{aligned}
&\frac{1}{8}\quad 1\leq k\leq 3,\\
&\frac{1}{16}\quad k=4,\\
&\frac{1}{2k} \quad k\geq 5.
	\end{aligned}\right.
\end{equation*}
From the bootstrap assumption (\ref{3.25})
\begin{equation*}
	|U_{j}(s)|\lesssim\delta\quad \forall 0\leq{j}\leq{k-2},
\end{equation*}
Taking this into (\ref{3.5:8}) yields
\begin{equation*}
	\begin{aligned}
		&\bigg|(U_{k})_{s}-\frac{\alpha_{k}}{s}U_{k}+\frac{k+1}{C_{k}s}C_{k}U_{k}+\frac{k+1}{C_{k}s}C_{k-1}U_{k-1}\bigg|+\\
		&\bigg|(U_{k-1})_{s}-\bigg(\frac{1}{k}+\alpha_{k}\bigg)\frac{1}{s}U_{k-1}-\frac{(k+1)(2k+1)}{2kC_{k}s}\frac{A_{k-1}}{A_{k}}C_{k-1}U_{k-1}\\
  &-\frac{(k+1)(2k+1)}{2kC_{k}}\frac{A_{k-1}}{A_{k}}C_{k}U_{k}\bigg|\lesssim\delta s^{-1}+s^{-1-\gamma_{k}}.
	\end{aligned}
\end{equation*}
Recall by the definition of $ B_{k} $ (see (\ref{3.16})), the above inequality is equivalent to
\begin{equation*}
	\partial_{s}\begin{pmatrix}
		U_{k}\\
		U_{k-1}
	\end{pmatrix}=\frac{1}{s}B_{k}\begin{pmatrix}
	U_{k}\\
	U_{k-1}
\end{pmatrix}+O(\delta s^{-1}+s^{-1-\gamma_{k}}).
	\end{equation*}
Then from (\ref{3.18}), we have
\begin{equation}\label{3.5:9}
	\bigg|(W_{k})_{s}-\frac{\mu^{k}_{2}}{s}W_{k}\bigg|+\bigg|(W_{k-1})_{s}-\frac{\mu^{k}_{1}}{s}W_{k-1}\bigg|\lesssim \delta s^{-1}+s^{-1-\gamma_{k}},
\end{equation}
where $ \mu^{k}_{2}<0<\mu^{k}_{1} $ are the eigenvalues of $ B_{k} $. The first term of the above inequality yields the control of the stable direction $ W_{k} $. In fact, we integrate the first bound of (\ref{3.5:9}) and use the initial assumption (\ref{3.19}) to obtain
\begin{equation*}
	\begin{aligned}
		|W_{k}(s)|&\leq \bigg(\frac{s_{0}}{s}\bigg)^{-\mu^{k}_{2}}|W_{k}(s_{0})|+\frac{C\delta}{s^{-\mu^{k}_{2}}}\int_{s_{0}}^{s}t^{-1}t^{-\mu^{k}_{2}}dt+\frac{C}{s^{-\mu^k_{2}}}\int_{s_{0}}^{s}t^{-1-\gamma_{k}-\mu^{k}_{2}}dt\\
		&\leq 1+C\delta+Cs^{-\gamma_{k}}.
	\end{aligned}
\end{equation*}
Hence taking $ \delta>0 $ small enough and $ s_{0}>0 $ large enough which only depend on $ C $, we conclude that
\begin{equation*}
	|W_{k}(s)|\leq\frac{3}{2},\quad \forall\ s\in[s_{0},s^{*}).
\end{equation*}
Finally, in order to close the bootstrap, we argue by contradiction and assume that 
\begin{equation*}
	{\rm{for}}\ {\rm{any}}\ {\rm{initial}}\ {\rm{data}}\ \bigg(W_{k-1}(s_{0}),\frac{U_{0}(s_{0})}{\delta},\cdot\cdot\cdot,\frac{U_{k-2}(s_{0})}{\delta}\bigg)\in B_{k}(2),\quad s^{*}<+\infty.
\end{equation*}
By the definition of $ s^{*} $, we arrive at
\begin{equation}\label{3.5:11}
	|W_{k-1}(s^{*})|^{2}+\sum_{j=0}^{k-2}\frac{|U_{j}(s^{*})|^{2}}{\delta^{2}}=4.
\end{equation}
We then compute from (\ref{3.5:8}), (\ref{3.5:9}) the strict outgoing condition defined by (\ref{3.5:11}) at the exit time $ s^{*} $:
\begin{equation*}
	\begin{aligned}
		\frac{1}{2}\frac{d}{ds}\bigg[W^{2}_{k-1}+\sum_{j=0}^{k-2}\frac{U^{2}_{j}}{\delta^{2}}\bigg]\bigg|_{s=s^{*}}=&W_{k-1}(W_{k-1})_{s}+\sum_{j=0}^{k-2}\frac{U_{j}(U_{j})_{s}}{\delta^{2}}\bigg|_{s=s^{*}}\\
		=&W_{k-1}\bigg(\frac{\mu^{k}_{1}}{s}W_{k-1}+O(\delta s^{-1}+s^{-1-\gamma_{k}})\bigg)\\
  &+\sum_{j=0}^{k-2}\frac{U_{j}\bigg(\bigg(\alpha_{k}+1-\frac{j}{k}\bigg)\frac{1}{s}U_{j}+O(s^{-1-\gamma_{k}})\bigg)}{\delta^{2}}\Bigg|_{s=s^{*}}\\
		\geq&\frac{C^{*}}{s^{*}}+O\bigg(\frac{\delta}{s^{*}}+\frac{1}{{s^{*}}^{(1+\gamma_{k})}}+\frac{1}{\delta}\frac{1}{{s^{*}}^{(1+\gamma_{k})}}
		\bigg),
	\end{aligned}
	\end{equation*}
here $ C^{*}>0 $  is a fixed constant. Now we firstly choose $ \delta>0 $ small enough and for such $ \delta>0 $ we choose $ s_{0}>0 $ large enough from which we  conclude that
\begin{equation*}
	\frac{1}{2}\frac{d}{ds}\bigg[W^{2}_{k-1}+\sum_{j=0}^{k-2}\frac{U^{2}_{j}}{\delta^{2}}\bigg]\bigg|_{s=s^{*}}>0.
\end{equation*}
This implies from a standard argument the continuity of the map
\begin{equation*}
	\bigg(W_{k-1}(s_{0}),\frac{U_{0}(s_{0})}{\delta},\cdot\cdot\cdot,\frac{U_{k-2}(s_{0})}{\delta}\bigg)\mapsto s^{*}\bigg(W_{k-1}(s_{0}),\frac{U_{0}(s_{0})}{\delta},\cdot\cdot\cdot,\frac{U_{k-2}(s_{0})}{\delta}\bigg)\in{\RR}
\end{equation*}
and hence the continuity of the map
\begin{equation*}
	\begin{aligned}
		B_{k}(2)&\rightarrow B_{k}(2)\\
		\bigg(W_{k-1},\frac{U_{0}}{\delta},\cdot\cdot\cdot,\frac{U_{k-2}}{\delta}\bigg)(s_{0})&\mapsto \bigg(W_{k-1},\frac{U_{0}}{\delta},\cdot\cdot\cdot,\frac{U_{k-2}}{\delta}\bigg)(s^{*}).
		\end{aligned}
\end{equation*}
Notice that this continuous map sends $B_{k}(2)$ onto its boundary and is the identity when restricted to the boundary $ S_{k-1}(2) $ which is a contradiction to the Brouwer's fixed point theorem. This concludes the proof for $ k\geq 1 $.
\begin{remark}
	Note that the initial data $ (W_{k-1}(s_{0}),U_{0}(s_{0}),\cdot\cdot\cdot,U_{k-2}(s_{0})) $ are by construction lying on a nonlinear manifold of codimension $k$ . The fact that such initial data forms a Lipschitz manifold in the $ H^{2} $ topology reduces to a local uniqueness of the flow $ \varepsilon(0)\mapsto (W_{k-1}(\varepsilon(0)),U_{0}(\varepsilon(0)),\cdot\cdot\cdot,U_{k-2}(\varepsilon(0))) $. Such a uniqueness problem has already been solved in other more complicated equations(see for example \cite{C} for the wave equation and \cite{JW},\cite{MRP} for the Sch${\rm{\ddot{o}}}$rdinger equation), and completely analogous approach can be applied here. We therefore omit the details.
\end{remark}
{\it\bf Step 3: Global $ H^{2} $ control.} From the proof of Step 1 and Step 2, one  sees that the solution remains in the bootstrap regime as long as it exists in $ H^{2} $, hence in order to show that $ s^{*}=+\infty $, we only need to prove that for any $  s\geq s_{0} $,
\begin{equation*}
	\lambda(s)>0,
\end{equation*}
and
\begin{equation*}
	\lVert u(s)\rVert_{L^{2}(|x|\geq\lambda(s))}+\lVert\nabla u(s)\rVert_{L^{2}(|x|\geq\lambda(s))}+\lVert \Delta u(s)\rVert_{L^{2}(|x|\geq\lambda(s))}<+\infty.
\end{equation*}
For the positivity of $ \lambda(s) $, the case $ k=0 $ follows from the bound (\ref{3.5:3}), while the case $ k\geq 1 $ follows from that $ a=-\frac{\lambda_{s}}{\lambda} $ and (\ref{le3.6:1.3}). The global $ L^{2} $-bound follows from the dissipation law satisfied by the solutions of (\ref{eq:1}):
\begin{equation*}
	\frac{1}{2}\frac{d}{dt}\lVert u\rVert^{2}_{L^{2}(|x|\geq\lambda(t))}+\lVert\nabla u\rVert^{2}_{L^{2}(|x|\geq\lambda(t))}=0,
\end{equation*}
which yields 
\begin{equation*}
 \lVert u(s)\rVert_{L^{2}(|x|\geq\lambda(s))}\leq\lVert u(s_{0})\rVert_{L^{2}(|x|\geq\lambda(s_{0}))}<+\infty, \quad \forall\ s\in[s_{0},+\infty).
\end{equation*} 
The global $ \dot{H}^{1} $-bound follows from the energy identity of (\ref{eq:2})
\begin{equation*}
	\frac{1}{2}\frac{d}{dt}\lVert \nabla u\rVert^{2}_{L^{2}(|x|\geq \lambda(t))}+\lVert\Delta u\rVert^{2}_{(L^{2}(|x|\geq\lambda(t)))}=2\pi(\lambda(t))^{2}(\dot{\lambda}(t))^{3},
\end{equation*}
which together with $ \dot{\lambda}<0 $ implies that
\begin{equation}\label{6.1}
	\lVert \nabla u(s)\rVert_{L^{2}(|x|\geq\lambda(s))}\leq\lVert \nabla u(s_{0})\rVert_{L^{2}(|x|\geq\lambda(s_{0}))}<+\infty, \quad \forall\ s\in[s_{0},+\infty).
\end{equation}
For the global $ \dot{H}^{2} $-bound, we take a cut-off function $ \chi\geq 0 $ with $ \chi\equiv 0 $ for $ r\leq 1 $ and $ \chi\equiv 1 $ for $ r\geq 2 $. Firstly, we claim that
\begin{equation*}
	\lVert \Delta u(s)\rVert_{L^{2}(\lambda(s)\leq r\leq2)}\leq C(s)<+\infty, \quad \forall\ s\geq s_{0}.
\end{equation*}
In fact, since 
\begin{equation*}
 u(t,r)=v(s,y),\ y=\frac{r}{\lambda(t)} ,
\end{equation*}
and the exponetial weight is uniformly bounded from below and above in $ \big\{1\leq y\leq\frac{2}{\lambda(t)}\big\} $, we have
\begin{equation*}
	\begin{aligned}
	\lVert \Delta u\rVert_{L^{2}\big(\lambda(t)\leq r\leq 2\big)}&=\frac{1}{\lambda^{\frac{1}{2}}(t)}\lVert\Delta v\rVert_{L^{2}\big(1\leq y\leq\frac{2}{\lambda(t)}\big)}\\
		&\lesssim \frac{1}{\lambda^{\frac{1}{2}}(t)}\bigg(\lVert\Delta Q_{\beta}\rVert_{L^{2}\big(1\leq y\leq\frac{2}{\lambda(t)}\big)}+\lVert\Delta\varepsilon\rVert_{L^{2}\big(1\leq y\leq\frac{2}{\lambda(t)}\big)}\bigg)\\
		&<C(s)<+\infty \quad\quad \forall\ s\geq s_{0}.
		\end{aligned}
\end{equation*}
Then for the bound in the regime $ r\geq2 $, we compute
\begin{equation}\label{6.2}
	\begin{aligned}
		\frac{1}{2}\frac{d}{ds}\int \chi|\Delta u|^{2}&=\bigg(\frac{1}{2}\frac{d}{dt}\int \chi|\Delta u|^{2}\bigg)\frac{dt}{ds}=\lambda^{2}\int \chi\Delta\partial_{t}u\Delta u=\lambda^{2}\int \chi\Delta^{2}u\Delta u\\
		&=\lambda^{2}\bigg(-\int \chi|\nabla\Delta u|^{2}+\frac{1}{2}\int \Delta\chi|\Delta u|^{2}\bigg)\leq C\lambda^{2}\lVert \Delta u\rVert^{2}_{L^{2}(\lambda(s)\leq|x|\leq 2)}.
	\end{aligned}
\end{equation}
Hence integrating in time the above inequality, we arrive at
\begin{equation*}
	\int \chi|\Delta u|^{2}(s)\leq \int \chi|\Delta u(s_{0})|^{2}+C\int_{s_{0}}^{s}\lambda^{2}\lVert\Delta u\rVert^{2}_{L^{2}(\lambda(\sigma)\leq |x|\leq 2)}d\sigma<+\infty, 
\end{equation*}
for each $s\in\left[s_{0},+\infty\right)$. This concludes the proof of Proposition \ref{pr3.3}.
\subsection{Proof of Theorem \ref{th1}}
We are now in a position to conclude the proof of Theorem \ref{th1}.\\ 
{\it\bf Step 1: Finite time melting.} We claim that the solution  constructed before melts in finite time and then compute the related melting rates. The proof relies on the reintegration of the modulation equations. We distinguish the case $ k=0 $ and $ k\geq 1 $ in the following discussion.\\
	\textbf{Case {\boldmath{$ k=0 $}}.} From (\ref{3.42}), we have
	\begin{equation}\label{3.6:1}
		b_{s}+\sqrt{\frac{2}{\pi}}b^{\frac{5}{2}}=O(b^{3}).
		\end{equation}
	Since from Proposition \ref{pr3.3} $ b>0 $, we multiply by $ b^{-\frac{5}{2}} $ on both sides of (\ref{3.6:1}) and obtain
	\begin{equation*}
		\frac{b_{s}}{b^{\frac{5}{2}}}=-\sqrt{\frac{2}{\pi}}+O(\sqrt{b}).
		\end{equation*}
	Integrating with respect to $ s $ on both sides yields
	\begin{equation}\label{3.6:2}
		b(s)=\bigg(b^{-\frac{3}{2}}(s_{0})+\frac{3}{2}\sqrt{\frac{2}{\pi}}(s-s_{0})+O\bigg(\int_{s_{0}}^{s}\sqrt{b}d\tau\bigg)\bigg)^{-\frac{2}{3}}.
	\end{equation}
From Proposition \ref{pr3.3} again, since $ 0<b(s)<b^{*} $ small enough, we get
\begin{equation}\label{6.6}
	b(s)\sim s^{-\frac{2}{3}}.
\end{equation}
Then injecting into (\ref{3.6:2}), one has
\begin{equation*}
	\bigg|b(s)-\bigg(\frac{3}{2}\sqrt{\frac{2}{\pi}}s\bigg)^{-\frac{2}{3}}\bigg|\lesssim \frac{b^{-\frac{3}{2}}(s_{0})+O\big(\int_{s_{0}}^{s}\sqrt{b}d\tau \big)}{s^{\frac{5}{3}}}\lesssim s^{-1}.
\end{equation*}
Recall also that
\begin{equation}\label{3.6:5}
	-\frac{\lambda_{s}}{\lambda}=a=b+O(b^{\frac{3}{2}})=\bigg(\frac{3}{2}\sqrt{\frac{2}{\pi}}s\bigg)^{-\frac{2}{3}}+O(s^{-1}),
\end{equation}
which implies
\begin{equation}\label{3.6:6}
	-\log \lambda(s)=3 \bigg(\frac{3}{2}\sqrt{\frac{2}{\pi}}\bigg)^{-\frac{2}{3}}s^{\frac{1}{3}}+O(\log s).
\end{equation}
Hence
\begin{equation*}
	\lambda(s)=e^{-3 (\frac{3}{2}\sqrt{\frac{2}{\pi}})^{-\frac{2}{3}}s^{\frac{1}{3}}+O(\log s)},
\end{equation*}
from which we deduce that
\begin{equation*}
	T=\int_{s_{0}}^{+\infty}\lambda^{2}(s)ds<+\infty.
\end{equation*}
We now compute the melting rate, from (\ref{3.6:6}) again, we have
\begin{equation}\label{6.7}
	\bigg|s^{-\frac{1}{3}}-3\frac{\big(\frac{3}{2}\sqrt{\frac{2}{\pi}}\big)^{-\frac{2}{3}}}{\log \frac{1}{\lambda}}\bigg|\lesssim \frac{\log s}{s^{\frac{2}{3}}}\lesssim \frac{\log |\log \lambda|}{|\log \lambda|^{2}}.
\end{equation}
    Next, injecting into (\ref{3.6:5}) yields
\begin{equation*}
	\begin{aligned}
		-\lambda\lambda_{t}=-\frac{\lambda_{s}}{\lambda}&=\bigg(\frac{3}{2}\sqrt{\frac{2}{\pi}}s\bigg)^{-\frac{2}{3}}+O(s^{-1})\\
		&=9\frac{\big(\frac{3}{2}\sqrt{\frac{2}{\pi}}\big)^{-2}}{(\log \lambda)^{2}}+O\bigg(\frac{1}{|\log \lambda|^{3}}+\frac{\log|\log\lambda|}{|\log\lambda|^{3}}\bigg),
	\end{aligned}
\end{equation*}
and thus
\begin{equation*}
	-(\lambda^{2})_{t}(\log \lambda^{2})^{2}=16\pi+O\bigg(\frac{\log|\log\lambda^{2}|}{|\log \lambda^{2}|}\bigg).
\end{equation*}
Integrating from $ t $ to $ T $ with $ \lambda(T)=0 $ yields
\begin{equation*}
	\lambda^{2}(t)=16\pi\frac{T-t}{|\log (T-t)|^{2}}[1+o_{t\rightarrow T}(1)].
\end{equation*}
\textbf{Case {\boldmath{$ k\geq 1 $}}.} From (\ref{3.9}), (\ref{3.15}), (\ref{le3.6:1.3})  and the bootstrap bounds (\ref{3.23}), (\ref{3.24}), (\ref{3.25}), we have
\begin{equation*}
	-\frac{\lambda_{s}}{\lambda}=a=a^{e}-\sum_{j=0}^{k}\frac{C_{j}}{\sqrt{b}}\tilde{b}_{j}+O\bigg(b^{\frac{3}{2}}+\frac{\sqrt{E}}{b^{\frac{1}{4}}}\bigg)=\frac{k+1}{2ks}+O\bigg(\frac{1}{s^{1+\alpha_{k}}}\bigg).
\end{equation*}
Integrating in  time  and there exists a constant $ c(u_{0}) $ such that
\begin{equation*}
	-\log \lambda(s)=\frac{k+1}{2k}\log s+c(u_{0})+o_{s\rightarrow+\infty}(1),
\end{equation*}
which is equivalent to 
\begin{equation}\label{3.6:7}
	\lambda(s)=c^{*}(u_{0})(1+o_{s\rightarrow+\infty}(1))s^{-\frac{k+1}{2k}}.
	\end{equation}
Hence
\begin{equation*}
	T=\int_{s_{0}}^{+\infty}\lambda^{2}(\sigma)d\sigma<+\infty.
\end{equation*}
Moreover, we have
\begin{equation*}
	\begin{aligned}
	T-t=\int_{s}^{+\infty}\lambda^{2}(\sigma)d\sigma&=\int_{s}^{+\infty}{c^{*}(u_{0})}^{2}(1+o_{\sigma\rightarrow+\infty}(1))^{2}\sigma^{-\frac{k+1}{k}}d\sigma\\
	&={c^{*}(u_{0})}^{2}ks^{-\frac{1}{k}}(1+o_{s\rightarrow+\infty}(1)),
	\end{aligned}
\end{equation*}
from which we deduce
\begin{equation}\label{6.11}
	\frac{1}{s}=\bigg[\frac{(1+o_{t\rightarrow T}(1))(T-t)}{{c^{*}(u_{0})}^{2}k}\bigg]^{k}.
\end{equation}
Then injecting into (\ref{3.6:7}), we obtain
\begin{equation*}
	\lambda(t)=c(u_{0},k)(1+o_{t\rightarrow T}(1))(T-t)^{\frac{k+1}{2}}.
	\end{equation*}
{\it\bf Step 2: Non-concentration of energy.} Take $R>0$ and a cut-off function
	\begin{equation*}
	\chi_{R}(x)=\chi\left(\frac{x}{R}\right)=
	\left\{\begin{aligned}
		&0 \qquad {\rm{for}}\ x\leq R,\\
		&1 \qquad {\rm{for}}\ x\geq 2R.
	\end{aligned}\right.
\end{equation*}
Then from the equation (\ref{eq:2}) and integrating by parts, we have
\begin{equation}\label{6.3}
\begin{aligned}
\frac{1}{2}\frac{d}{dt}\int \chi_{R}|\nabla u|^{2}dx&=\int \chi_{R}\nabla u\cdot\nabla u_{t}dx=\int \chi_{R}\nabla u\cdot\nabla\Delta u dx\\
&=-\int (\nabla \chi_{R}\cdot\nabla u)\Delta u dx-\int \chi_{R}|\Delta u|^{2}dx\\
&=-\int \chi_{R}|\Delta u|^{2}dx+\frac{1}{2}\int |\nabla u|^{2}r\frac{\partial}{\partial r}\left(\frac{\chi^{'}_{R}}{r}\right)dx,
\end{aligned}
\end{equation}
which combines with (\ref{6.1}) yields
\begin{equation}\label{6.4}
\int_{0}^{T}\left(\int \chi_{R}|\Delta u|^{2}dx\right)dt<+\infty,\quad \forall R>0.
\end{equation}
Hence for any $ 0<\tau<T-t$,
\begin{equation*}
\begin{aligned}
\int \chi_{R}|\nabla u(t+\tau)-\nabla u(t)|^{2}dx&=\int \chi_{R}\left|\int_{t}^{t+\tau}(\partial_{t}\nabla u)(\sigma,x)d\sigma\right|^{2}dx\\
&\leq \int \chi_{R}\left(\tau \int_{t}^{t+\tau}|\partial_{\sigma}\nabla u|^{2}d\sigma\right)dx\\
&\leq \tau \int_{0}^{T}\left(\int \chi_{R}|\partial_{t}\nabla u|^{2}dx\right)dt\\
&=\tau \int_{0}^{T}\left(\int \chi_{R}|\Delta\nabla u|^{2}dx\right)dt\\
&\leq C(R)\tau,
\end{aligned}
\end{equation*}
where the last line follows by integrating (\ref{6.2}) with $\chi=\chi_{R}$ and then using (\ref{6.4}). As a result, for any fixed $R>0$, $\nabla u(t,x)$ is a Cauchy sequence in $L^{2}(|x|\geq 2R)$ as $t\rightarrow T$ and there exists $u^{*}$ such that
\begin{equation}\label{6.5}
    \forall\ R>0,\quad \nabla u(t)\rightarrow\nabla u^{*}\quad {\rm{in}}\ L^{2}(|x|\geq 2R)\ {\rm{as}}\ t\rightarrow T.
\end{equation}
Moreover,  the uniform bound of $\lVert\nabla u(t)\rVert_{L^{2}}$ (see (\ref{6.1})) ensures that
\begin{equation}\label{6.8}
    u^{*}\in \dot{H}^{1}\quad {\rm{and}}\quad \nabla u(t)\rightharpoonup \nabla u^{*}\quad L^{2}(|x|\geq \lambda(t))\ {\rm{as}}\ t\rightarrow T.
\end{equation}
Pick now
\begin{equation}\label{6.9}
    R(t)=\lambda(t)B(t),
\end{equation}
where
\begin{equation}\label{6.10}
B(t)=\sqrt{\frac{|\log b|}{2b}}.
\end{equation}
Since $0<b\ll 1$, we have $B(t)\gg 1$. Then from (\ref{6.3}),
\begin{equation*}
    \left|\frac{d}{d\tau}\int \chi_{R(t)}|\nabla u(\tau)|^{2}dx\right|\lesssim\frac{1}{R^{2}(t)}\int |\nabla u(\tau)|^{2}dx+\int \chi_{R(t)}|\Delta u(\tau)|^{2}dx.
\end{equation*}
Integrating the above inequality over $[t,T)$ and combining with (\ref{6.1}), (\ref{6.5}) gives
\begin{equation}\label{6.12}
    \left|\int \chi_{R(t)}|\nabla u(t)|^{2}dx-\int \chi_{R(t)}|\nabla u^{*}|^{2}dx\right|\lesssim \frac{T-t}{R^{2}(t)}+\int_{t}^{T}\int_{|x|\geq \lambda(t)}|\Delta u(\tau)|^{2}dxd\tau.
\end{equation}
Next, we claim that
\begin{equation*}
    \frac{T-t}{R^{2}(t)}\rightarrow 0\quad {\rm{as}}\ t\rightarrow T.
\end{equation*}
For $k=0$, from (\ref{6.6}) and (\ref{6.7}), we get
\begin{equation*}
    b\sim\frac{1}{|\log \lambda|^{2}}.
    \end{equation*}
    Then combining with (\ref{th1.1}), (\ref{6.9}) and (\ref{6.10}), we arrive at
    \begin{equation*}
        \begin{aligned}
        \frac{T-t}{R^{2}(t)}=\frac{T-t}{\lambda^{2}(t)}\frac{2b}{|\log b|}&\sim \frac{|\log (T-t)|^{2}}{\log|\log \lambda|}\frac{2}{|\log \lambda|^{2}}\\
        &\sim\frac{|\log (T-t)|^{2}}{\log|\log \lambda|}\frac{1}{\left|\log \frac{T-t}{|\log (T-t)|^{2}}\right|^{2}}\\
        &\lesssim \frac{1}{\log|\log \lambda|}\rightarrow 0 \quad {\rm{as}}\ t\rightarrow T.
        \end{aligned}
    \end{equation*}
    For $k\geq 1$, from (\ref{th1.2}),  (\ref{3.9}), (\ref{3.6:7}), (\ref{6.11}), (\ref{6.9}) and (\ref{6.10}), one has
    \begin{equation*}
        \begin{aligned}
            \frac{T-t}{R^{2}(t)}=\frac{T-t}{\lambda^{2}(t)}\frac{2b}{|\log b|}\sim \frac{(T-t)(T-t)^{k}}{(T-t)^{k+1}|\log b|}\lesssim \frac{1}{|\log b|}\rightarrow 0\quad {\rm{as}}\ t\rightarrow T.
        \end{aligned}
    \end{equation*}
    In the above argument, one can also infer that
    \begin{equation*}
        R(t)\rightarrow 0\quad {\rm{as}}\ t\rightarrow T.
    \end{equation*}
    Hence, setting $t\rightarrow T$ on both sides of (\ref{6.12}), we conclude that
    \begin{equation}\label{6.13}
        \lim_{t\rightarrow T}\int \chi_{R(t)}|\nabla u(t)|^{2}dx=\int |\nabla u^{*}|^{2}.
    \end{equation}
    Finally, we claim that
    \begin{equation}\label{6.14}
        \lim_{t\rightarrow T}\int_{|x|\geq\lambda(t)}(1-\chi_{R(t)})|\nabla u(t)|^{2}dx=0,
    \end{equation}
    which together with (\ref{6.13}) and (\ref{6.8}) yields (\ref{th1.3}). For $k=0$, from (\ref{pr2.4:2}), (\ref{pr2.4:3}), (\ref{3.5}), (\ref{3.21}), (\ref{leCH_{b}:14}), (\ref{6.9}) and (\ref{6.10}), we obtain
    \begin{align*}
        \int_{|x|\geq \lambda(t)}(1-\chi_{R(t)})|\nabla u(t)|^{2}dx&\leq \int_{\lambda(t)\leq |x|\leq 2R(t)}|\nabla u(t)|^{2}dx\\
        &=\lambda(t)\int_{1\leq |y|\leq 2B(t)}|\nabla v(t,y)|^{2}dy\\
        &\lesssim \lambda(t)\int_{1\leq y\leq 2B(t)}(b^{3}|\partial_{y}\eta_{b,0}|^{2}y^{2}+|\partial_{y}\varepsilon|^{2}y^{2})dy\\
        &\lesssim \lambda(t)e^{2bB^{2}}\left(b^{3}\lVert\partial_{y}\eta_{b,0}\rVert_{b}^{2}+\lVert\partial_{y}\varepsilon\rVert_{b}^{2}\right)\\
        &\lesssim \lambda(t)\frac{1}{b(t)}\left(b(t)^{3}\frac{1}{b(t)}+b^{3}(t)\right)\\
        &\lesssim \lambda(t)b(t)\rightarrow 0\quad {\rm{as}}\ t\rightarrow T,
    \end{align*}
    and for $k\geq 1$, similar computation gives
    \begin{equation*}
        \int_{|x|\geq \lambda(t)}(1-\chi_{R(t)})|\nabla u(t)|^{2}dx\lesssim b^{\frac{1}{2}}(t)\lambda(t)\rightarrow 0\quad {\rm{as}}\ t\rightarrow T,
    \end{equation*}
then (\ref{6.14}) is proved. This concludes the proof of Theorem \ref{th1}.
\section*{Acknowledgments}
The author is grateful to Prof. Lifeng Zhao for the constant support and many helpful discussions.  C. Zhang is supported by NSFC Grant of China No. 12271497 and No. 12341102.
    \appendix
 \section{Coercivity estimates}
 We now state the main coercivity property at the heart of the energy estimate. The proof is a perturbation consequence of the harmonic oscillator estimates (\ref{eq:12}) and (\ref{pr2.3.7}) which is similar to that in (\cite{MP})
 \begin{lemma}[Coercivity of $H_{b}$]\label{leCH_{b}}
 	Let $ k\in\NN $ and $ 0<b<b^{*}(k) $ small enough. Suppose $ u\in H^{3}_{\rho,b} $ satisfy
 	\begin{equation*}
 		u(\sqrt{b})=0, \quad \langle u,\psi_{b,j}\rangle_{b}=0, \quad \forall\ 0\leq{j}\leq{k}
 	\end{equation*}
 Then
 \begin{equation}\label{leCH_{b}:1}
 	 \lVert H_{b}u\rVert^{2}_{L^{2}_{\rho,b}}\gtrsim \lVert \Delta u\rVert^{2}_{L^{2}_{\rho,b}}+\lVert (1+z)\partial_{z}u\rVert^{2}_{L^{2}_{\rho,b}}+\lVert(1+z)u\rVert^{2}_{L^{2}_{\rho,b}}+b|\partial_{z}u(\sqrt{b})|^{2}.
 \end{equation}
Moreover, there exists a constant $ c_{k}>0 $ such that
\begin{equation}\label{leCH_{b}:2}
	\lVert\partial_{z}H_{b}u(z)\rVert^{2}_{L^{2}_{\rho,b}}\geq(2k+2+O(\sqrt{b}))\lVert H_{b}u(z)\rVert^{2}_{L^{2}_{\rho,b}}-c_{k}b^{\frac{5}{2}}|H_{b}u(\sqrt{b})|^{2},
	\end{equation}
and
\begin{equation}\label{leCH_{b}:3}
	\lVert zH_{b}u(z)\rVert^{2}_{L^{2}_{\rho,b}}\lesssim\lVert\partial_{z}H_{b}u(z)\rVert^{2}_{L^{2}_{\rho,b}}+b^{\frac{3}{2}}|H_{b}u(\sqrt{b})|^{2}.
\end{equation}
 \end{lemma}
\begin{proof}
	\textbf{Step 1:} \textbf{Proof of (\ref{leCH_{b}:1}).}
	Pick a small constant $ \delta>0 $ to be determined later, we compute
	\begin{equation*}
		\lVert (H_{b}-\delta)u\rVert^{2}_{L^{2}_{\rho,b}}=\langle(H_{b}-\delta)u,(H_{b}-\delta)u\rangle_{b}=\lVert H_{b}u\rVert^{2}_{L^{2}_{\rho,b}}-2\delta\langle H_{b}u,u\rangle_{b}+\delta^{2}\lVert u\rVert^{2}_{L^{2}_{\rho,b}},
		\end{equation*}
	hence from $ u(\sqrt{b})=0 $ and integrating by parts,
	\begin{equation}\label{leCH_{b}:4}
		\begin{aligned}
			\lVert H_{b}u\rVert^{2}_{L^{2}_{\rho,b}}&=\lVert (H_{b}-\delta)u\rVert^{2}_{L^{2}_{\rho,b}}+2\delta\langle H_{b}u,u\rangle_{b}-\delta^{2}\lVert u\rVert^{2}_{L^{2}_{\rho,b}}\\
			&=\lVert (H_{b}-\delta)u\rVert^{2}_{L^{2}_{\rho,b}}+2\delta\lVert\partial_{z}u\rVert^{2}_{L^{2}_{\rho,b}}-\delta^{2}\lVert u\rVert^{2}_{L^{2}_{\rho,b}}.
			\end{aligned}
	\end{equation}
By spectral gap estimate (\ref{pr2.3.7}), we have 
\begin{equation*}
	\lVert\partial_{z}u\rVert^{2}_{L^{2}_{\rho,b}}\geq(2k+2+O(\sqrt{b}))\lVert u\rVert^{2}_{L^{2}_{\rho,b}}.
\end{equation*}
Taking $ \delta>0 $ small enough and inserting into (\ref{leCH_{b}:4}) yields
\begin{equation*}
	\lVert H_{b}u\rVert^{2}_{L^{2}_{\rho,b}}\gtrsim \lVert(H_{b}-\delta)u\rVert^{2}_{L^{2}_{\rho,b}}+\lVert\partial_{z}u\rVert^{2}_{L^{2}_{\rho,b}}+\lVert u\rVert^{2}_{L^{2}_{\rho,b}}.
\end{equation*}
Moreover, from Lemma \ref{le:2}
\begin{equation}\label{leCH_{b}:5}
	\lVert H_{b}u\rVert^{2}_{L^{2}_{\rho,b}}\gtrsim \lVert(H_{b}-\delta)u\rVert^{2}_{L^{2}_{\rho,b}}+\lVert\partial_{z}u\rVert^{2}_{L^{2}_{\rho,b}}+\lVert (1+z)u\rVert^{2}_{L^{2}_{\rho,b}}.
\end{equation}
By direct computation,
\begin{equation}\label{leCH_{b}:6}
	\lVert H_{b}u\rVert^{2}_{L^{2}_{\rho,b}}=\langle (-\Delta+\Lambda)u,(-\Delta+\Lambda)u\rangle_{b}=\lVert \Delta u\rVert^{2}_{L^{2}_{\rho,b}}+\lVert\Lambda u\rVert^{2}_{L^{2}_{\rho,b}}-2\langle\Lambda u,\Delta u\rangle_{b},
\end{equation}
and note that
\begin{equation*}
	\begin{aligned}
	\langle \Lambda u,\Delta u\rangle_{b}&=\int_{\sqrt{b}}^{+\infty}\Lambda u\Delta u\rho z^{2}dz=\int_{\sqrt{b}}^{+\infty}z\partial_{z}u\bigg(\partial_{zz}u+\frac{2}{z}\partial_{z}u\bigg)\rho z^{2}dz\\
	&=2\lVert\partial_{z}u\rVert^{2}_{L^{2}_{\rho,b}}+\int_{\sqrt{b}}^{+\infty}\frac{1}{2}\partial_{z}((\partial_{z}u)^{2})\rho z^{3}dz\\
	&=2\lVert\partial_{z}u\rVert^{2}_{L^{2}_{\rho,b}}+\frac{1}{2}(\partial_{z}u)^{2}\rho z^{3}\bigg|^{+\infty}_{\sqrt{b}}-\int_{\sqrt{b}}^{+\infty}\bigg(\frac{3}{2}\rho z^{2}-\frac{1}{2}\rho z^{4}\bigg)(\partial_{z}u)^{2}dz\\
	&=\frac{1}{2}\lVert \partial_{z}u\rVert^{2}_{L^{2}_{\rho,b}}+\frac{1}{2}\lVert \Lambda u\rVert^{2}_{L^{2}_{\rho,b}}-\frac{1}{2}e^{-\frac{b}{2}}b^{\frac{3}{2}}(\partial_{z}u(\sqrt{b}))^{2}.
\end{aligned}
\end{equation*}
Then inserting into (\ref{leCH_{b}:6}) yields
\begin{equation}\label{leCH_{b}:7}
	\lVert H_{b}u\rVert^{2}_{L^{2}_{\rho,b}}=\lVert\Delta u\rVert^{2}_{L^{2}_{\rho,b}}-\lVert \partial_{z}u\rVert^{2}_{L^{2}_{\rho,b}}+e^{-\frac{b}{2}}b^{\frac{3}{2}}(\partial_{z}u(\sqrt{b}))^{2}.
	\end{equation}
As a result, from (\ref{leCH_{b}:5}) and (\ref{leCH_{b}:7})
\begin{equation*}
	\lVert \Delta u\rVert^{2}_{L^{2}_{\rho,b}}\lesssim\lVert H_{b}u\rVert^{2}_{L^{2}_{\rho,b}},
\end{equation*}
and from (\ref{leCH_{b}:6}) together with Cauchy-Schwartz inequality, we can also infer that
\begin{equation*}
	\lVert\Lambda u\rVert^{2}_{L^{2}_{\rho,b}}\lesssim\lVert H_{b}u\rVert^{2}_{L^{2}_{\rho,b}}.
\end{equation*}
To estimate the term $ b|\partial_{z}u(\sqrt{b})|^{2} $,  firstly recall (\ref{le1.14}) that
\begin{equation*}
	\langle H_{b}u,\frac{1}{z}\rangle_{b}=e^{-\frac{b}{2}}\sqrt{b}\partial_{z}u(\sqrt{b})-\int_{\sqrt{b}}^{+\infty}\rho zudz.
\end{equation*}
Hence
\begin{equation*}
	\begin{aligned}
		\sqrt{b}|\partial_{z}u(\sqrt{b})|&\lesssim \lVert u\rVert_{L^{2}_{\rho,b}}+\bigg|\langle H_{b}u,\frac{1}{z}\rangle_{b}\bigg|\\
		&\lesssim\lVert u\rVert_{L^{2}_{\rho,b}}+\lVert H_{b}u\rVert_{L^{2}_{\rho,b}}\lesssim\lVert H_{b}u\rVert_{L^{2}_{\rho,b}}.
	\end{aligned}
\end{equation*}
where  (\ref{leCH_{b}:5}) has been used. Now (\ref{leCH_{b}:1}) follows.\\
\textbf{Step 2:} \textbf{Proof of (\ref{leCH_{b}:2}) and (\ref{leCH_{b}:3}).}
Firstly, define the radial function
\begin{equation}\label{leCH_{b}:9}
	v(z)=\left\{\begin{aligned}
	&H_{b}u(\sqrt{b})\quad 0\leq{z}<\sqrt{b},\\
	&H_{b}u(z)\quad z\geq\sqrt{b},
	\end{aligned}\right.
\end{equation}
and note that $ v\in H^{1}_{\rho,0} $. Consider the function
\begin{equation*}
	\omega=v-\sum_{j=0}^{k}\frac{\langle v,P_{j}\rangle_{0}}{\langle P_{j},P_{j}\rangle_{0}}P_{j}.
\end{equation*}
Then $ \langle\omega,P_{j}\rangle_{0}=0,\quad \forall\ 0\leq{j}\leq{k} $  and by spectral gap estimate (\ref{eq:12}), we deduce 
\begin{equation}\label{leCH_{b}:10}
	\lVert \partial_{z}w\rVert^{2}_{L^{2}_{\rho,0}}\geq(2k+2)\lVert \omega\rVert^{2}_{L^{2}_{\rho,0}}\geq(2k+2)\int_{\sqrt{b}}^{+\infty}|\omega|^{2}e^{-\frac{z^{2}}{2}}z^{2}dz,
\end{equation}
Note also that 
\begin{equation*}
	\begin{aligned}
		|\langle v,P_{j}\rangle_{0}|&=\bigg|\bigg(\int_{0}^{\sqrt{b}}P_{j}e^{-\frac{z^{2}}{2}}z^{2}dz\bigg)H_{b}u(\sqrt{b})+\langle H_{b}u,P_{j}\rangle_{b}\bigg|\\
		&\lesssim b^{\frac{3}{2}}|H_{b}u(\sqrt{b})|+|\langle H_{b}u, P_{j}+\sqrt{b}\psi_{b,j}-\sqrt{b}\psi_{b,j}\rangle_{b}|\\
		&=b^{\frac{3}{2}}|H_{b}u(\sqrt{b})|+|\langle H_{b}u, P_{j}+\sqrt{b}\psi_{b,j}\rangle_{b}|\\
		&\leq b^{\frac{3}{2}}|H_{b}u(\sqrt{b})|+\lVert H_{b}u\rVert_{L^{2}_{\rho,b}}\lVert P_{j}+\sqrt{b}\psi_{b,j}\rVert_{L^{2}_{\rho,b}}\\
		&\lesssim b^{\frac{3}{2}}|H_{b}u(\sqrt{b})|+\sqrt{b}\lVert H_{b}u\rVert_{L^{2}_{\rho,b}},
	\end{aligned}
\end{equation*}
where in the last line we have used (\ref{pr2.3:48}). Hence, by direct computation,
%As a result, 
%\begin{equation}\label{leCH_{b}:12}
%\begin{aligned}
	%\lVert v-\omega\rVert_{H^{1}_{\rho,0}}=\bigg\lVert %\sum_{j=0}^{k}\frac{\langle v,P_{j}\rangle_{0}}{\langle %P_{j},P_{j}\rangle_{0}}P_{j}\bigg\rVert_{H^{1}_{\rho,0}}
 %&\lesssim\sum_{j=0}^{k}|\langle v,P_{j}\rangle_{0}|\\
 %&\lesssim b^{\frac{3}{2}}|H_{b}u(\sqrt{b})|+\sqrt{b}\lVert %H_{b}u\rVert_{L^{2}_{\rho,b}}.
% \end{aligned}
%\end{equation}

\begin{align*}
\lVert\partial_{z}\omega\rVert^{2}_{L^{2}_{\rho,0}}=&\int_{0}^{+\infty}\left|\partial_{z}v-\sum_{j=0}^{k}\frac{\langle v,P_{j}\rangle_{0}}{\langle P_{j},P_{j}\rangle_{0}}\partial_{z}P_{j}\right|^{2}\rho z^{2}dz\\
=&\int_{\sqrt{b}}^{+\infty}\left|\partial_{z}\left(H_{b}u\right)\right|^{2}+\sum_{j=0}^{k}\int_{0}^{+\infty}\left|\frac{\langle v,P_{j}\rangle_{0}}{\langle P_{j},P_{j}\rangle_{0}}\right|^{2}\left|\partial_{z}P_{j}\right|^{2}\rho z^{2}dz\\
&+\sum_{i\neq j}\frac{\langle v,P_{j}\rangle_{0}\langle v,P_{i}\rangle_{0}}{\langle P_{i},P_{i}\rangle_{0}\langle P_{j},P_{j}\rangle_{0}}\int_{0}^{+\infty}\partial_{z}P_{j}\partial_{z}P_{i}\rho z^{2}dz\\
&-2\sum_{j=0}^{k}\int_{\sqrt{b}}^{+\infty}\partial_{z}v\frac{\langle v,P_{j}\rangle_{0}}{\langle P_{j},P_{j}\rangle_{0}}\partial_{z}P_{j}\rho z^{2}dz\\
=&\lVert\partial_{z}H_{b}u\rVert^{2}_{L^{2}_{\rho,b}}+O\left(b^{\frac{5}{2}}|H_{b}u(\sqrt{b})|^{2}+\sqrt{b}\lVert H_{b}u\rVert^{2}_{L^{2}_{\rho,b}}\right).
\end{align*}
Similar computation gives
\begin{equation*}
    \int_{\sqrt{b}}^{+\infty}|\omega|^{2}e^{-\frac{z^{2}}{2}}z^{2}dz=\lVert H_{b}u\rVert^{2}_{L^{2}_{\rho,b}}+O\left(b^{\frac{5}{2}}|H_{b}u(\sqrt{b})|^{2}+\sqrt{b}\lVert H_{b}u\rVert^{2}_{L^{2}_{\rho,b}}\right).
\end{equation*}
Then injecting the above two equalities into (\ref{leCH_{b}:10}),  we obtain
\begin{equation*}
	\lVert \partial_{z}H_{b}u\rVert^{2}_{L^{2}_{\rho,b}}\geq(2k+2+O(\sqrt{b}))\lVert H_{b}u(z)\rVert^{2}_{L^{2}_{\rho,b}}-c_{k}b^{3}|H_{b}u(\sqrt{b})|^{2},
\end{equation*}
which is (\ref{leCH_{b}:2}).\\
From Lemma \ref{le:2} and (\ref{leCH_{b}:2}), 
\begin{equation*}
	\begin{aligned}
		\lVert zH_{b}u\rVert^{2}_{L^{2}_{\rho,b}}&\lesssim \lVert zv\rVert^{2}_{L^{2}_{\rho,0}}+\int_{0}^{\sqrt{b}}z^{2}|H_{b}u(\sqrt{b})|^{2}\rho z^{2}dz\\
		&\lesssim\lVert zv\rVert^{2}_{L^{2}_{\rho,0}}+b^{\frac{5}{2}}|H_{b}u(\sqrt{b})|^{2}\\
		&\lesssim\lVert\partial_{z}v\rVert^{2}_{L^{2}_{\rho,0}}+\lVert v\rVert^{2}_{L^{2}_{\rho,0}}+b^{\frac{5}{2}}|H_{b}u(\sqrt{b})|^{2}\\
		&\lesssim\lVert\partial_{z}v\rVert^{2}_{L^{2}_{\rho,b}}+\lVert v\rVert^{2}_{L^{2}_{\rho,b}}+b^{\frac{3}{2}}|H_{b}u(\sqrt{b})|^{2}\\
		&\lesssim\lVert\partial_{z}H_{b}u\rVert^{2}_{L^{2}_{\rho,b}}+b^{\frac{3}{2}}|H_{b}u(\sqrt{b})|^{2},	
	\end{aligned}
\end{equation*}
which is (\ref{leCH_{b}:3}).
	\end{proof}
We now renormalise Lemma  {\ref{leCH_{b}}} by setting $ \varepsilon(y)=u(\sqrt{b}y) $ and together with (\ref{pr2.4:10}) which yields:
\begin{lemma}[Coercivity of $ \mathcal{H}_{b} $]\label{leCH{b}}
	Let $ k\in\NN $ and $ 0<b<b^{*}(k) $ small enough. If $ \varepsilon\in H^{3}_{b} $ satisfy
	\begin{equation*}
		\varepsilon(1)=0,\quad (\varepsilon,\eta_{b,j})_{b}=0,\quad \forall\ 0\leq{j}\leq{k},
	\end{equation*}
then 
\begin{equation}\label{leCH_{b}:14}
	\lVert \mathcal{H}_{b}\varepsilon\rVert^{2}_{b}\gtrsim \lVert \Delta\varepsilon\rVert^{2}_{b}+b\lVert\partial_{y}\varepsilon\rVert^{2}_{b}+b^{2}\lVert\Lambda\varepsilon\rVert^{2}_{b}+b^{2}\lVert\varepsilon\rVert^{2}_{b}+b^{3}\lVert y\varepsilon\rVert^{2}_{b}+\sqrt{b}(\partial_{y}\varepsilon(1))^{2}.
\end{equation}
Moreover, there exists a constant $ c_{k}>0 $ satisfy
\begin{equation}\label{leCH_{b}:15}
	\lVert \partial_{y}\mathcal{H}_{b}\varepsilon(y)\rVert^{2}_{b}\geq(2k+2+O(\sqrt{b}))b\lVert\mathcal{H}_{b}\varepsilon\rVert^{2}_{b}-c_{k}b^{2}|\mathcal{H}_{b}\varepsilon(1)|^{2},
\end{equation}
and
\begin{equation}\label{leCH_{b}:16}
	b\lVert y\mathcal{H}_{b}\varepsilon\rVert^{2}_{b}\lesssim \frac{\lVert \partial_{y}\mathcal{H}_{b}\varepsilon\rVert^{2}_{b}}{b}+|\mathcal{H}_{b}\varepsilon(1)|^{2}.
\end{equation}
\end{lemma}
\section{Cauchy theory for the radial Stefan problem}
In this section, we establish the Cauchy theory for the radial Stefan problem in 3D. Since the similar results have already been proved in (\cite{MP}) in 2D and the methods there used can also been applied here,  we  only briefly recall the main steps and will not give so many details in the following argument.\\
Taking into account that the problem  is a free boundary problem, we firstly  convert it to a fixed boundary problem. Hence we pull back the problem (\ref{eq:2}) onto the fixed domain $\Omega:=\{\vec{y}\in\RR^{3}|\ y=|\vec{y}|\geq 1\}$ and define the radial pull-back temperature function $\omega:\Omega\rightarrow\RR$ by
\begin{equation*}
    \omega(t,y)=u(t,\lambda(t)y)\quad {\rm{and}}\quad y=|\vec{y}|.
\end{equation*}
Inserting into (\ref{eq:2}) gives the equations for $\omega$:
\begin{equation}\label{B.2}
	\left\{\begin{aligned}
&\partial_{t}\omega-\frac{\dot{\lambda}}{\lambda}\Lambda\omega-\frac{1}{\lambda^{2}}\Delta\omega=0\quad {\rm{in}}\ \Omega,\\
&\partial_{y}\omega(t,1)=-\dot{\lambda(t)}\lambda(t),\\
&\omega(t,1)=0,\\
&\omega(0,\cdot)=\omega_{0},\quad \lambda(0)=\lambda_{0}.
	\end{aligned}\right.
\end{equation}
Next, direct computation together with the equations (\ref{B.2}) which is similar to the two dimensional case in \cite{MP} yields the following energy identities.
\begin{lemma}[Energy identities]
  Let $(\omega,\lambda)$ be a solution to the fixed boundary problem (\ref{B.2}) on the interval $[0,T]$ and assume that $\lambda(0)>0$, $\omega_{0}|_{y=1}=0$, $\omega(t,\cdot)\in H^{2}(\Omega)$ for all $t\in[0,T]$. Then the following energy identities hold on $[0,T]$:
  \begin{equation}\label{B.3}
      \frac{1}{2}\frac{d}{dt}\lVert\omega\rVert^{2}_{L^{2}(\Omega)}+\frac{1}{\lambda^{2}}\lVert \nabla \omega\rVert^{2}_{L^{2}(\Omega)}=-\frac{3}{2}\frac{\dot{\lambda}}{\lambda}\lVert \omega\rVert^{2}_{L^{2}(\Omega)},
  \end{equation}
  \begin{equation}\label{B.4}
      \frac{1}{2}\frac{d}{dt}\lVert \nabla\omega\rVert^{2}_{L^{2}(\Omega)}+\frac{1}{\lambda^{2}}\lVert\Delta\omega\rVert^{2}_{L^{2}(\Omega)}=-\frac{1}{2}\frac{\dot{\lambda}}{\lambda}\lVert \nabla\omega\rVert^{2}_{L^{2}(\Omega)}+2\pi(\dot{\lambda})^{3}\lambda,
  \end{equation}
  \begin{equation}\label{B.5}
      \frac{1}{2}\frac{d}{dt}\lVert\Delta\omega\rVert^{2}_{L^{2}(\Omega)}+\frac{1}{\lambda^{2}}\lVert\nabla\Delta\omega\rVert^{2}_{L^{2}(\Omega)}=\frac{4\pi}{3}\frac{d}{dt}(\dot{\lambda}\lambda)^{3}+\frac{1}{2}\frac{\dot{\lambda}}{\lambda}\lVert\Delta\omega\rVert^{2}_{L^{2}(\Omega)}+2\pi\dot{\lambda}^{5}\lambda^{3}+4\pi\dot{\lambda}^{4}\lambda^{2}.
  \end{equation}
\end{lemma}
By the energy identities above, we can obtain the following prior estimates. Before stating the main results, let us define the energy quantities:
\begin{equation}\label{B.6}
    E(t)=\sup_{0\leq s\leq t} \bigg\{\frac{1}{2}\lVert\omega(s,\cdot)\rVert^{2}_{L^{2}(\Omega)}+\frac{1}{2}\lVert\nabla\omega(s,\cdot)\rVert^{2}_{L^{2}(\Omega)}+\frac{1}{2}\lVert\Delta\omega(s,\cdot)\rVert^{2}_{L^{2}(\Omega)}\bigg\},
    \end{equation}
\begin{equation*}
    D(t)=\frac{1}{\lambda^{2}(t)}\lVert\nabla\omega(t,\cdot)\rVert^{2}_{L^{2}(\Omega)}+\frac{1}{\lambda^{2}(t)}\lVert\Delta\omega(t,\cdot)\rVert^{2}_{L^{2}(\Omega)}+\frac{1}{\lambda^{2}(t)}\lVert\nabla\Delta\omega(t,\cdot)\rVert^{2}_{L^{2}(\Omega)}.
\end{equation*}
\begin{lemma}[Prior estimates]
    Assume that $(\omega,\lambda)$ is a smooth solution of (\ref{B.2}) on some interval $[0,T^{*}]$. Assume that $\lambda_{0}>0$, $\omega_{0}|_{y=1}=0$, and $\omega(t,\cdot)\in H^{2}(\Omega)$ for $t\in[0,T^{*}]$. Then there exists a $T=T(E(0),\lambda_{0})>0$ with $T\leq T^{*}$ such that for any $t\in[0,T]$ the following prior estimates hold:
    \begin{align}
        &E(t)\leq 4 C(E(0),\lambda_{0}),\\
        &\lambda(t)>\frac{\lambda_{0}}{2}.
    \end{align}
\end{lemma}
\begin{proof}
    We briefly recall the proof which has already been done in \cite{MP} in 2D. From the energy identities (\ref{B.3}), (\ref{B.4}) and (\ref{B.5}), we have
    \begin{equation}\label{B.10}
        \begin{aligned}
            E(t)+\int_{0}^{t}D(s)ds\leq& E(0)+\frac{4\pi}{3}\sup_{0\leq s\leq  t}(\lambda\dot{\lambda})^{3}-\frac{4\pi}{3}(\lambda_{0}\dot{\lambda}(0))^{3}+3\left(\int_{0}^{t}\frac{|\dot{\lambda}(s)|}{\lambda(s)}ds\right)E(t)\\
&+4\pi\left(\int_{0}^{t}\dot{\lambda}^{4}\lambda^{2}+|\dot{\lambda}|^{5}\lambda^{3}+|\dot{\lambda}|^{3}\lambda ds\right).
        \end{aligned}
    \end{equation}
    By the radial sobolev inequality,
    \begin{equation}\label{B.11}
        \omega^{2}_{y}(1)\lesssim\left(\lVert\Delta\omega\rVert_{L^{2}(\Omega)}+\lVert\nabla\omega\rVert_{L^{2}(\Omega)}\right)\lVert\nabla\omega\rVert_{L^{2}(\Omega)},
    \end{equation}
    and the boundary condition of the equation (\ref{B.2}), we obtain
    \begin{equation}\label{B.12}
        |\dot{\lambda}(t)|\lesssim\frac{\sqrt{E(t)}}{\lambda(t)}.
    \end{equation}
    Hence,
    \begin{equation}\label{B.13}
    \begin{aligned}
        &3\left(\int_{0}^{t}\frac{|\dot{\lambda}(s)|}{\lambda(s)}ds\right)E(t)+4\pi\int_{0}^{t}\left(\dot{\lambda}^{4}\lambda^{2}+|\dot{\lambda}|^{5}\lambda^{3}+|\dot{\lambda}|^{3}\lambda\right)ds\\
        \lesssim&t\left(E^{\frac{3}{2}}(t)+E^{2}(t)+E^{\frac{5}{2}}(t)\right)\sup_{0\leq s\leq t}\frac{1}{\lambda^{2}(s)}.
        \end{aligned}
    \end{equation}
    For the boundary term, from (\ref{B.11}) and the H\"{o}lder inequality,
    \begin{equation}\label{B.14}
        \begin{aligned}
        |\lambda\dot{\lambda}|^{3}&\lesssim\left(\lVert\Delta\omega\rVert^{\frac{3}{2}}_{L^{2}(\Omega)}+\lVert\nabla\omega\rVert^{\frac{3}{2}}_{L^{2}(\Omega)}\right)\lVert\nabla\omega\rVert^{\frac{3}{2}}_{L^{2}(\Omega)}\\
        &\leq \delta E+C_{\delta}\lVert\nabla\omega\rVert^{6}_{L^{2}(\Omega)}.
        \end{aligned}
    \end{equation}
   It remains to estimate the term $\lVert\nabla\omega\rVert_{L^{2}(\Omega)}$. Indeed, integrating the identity (\ref{B.4}) yields
   \begin{equation*}
       \begin{aligned}
           \frac{1}{2}\sup_{0\leq s \leq t}\lVert\nabla\omega\rVert^{2}_{L^{2}(\Omega)}&\leq\frac{1}{2}\lVert\nabla\omega_{0}\rVert^{2}_{L^{2}(\Omega)}+\frac{1}{2}\int_{0}^{t}\bigg|\frac{\dot{\lambda}}{\lambda}\bigg|\lVert\nabla\omega\rVert^{2}_{L^{2}(\Omega)}+2\pi\int_{0}^{t}|\dot{\lambda}|^{3}\lambda ds\\
           &\lesssim E(0)+CtE^{\frac{3}{2}}(t)\sup_{0\leq s\leq t}\frac{1}{\lambda^{2}(s)}.
       \end{aligned}
   \end{equation*}
   Then injecting into (\ref{B.14}) gives
   \begin{equation*}
       |\lambda\dot{\lambda}|^{3}\leq \delta E+C_{\delta}\left(CE(0)+CtE^{\frac{3}{2}}(t)\sup_{0\leq s\leq t}\frac{1}{\lambda^{2}(s)}\right)^{3}
   \end{equation*}
   Plugging this bound back to (\ref{B.10}) and together with (\ref{B.13}), we conclude that
   \begin{equation*}
       E(t)+\int_{0}^{t}D(s)ds\leq C(E(0),\lambda_{0})+r(t)p\left(E^{\frac{1}{2}}(t)\right)q\left(\sup_{0\leq s\leq t}\frac{1}{\lambda^{2}(s)}\right),
   \end{equation*}
   where $r(t),p(t),q(t)$ are all increasing polynomial functions on $[0,+\infty)$ with $r(0)=0$. Since $\lambda(t)=\lambda_{0}+\int_{0}^{t}\dot{\lambda}(s)ds$, it follows from (\ref{B.12}) that
   \begin{equation*}
       \lambda(t)\geq \lambda_{0}-t\sup_{0\leq s\leq t}|\dot{\lambda}(s)|\geq\lambda_{0}-Ct\left(\sup_{0\leq s\leq t}\frac{1}{\lambda(s)}\right)E^{\frac{1}{2}}(t).
   \end{equation*}
   Set
\begin{equation*}
    T^{'}=\sup \bigg\{t\geq 0\ \bigg|\ E(t) \leq 4C(E(0),\lambda_{0}),\lambda(t)>\frac{\lambda_{0}}{2}\bigg\}.
\end{equation*}
By the continuity of $E(t)$ and $\lambda(t)$, it follows that $T^{'}>0$. Then by standard bootstrap argument, there exists a $T=T(E(0),\lambda_{0})$ which satisfies
\begin{equation*}
    T\leq \frac{3\lambda^{2}_{0}}{16C\sqrt{C(E(0),\lambda_{0})}}\quad {\rm{and}}\quad r(T)p\left(2\sqrt{C(E(0),\lambda_{0}}\right)q\left(\frac{4}{\lambda^{2}_{0}}\right)\leq C(E(0),\lambda_{0}).
\end{equation*}
such that for any $t\in[0,T]$,
\begin{align*}
    E(t)\leq 2C&(E(0),\lambda_{0})\leq 4C(E(0),\lambda_{0}),\\
    &\lambda(t)\geq\frac{3}{4}\lambda_{0}>\frac{1}{2}\lambda_{0},
\end{align*}
and this concludes the proof of lemma.
\end{proof}
\begin{theorem}[Local well-posedness]\label{thB.4}
    Let $\omega_{0}\in H^{2}(\Omega)$, $\lambda_{0}>0$, and $\omega_{0}|_{y=1}=0$. Then there exists a time $T=T(\lVert\omega_{0}\rVert_{H^{2}(\Omega)},\lambda_{0})>0$ and a solution $(\omega,\lambda)$ to the equation (\ref{B.2}) on the time interval $[0,T]$ such that
    \begin{align}
&\omega\in C([0,T],H^{2}(\Omega))\cap L^{2}([0,T],H^{3}(\Omega)),\notag\\
&\omega_{t}\in C([0,T],L^{2}(\Omega))\cap L^{2}([0,T],H^{1}(\Omega)),\label{B.15}\\
&\lambda\in C^{1}([0,T],\RR),\notag
\end{align}
and for any $t\in[0,T]$, the following bounds hold:
\begin{equation*}
    E(t)\leq C(E(0),\lambda_{0}),\quad \lambda(t)>\frac{\lambda_{0}}{2}.
\end{equation*}
where the energy $E(t)$ is defined by (\ref{B.6}). Moreover, if $T_{\rm{max}}$ is the maximal time of existence of a solution $(\omega,\lambda)$ satisfying (\ref{B.15}), then
\begin{equation*}
    {\rm{either}}\quad \lim_{t\rightarrow {T_{\rm{max}}}}\lVert \omega(t,\cdot)\rVert_{H^{2}(\Omega)}=+\infty\quad {\rm{or}}\quad \lim_{t\rightarrow {T_{\rm{max}}}}\lambda(t)=0.
\end{equation*}
\end{theorem}
The proof of the existence of such solution follows from a standard iteration argument and the blow-up criterion is a consequence of (\ref{B.12}). Details have already been given in \cite{MP} in 2D and the 3D case is just the same, so we omit the details here.\\
Finally, we establish the Cauchy theory for the radial Stefan problem (\ref{eq:2}). It follows from the change of variables $u(t,r)=\omega(t,r/ \lambda(t))$ and Theorem \ref{thB.4}.
\begin{theorem}[Well-posedness in $H^{2}$]
Let $u_{0}\in H^{2}(\Omega(\lambda_{0}))$, $\lambda_{0}>0$, $u_{0}(\lambda_{0})=0$. Then there exists a time $T=T(\lVert u_{0}\rVert_{ H^{2}(\Omega(\lambda_{0}))},\lambda_{0})>0$ and a solution $(u,\lambda)$ to the Stefan problem (\ref{eq:2}) on the time interval $[0,T]$ such that
\begin{align}
&u\in C([0,T],H^{2}(\Omega(t)))\cap L^{2}([0,T],H^{3}(\Omega(t))),\notag\\
&u_{t}\in C([0,T],L^{2}(\Omega(t)))\cap L^{2}([0,T],H^{1}(\Omega(t))),\label{B.1}\\
&\lambda\in C^{1}([0,T],\RR),\notag
\end{align}
and the following bounds hold:
\begin{equation*}
    \lVert u\rVert_{H^{2}(\Omega(t))}\leq C=C(\lVert u_{0}\rVert_{ H^{2}(\Omega(\lambda_{0}))},\lambda_{0}),\quad \lambda(t)>\frac{\lambda_{0}}{2}.
\end{equation*}
Moreover, if $T_{\rm{max}}$ is the maximal time of existence of a solution $(u,\lambda)$ satisfying (\ref{B.1}), then
\begin{equation*}
    {\rm{either}}\quad \lim_{t\rightarrow {T_{\rm{max}}}}\lVert u(t,\cdot)\rVert_{H^{2}(\Omega(t))}=+\infty\quad {\rm{or}}\quad \lim_{t\rightarrow {T_{\rm{max}}}}\lambda(t)=0.
\end{equation*}
\end{theorem}


\begin{thebibliography}{80}
\bibitem{AHV}Andreucci, D.; Herrero, M. A.; Velázquez, J. J. L. The classical one-phase Stefan problem: A catalog of interface behaviors. Surveys in Industrial Mathematics 9 (2001), 247--337.
\bibitem{ACS1}Athanasopoulos, I.; Caffarelli, L.; Salsa, S. Regularity of the free boundary in parabolic phase-transition problems. Acta Math. 176 (1996), no. 2, 245--282. 
\bibitem{ACS2}Athanasopoulos, I.; Caffarelli, L. A.; Salsa, S. Phase transition problems of parabolic type: flat free boundaries are smooth. Comm. Pure Appl. Math. 51 (1998), no. 1, 77--112.
\bibitem{BW}Bourgain, J.; Wang, W. Construction of blowup solutions for the nonlinear Schrödinger equation with critical nonlinearity. Ann. Scuola Norm. Sup. Pisa Cl. Sci. (4) 25 (1997), no. 1-2, 197–215.
\bibitem{CL1}Caffarelli, L. A. The regularity of free boundaries in higher dimensions. Acta Math. 139 (1977), no. 3--4, 155--184.
\bibitem{CL2}Caffarelli, L. A. Some aspects of the one-phase Stefan problem. Indiana Univ. Math. J. 27 (1978), no. 1, 73--77.
\bibitem{CL3}Caffarelli, L. A.; Evans, L. C. Continuity of the temperature in the two-phase Stefan problem. Arch. Rational Mech. Anal. 81 (1983), no. 3, 199--220.
\bibitem{CL4}Caffarelli, L. A.; Friedman, A. Continuity of the temperature in the Stefan problem. Indiana Univ. Math. J. 28 (1979), no. 1, 53--70.
\bibitem{CS}Caffarelli, L. A.; Salsa, S. A geometric approach to free boundary problems. Graduate Studies in Mathematics, 68. American Mathematical Society, Providence, R.I., (2005).
\bibitem{C1}Collot, C. Nonradial type II blow up for the energy-supercritical semilinear heat equation, Anal. PDE 10 (2017), no. 1, 127--252.
\bibitem{C}Collot, C. Type II blow up manifolds for the energy supercritical semilinear wave equation, Mem. Amer. Math. Soc. 252 (2018) v+163.
\bibitem{CFP}Collot, C.; Merle, F.; Raphaël, P. Strongly anisotropic type II blow up at an isolated point, J. Amer. Math. Soc. 33 (2020), no. 2, 527--607. 
\bibitem{CTNV}Collot, C.; Ghoul, T.-E.; Masmoudi, N.; Nguyen, V. T. Refined description and stability for singular solutions of the 2D Keller--Segel system, Commun. Pure Appl. Math. 75 (2022) 1419--1516.
\bibitem{CK}Choi, S; Kim, I. C. Regularity of one-phase Stefan problem near Lipschitz initial data. Amer. J. Math. 132 (2010), no. 6, 1693--1727.
\bibitem{F}Friedman, A. The Stefan problem in several space variables. Trans. Amer. Math. Soc. 133 (1968), 51--87.
\bibitem{FK}Friedman, A.; Kinderlehrer, D. A one phase Stefan problem. Indiana Univ. Math. J. 24 (1974/75), no. 11, 1005--1035.
\bibitem{MP}Hadžić, M.; Raphaël, P. On melting and freezing for the 2D radial Stefan problem, J. Eur. Math. Soc. 21  (2019), no. 11, 3259--3341.
\bibitem{HS2}Hadžić, M.; Shkoller, S. Stability and decay in the classical Stefan problem. Comm. Pure Appl. Math. 68 (2015), no. 5, 689–757.
\bibitem{HS1}Hadžić, M.; Shkoller, S. Well-posedness for the classical Stefan problem and the vanishing surface tension limit. Arch. Ration. Mech. Anal. 223 (2017), no. 1, 213–264.
\bibitem{H} Hanzawa, E. I. Classical solutions of the Stefan problem. Tôhoku Math. J. (2) 33 (1981), no. 3, 297--335. 
\bibitem{HHV}Herraiz, L. A.; Herrero, M. A.; Velázquez, J. J. L. A note on the dissolution of spherical crystals. Proc. Roy. Soc. Edinburgh Sect. A 131 (2001), no. 2, 371–389.
\bibitem{HV}Herrero, M. A.; Velázquez, J. J. L. On the melting of ice balls. SIAM J. Math. Anal. 28 (1997), no. 1, 1–32.
\bibitem{HR}Hillairet, M.; Raphaël, P. Smooth type II blow-up solutions to the four-dimensional energy-critical wave equation. Anal PDE.  5 (2012), no. 4, 777–829. 
\bibitem{J}Jendrej, J. Construction of type II blow-up solutions for the energy-critical wave equation in dimension 5. J. Funct. Anal. 272 (2017),  no. 3, 866–917.
\bibitem{K}Kamenomostskaja, S. L. On Stefan's problem.  Mat. Sb. (N.S.) 53(95) (1961), 489–514.
\bibitem{Kim}Kim, I. C. Uniqueness and existence results on the Hele-Shaw and the Stefan problems. Arch. Ration. Mech. Anal. 168 (2003), no. 4, 299--328.
\bibitem{KN}Kinderlehrer, D.; Nirenberg, L. Regularity in free boundary problems. Ann. Scuola Norm. Sup. Pisa Cl. Sci. (4) 4 (1977), no. 2, 373--391.
\bibitem{KN1}Kinderlehrer, D.; Nirenberg, L. The smoothness of the free boundary in the one phase Stefan problem. Comm. Pure Appl. Math. 31 (1978), no. 3, 257--282.
\bibitem{KH}Koch, H. Classical solutions to phase transition problems are smooth. Comm. Partial Differential Equations. 23 (1998), no. 3--4, 389--437.
\bibitem{JW}Krieger, J.; Schlag, W. Non-generic blow-up solutions for the critical focusing NLS in 1-D. J. Eur. Math. Soc. 11 (2009), no. 1, 1–125.
\bibitem{KST}Krieger, J.; Schlag, W.; Tataru, D. Renormalization and blow up for charge one equivariant critical wave maps. Invent. Math. 171 (2008), no. 3, 543–615.
\bibitem{KST2}Krieger, J.; Schlag, W.; Tataru, D. Slow blow-up solutions for the $H^{1}(\RR^{3})$ critical focusing semilinear wave equation. Duke Math. J. 147 (2009), no. 1, 1–53.
\bibitem{LSU}Ladyženskaja, O. A.; Solonnikov, V. A.; Uralceva, N. N. Linear and quasilinear equations of parabolic type. Translations of Mathematical Monographs, 23. American Mathematical Society, Providence, R. I., 1968.
\bibitem{L}Lebedev, N. N. Special Functions and Their Applications, Dover, New York, 1972.
\bibitem{M} Meirmanov, A. M. The Stefan problem. de Gruyter Expositions in Mathematics, 3. Walter de Gruyter, Berlin, 1992.
\bibitem{MR3}Merle, F.; Raphaël, P. Sharp upper bound on the blow-up rate for the critical nonlinear Schrödinger equation. Geom. Funct. Anal. 13 (2003), no. 3, 591–642.
\bibitem{MR1}Merle, F.; Raphaël, P. The blow-up dynamic and upper bound on the blow-up rate for critical nonlinear Schrödinger equation. Ann. of Math. Ann. of Math. (2) 161 (2005), no. 1, 157–222.
\bibitem{MR4}Merle, F.; Raphaël, P. Profiles and quantization of the blow up mass for critical nonlinear Schrödinger equation, Comm. Math. Phys. 253 (2005), no. 3, 675--704.
\bibitem{MR2}Merle, F.; Raphaël, P. On a sharp lower bound on the blow-up rate for the $ L^{2} $ critical nonlinear Schrödinger equation. J. Amer. Math. Soc. 19 (2006), no. 1, 37–90.
\bibitem{MR5}Merle, F.; Raphaël, P. On universality of blow-up profile for $L^{2}$ critical nonlinear Schrödinger equation, Invent. Math. 156 (2004), no. 3, 565--672. 
\bibitem{MRR1}Merle, F.; Raphaël, P.; Rodnianski, I. Blow up dynamics for smooth data equivariant solutions to the critical Schrödinger map problem. Invent. Math. 193 (2013), no. 2, 249–365.
\bibitem{MRP} Merle, F.; Raphaël, P.; Rodnianski, I. Type II blow up for the energy supercritical NLS. Camb. J. Math. 3 (2015), no. 4, 439–617.
\bibitem{FPIJ}Merle, F.; Raphael, P.; Rodnianski, I.; Szeftel, J. On blow up for the energy super critical defocusing nonlinear Schrödinger equations. Invent. Math. 227 (2022), no. 1, 247–413.
\bibitem{MF}Morse, P. M.; Feshbach, H. Methods of Theoretical Physics, Vol. I, McGraw--Hill, New York, 1953.
\bibitem{P}Perelman; G. On the blow-up phenomenon for the critical nonlinear Schrödinger equation in 1D. In: Séminaire: Équations aux Dérivées Partielles, 1999--2000, École Polytech., Palaiseau, exp. III, 16 pp. (2000). 
\bibitem{PSS} Prüss, J.; Saal, J.; Simonett, G. Existence of analytic solutions for the classical Stefan problem. Math. Ann. 338 (2007), no. 3, 703--755. 
\bibitem{R}Raphaël, P. Stability of the log-log bound for blow up solutions to the critical nonlinear Schrödinger equation, Math. Ann. 331 (2005), no. 3, 577--609. 
\bibitem{PR}Raphaël, P.; Rodnianski, I. Stable blow up dynamics for the critical co-rotational wave maps and equivariant Yang-Mills problems. Publ. Math. Inst. Hautes Études Sci. 115 (2012), 1–122.
\bibitem{RS}Raphaël, P.; Schweyer, R. Quantized slow blow up dynamics for the corotational energy critical harmonic heat flow. Anal. PDE. 7 (2014), no. 8, 1713–1805.
\bibitem{RS1}Raphaël, P.; Schweyer, R. Stable blow up dynamics for the 1-corotational energy critical harmonic heat flow. Comm. Pure Appl. Math. 66 (2013), no. 3, 414–480.
\bibitem{RS2}Raphaël, P.; Schweyer, R. On the stability of critical chemotactic aggregation. Math. Ann. 359 (2014), no. 1-2, 267–377.
\bibitem{IJ}Rodnianski, I.; Sterbenz, J. On the formation of singularities in the critical O(3) $\sigma$-model. Ann. of Math. (2) 172 (2010), no. 1, 187–242.
\bibitem{Ru}Rubenstein, L. I. The Stefan Problem, Transl. Math. Monographs 27, AMS, Providence, RI, 1971.
\bibitem{RS3}Schweyer, R. Type II blow-up for the four dimensional energy critical semi linear heat equation. J. Funct. Anal. 263 (2012), no. 12, 3922–3983.
\bibitem{SF} Solonnikov, V. A.; Frolova, E. V. Lp-theory for the Stefan problem. J. Math. Sci. (New York) 99 (2000), no. 1, 989--1006.
\bibitem{SG}Szegö, G. Orthogonal Polynomials. Amer. Math. Soc., Providence, RI (1939).
\bibitem{V}Velázquez, J. J. L. Singular solutions of partial differential equations modelling chemotactic aggregation. In: Proc. ICM 2006 (Madrid), Vol. III, Eur. Math. Soc. (2006), 321--338. 
\end{thebibliography}
\end{document}